\documentclass{article}

\usepackage[utf8]{inputenc}
\usepackage[cm]{fullpage}
\usepackage{graphicx, amsmath, amsfonts, amsthm, amssymb, latexsym, verbatim,
pifont, xspace, mathrsfs, tabularx}

\usepackage[shortlabels]{enumitem}
\usepackage{tocloft}

\usepackage{thmtools}
\usepackage{thm-restate}
\usepackage[colorlinks, hypertexnames=false]{hyperref}
\hypersetup{linkcolor=blue, urlcolor=blue, citecolor=red}
\usepackage[capitalise, nameinlink]{cleveref}

\usepackage{algorithmicx, algpseudocode}
\usepackage[ruled]{algorithm}
\algrenewcommand{\algorithmiccomment}[1]{\hfill[{\it #1}]}

\usepackage{tablefootnote}
\usepackage{caption}
\usepackage{subcaption}
\usepackage{mathbbol} % for lower case mathbb

\usepackage{calc}
\usepackage[table]{xcolor}

\usepackage{tikz}
\usepackage{tikz-cd}
\usetikzlibrary{decorations, arrows}
\usepackage{lscape}
\usepackage{multirow}
\usepackage{chngcntr}
\counterwithin{figure}{section}
\counterwithin{table}{section}

\pgfdeclaredecoration{sl}{initial}{
  \state{initial}[width=\pgfdecoratedpathlength-1sp]{
    \pgfmoveto{\pgfpointorigin}
  }
  \state{final}{
    \pgflineto{\pgfpointorigin}
  }
}
\pgfarrowsdeclarecombine{aa}{aa}{angle 90}{angle 90}{angle 90}{angle 90}

\definecolor{color0}{RGB}{127, 0, 255}
\definecolor{color1}{RGB}{128, 128, 128}
\definecolor{color2}{RGB}{255, 0, 127}
\definecolor{color3}{RGB}{255, 0, 255}
\definecolor{color4}{RGB}{0, 0, 255}
\definecolor{color5}{RGB}{0, 255, 0}
\definecolor{color6}{RGB}{255, 128, 0}
\definecolor{color7}{RGB}{255, 0, 0}
\definecolor{color8}{RGB}{0, 128, 255}
\definecolor{color9}{RGB}{128, 255, 0}

\tikzstyle{node}=[rectangle,
  draw,
  fill=white,
  opacity=1,
  text opacity=1,
  inner sep=3pt,
  thick,
minimum size=1.6em]
\tikzstyle{mylabel}=[draw=none, opacity=0]

\tikzstyle{solid black}=[->, draw=black, thick, -angle 90]
\tikzstyle{dashed black}=[->, draw=black, thick, dashed, -angle 90]

\tikzstyle{a}=[-angle 90, draw=color0, thick]
\tikzstyle{b}=[-angle 90, draw=color1, thick]
\tikzstyle{c}=[-angle 90, draw=color2, thick]
\tikzstyle{d}=[-angle 90, draw=color3, thick]
\tikzstyle{e}=[-angle 90, draw=color4, thick]

\tikzset{arrow/.style={-angle 90, shorten >=2pt, shorten <=2pt}}
\tikzset{parallel_arrow_1/.style={-angle 90,
decoration={sl,raise=1mm},decorate, shorten >=2pt, shorten <=2pt}}
\tikzset{parallel_arrow_2/.style={-angle 90,
decoration={sl,raise=-1mm},decorate, shorten >=2pt, shorten <=2pt}}
\tikzset{aarrow/.style={-aa, shorten >=2pt, shorten <=2pt, densely dashed}}
\tikzset{parallel_aarrow_1/.style={-aa,
    decoration={sl,raise=1mm},decorate, shorten >=2pt, shorten <=2pt,
densely dashed}}
\tikzset{parallel_aarrow_2/.style={-aa,
    decoration={sl,raise=-1mm},decorate, shorten >=2pt, shorten <=2pt,
densely dashed}}
\newcommand{\id}{\mbox{\rm id}}
\newcommand{\set}[2]{\ensuremath{\{\: #1 \: :\: #2 \:\}}}
\newcommand{\genset}[1]{\ensuremath{\langle\: #1 \:\rangle}}
\renewcommand{\to}{\longrightarrow}
\newcommand{\Z}{\mathbb{Z}}

\newcommand{\N}{\mathbb{N}}
\renewcommand{\P}{\mathcal{P}}

\newcommand{\Stab}{\operatorname{Stab}}
\newcommand{\im}{\operatorname{im}}
\newcommand{\QuoStab}[3]{\Stab_{#2}(#3)/\ker(#1)}
\newcommand{\Sym}{\operatorname{Sym}}
\newcommand{\defn}[1]{\textbf{\textit{#1}}}
\renewcommand{\L}{\mathscr{L}}
\newcommand{\R}{\mathscr{R}}

\newcommand{\T}{\mathbb{T}}
\newcommand{\X}{\mathbb{X}}
\newcommand{\Y}{\mathbb{Y}}
\newcommand{\V}{\mathbb{V}}
\newcommand{\E}{\mathbb{E}}
\newcommand{\vv}{\mathbb{v}}

\newcommand{\w}{\mathbb{w}}
\newcommand{\Reach}{\operatorname{Desc}}
\newcommand{\Kids}{\operatorname{Kids}}

\newtheorem{thm}{Theorem}[section]
\newtheorem{prop}[thm]{Proposition}
\newtheorem{cor}[thm]{Corollary}
\newtheorem{lemma}[thm]{Lemma}
\newtheorem{defi}[thm]{Definition}
\newenvironment{de}{
\begin{defi} \rm}{
\end{defi}}
\newtheorem{exam}[thm]{Example}
\newenvironment{ex}{
\begin{exam} \rm}{
\end{exam}}
\newtheorem{remark}[thm]{Remark}

\numberwithin{equation}{section}
\newtheorem*{lemma-no-num}{Lemma}

\usepackage[backend=biber, bibencoding=utf8, giveninits=true,
sortcites]{biblatex}
\newcommand{\libsemigroups}{\textsc{libsemigroups}~\cite{Mitchell2022ab}\xspace}
\newcommand{\libsemigroupspybind}{\textsc{libsemigroups\_pybind11}~\cite{Mitchell2022ad}\xspace}
\newcommand{\python}{\textsc{Python}\xspace}
\newcommand{\GAP}{\textsc{GAP}~\cite{GAP2022aa}\xspace}
\newcommand{\threemanifolds}{\textsc{3\-manifolds}~\cite{Culler2025aa}\xspace}
\newcommand{\Semigroups}{\textsc{Semigroups}~\cite{Mitchell2022ac}\xspace}
\newcommand{\CREAM}{\textsc{CREAM}~\cite{Pereira2022aa}\xspace}

\setlist{itemsep=0em}
\makeatletter
\newcommand\footnoteref[1]{\protected@xdef\@thefnmark{\ref{#1}}\@footnotemark}
\makeatother

\title{Computing finite index congruences of finitely presented
semigroups and monoids}
\author{Marina Anagnostopoulou-Merkouri, Reinis Cirpons, James D. Mitchell, and
Maria Tsalakou}

\begin{document}
\maketitle

\begin{abstract}
  In this paper, we describe an algorithm for computing the left,
  right, or 2-sided
  congruences of a finitely presented semigroup or monoid with finitely many
  classes, and an alternative algorithm when the finitely presented semigroup or
  monoid is finite. We compare the two algorithms presented with
  existing algorithms
  and implementations. The first algorithm is a generalization of
  Sims' low-index
  subgroup algorithm for finding the congruences of a monoid. The
  second algorithm
  involves determining the distinct principal congruences, and then
  finding all of
  their possible joins. Variations of this algorithm have been suggested in
  numerous contexts by numerous authors. We show how to utilize the theory of
  relative Green's relations, and a version of Schreier's Lemma for monoids, to
  reduce the number of principal congruences that must be generated as the first
  step of this approach. Both of the algorithms described in this paper are
  implemented in the \GAP package \Semigroups, and the first algorithm is
  available in the C++ library \libsemigroups and in its \python bindings
  \libsemigroupspybind.
\end{abstract}

\tableofcontents
%\listoffigures
%\listoftables

\section{Introduction}\label{section-intro}

In this paper, we are concerned with the problem of computing finite index
congruences of a finitely presented semigroup or monoid. One case of
particular interest is
computing the entire lattice of congruences of a finite semigroup or monoid.
We will present two algorithms that can perform these computations and compare
them with each other and to existing algorithms and their implementations.
The first algorithm is the only one of its kind, permitting the
computation of finite index 1-sided and 2-sided congruences, and a
host of other things (see \cref{section-applications-of-low-index-algo})
of infinite finitely presented semigroups and monoids. Although this
first algorithm is not specifically designed to find finite index
subgroups of finitely presented groups, it can be used for such
computations and is sometimes faster than the existing
implementations in \textsc{3Manifolds}~\cite{Culler2025aa} and \GAP;
see~\cref{table-infinite}.
The second algorithm we present is, for many examples, several orders
of magnitude faster than any existing method,
and in many cases permits computations that were previously unfeasible.
Examples
where an implementation of an existing
algorithm, such as that in \cite{Pereira2022aa}, is faster are
limited to those with total runtime below 1 second; see
\cref{table-numbers}. Some further
highlights include: computing the numbers of right/left congruences of many
classical examples of finite transformation and diagram monoids, see
\cref{appendix-numbers}; reproducing and extending the computations
from~\cite{Bailey2016aa} to find the number of congruences in
free semigroups and monoids (see \cref{table-2-sided}); computational
experiments with the algorithms implemented were crucial in
determining the minimum transformation representation of the
so-called diagram monoids in~\cite{East2024aa}; and in classifying the maximal
and minimal 1-sided congruences of the full transformation monoids
in~\cite{Cirpons2025aa}. In \cref{appendix-benchmarks}, we present significant
quantitative data exhibiting  the performance of our algorithms.
\cref{appendix-numbers} provides a wealth of data generated using the
implementation of the algorithms described here.
For example, the sequences of numbers of minimal 1-sided congruences
of a number of well-studied transformation monoids are apparent in
several of the tables in \cref{appendix-numbers}, such as
\cref{table-motzkin-dual-sym-inv}.

The question of determining the lattice of 2-sided congruences of a semigroup or
monoid is classical and has been widely studied in the literature; see, for
example,~\cite{Maltsev1952aa}.
% TODO(later) cite more papers about 2-sided
% congs of semigroups
Somewhat more recently, this interest was rekindled by
Ara\'{u}jo, Bentz, and Gomes in~\cite{Araujo2018aa}, Young (n\'{e} Torpey)
in~\cite{Torpey2019aa}, and the third author of the present article, which
resulted in~\cite{East2018aa} and its numerous offshoots~\cite{Brookes2020aa,
Dandan2019aa, East2022ab, East2020aa, East2018ab, East2022aa, East2019aa}. The
theory of 2-sided congruences of a monoid is analogous to the theory of normal
subgroups of a group, and 2-sided congruences play the same role for
monoids with respect to quotients and homomorphisms. As such, it is
perhaps not surprising
that the theory of 2-sided congruences of semigroups and monoids is rather rich.
The 2-sided congruences of certain types of semigroup are completely
classified, for a small sample among many, via linked triples for regular Rees
0-matrix semigroups~\cite[Theorem 3.5.8]{Howie1995aa}, or via the kernel and
trace for inverse semigroups~\cite[Section 5.3]{Howie1995aa}.

The literature relating to 1-sided congruences is less well-developed; see,
for example,~\cite{Brookes2020aa,Meakin1975aa}. Subgroups are to groups what
1-sided congruences are to semigroups. This accounts, at least in part, for
the relative scarcity of results in the literature on 1-sided congruences.
The number of such congruences can be enormous, and the structure of the
corresponding lattices can be wild. For example, the full transformation monoid
of degree $4$ has size $256$ and possesses $22,069,828$ right
congruences\footnote{\label{footnote-numbers}This number was computed
  for the first time using the algorithm described in
  \cref{section-low-index} as implemented in the C++ library
  \libsemigroups whose authors include the authors of the present
paper. It was not previously known.}.
Another example is that of the stylic monoids from~\cite{Abram2022aa}, which
are finite quotients of the well-known plactic
monoids~\cite{Knuth1970aa,Lascoux1981aa}. The stylic monoid with $5$ generators
has size $51$, while the number of left congruences is
$1,431,795,099$\footnoteref{footnote-numbers}.

The purpose of this paper is to provide general computational tools for
computing the 1- and 2-sided congruences of a finite, or finitely presented,
monoid.
There are a number of examples in the literature of such general
algorithms; notable examples
include~\cite{Freese2008aa},~\cite{Torpey2019aa}, and~\cite{Araujo2022aa},
which describes an implementation of the algorithms from~\cite{Freese2008aa}.

The first of the two algorithms we present is a generalization of
Sims' low-index
subgroup algorithm for congruences of a finitely presented monoid; see Section
5.6 in~\cite{Sims1994aa} for details of Sims' algorithm;  some
related algorithms and applications of the low-index subgroups
algorithm can be found
in~\cite{HoltRees1992aa},~\cite{Hulpke2001aa},~\cite[Section
6]{Neubuser1982aa}, and~\cite{Rauzy2021}.
A somewhat similar algorithm for computing
low-index ideals of a finitely presented monoid with decidable word
problem was given
by Jura in~\cite{JuraIdeals1} and~\cite{JuraIdeals2} (see
also~\cite{Ruskuc1998aa}).
We will refer to the algorithm presented here as the \defn{low-index congruence
algorithm}. We present a unified framework for computing various
special types of congruences, including:  2-sided congruences;
congruences including or excluding given pairs of elements (leading
to the ability to compute specific parts of a congruence lattice);
solving the word problem in residually finite semigroups or monoids;
congruences such that the corresponding quotient is a group;
congruences arising from 1- or 2-sided ideals; and 1-sided
congruences representing a faithful action of the original monoid;
see~\cref{section-applications-of-low-index-algo} for details. This
allows us to, for example, implement a method for computing the
finite index ideals of a finitely presented semigroup akin
to~\cite{JuraIdeals1, JuraIdeals2} by implementing a single function
(\cref{alg-is-right-rees-congruence}) which determines if a finite
index right congruence is a Rees congruence.

The low-index congruence algorithm takes as input a finite monoid
presentation defining
a monoid $M$, and a positive integer $n$. It permits the congruences with up to
$n$ classes to be iterated through without repetition, while only holding a
representation of a single such congruence in memory.
Each congruence is represented as a certain type of directed graph, which we
refer to as \defn{word graphs}; see~\cref{section-prelims} for the definition.
The space complexity of this approach is $O(mn)$ where $m$ is the number of
generators of the input monoid, and $n$ is the input positive integer. Roughly
speaking, the low-index algorithm performs a backtracking search in a tree whose
nodes are word graphs. If the monoid $M$ is finite, then setting $n = |M|$
allows us to determine all of the left, right, or 2-sided congruences of $M$.
Finding all of the subgroups of a finite group was perhaps not the original
motivation behind Sims' low-index subgroup algorithm from~\cite[Section
5.6]{Sims1994aa}. In particular, there are likely better ways of finding all
subgroups of a finite group; see, for example,~\cite{Holt2005aa},
~\cite{Hulpke2018aa} and the references therein. On the other hand, in some
sense, the structure of semigroups and monoids in general is less constrained
than that of groups, and in some cases the low-index congruence algorithm is the
best or only available means of computing congruences.

To compute the actual lattice of congruences obtained from the low-index
congruences algorithm, we require a mechanism for computing the join or meet of
two congruences given by word graphs. We show that two well-known algorithms
for finite state automata can be used to do this. More specifically, a
superficial modification of the Hopcroft-Karp Algorithm~\cite{Hopcroft1971aa},
for checking if two finite state automata recognise the same language, can be
used to compute the join of two congruences. Similarly, a minor modification of
the standard construction of an automaton recognising the intersection of two
regular languages can be utilised to compute meets; see for
example~\cite[Theorem 1.25]{Sipser2013aa}.
For more background on automata theory, see~\cite{Pin2022aa}.

As described in Section 5.6 of~\cite{Sims1994aa}, one motivation of Sims'
low-index subgroup algorithm was to provide an algorithm for proving
non-triviality of the group $G$ defined by a finite group presentation by
showing that $G$ has a subgroup of index greater than $1$. The low-index
congruences algorithm presented here can similarly be used to prove the
non-triviality of the monoid $M$ defined by a finite monoid presentation by
showing that $M$ has a congruence with more than one class.
There are a number of other possible applications: to determine small degree
transformation representations of monoids (every monoid has a faithful action on
the classes of a right congruence); to prove that certain relations in a
presentation are irredundant; or more generally to show that the monoids
defined by two presentations are not isomorphic (if the monoids defined by
  $\langle A\mid R\rangle$ and $\langle B\mid S\rangle$ have
  different numbers of
congruences with $n$ classes, then they are not isomorphic). Further
applications are discussed in \cref{section-applications-of-low-index-algo}.

The low-index congruence algorithm is implemented in the open-source C++
library \libsemigroups, and available for use in the \GAP package \Semigroups,
and the \python package \libsemigroupspybind. The low-index algorithm is almost
embarrassingly parallel, and the implementation in \libsemigroups is
parallelised using a version of the work stealing queue described
in~\cite[Section 9.1.5]{Williams2019aa}; see \cref{section-low-index} and
\cref{appendix-benchmarks} for more details.

The second algorithm we present is more straightforward than the low-index
congruence algorithm, and is a variation on a theme that has been suggested in
numerous contexts, for example, in~\cite{Araujo2022aa,Freese2008aa,
Torpey2019aa}. There are two main steps to this procedure. First, the distinct
principal congruences are determined, and, second, all possible joins of the
principal congruences are found. Unlike the low-index congruence algorithm,
this algorithm cannot be used to compute anything about infinite finitely
presented semigroups or monoids. This second algorithm is implemented in the
\GAP package \Semigroups.

% FINAL rereading to here

The first step of the second algorithm, as described in, for
example,~\cite{Araujo2022aa, Freese2008aa, Torpey2019aa}, involves computing
the principal congruence, of the input monoid $M$, generated by every pair $(x,
y) \in M \times M$. Of course, in practice, if $(x, y) \in M\times M$, then the
principal congruences generated by $(x, y)$ and $(y, x)$ are the same, and so
only $|M|(|M| - 1) / 2$ principal congruences are actually generated. In either
case, this requires the computation of $O(|M| ^ 2)$ such principal congruences.
We will show that certain of the results from~\cite{East2019aa} can be
generalized to provide a sometimes smaller set of pairs $(x, y) \in M
\times M$ required to generate all of the principal congruences of $M$. In
particular, we show how to compute relative Green's $\mathscr{R}$-class and
$\mathscr{J}$-class representatives of elements in the direct product $M\times
M$ modulo its submonoid $\Delta_M = \set{(m, m)}{m \in M}$. Relative Green's
relations were introduced in~\cite{Wallace1963aa}; see also~\cite{Cain2012aa,
Gray2008aa}. It is straightforward to verify that if $(x, y), (z, t)\in M
\times M$ are $\mathscr{R}$-related modulo $\Delta_M$, then the principal right
congruences generated by $(x, y)$ and $(z, t)$ coincide; see
\cref{prop-relative-greens-cong}\ref{prop-relative-greens-cong-i} for a proof.
We will show that it is possible to reduce the problem of computing relative
$\mathscr{R}$-class representatives to the problems of computing the right
action of $\Delta_M$ on a set, and membership testing in an associated
(permutation) group. The relative $\mathscr{J}$-class representatives
correspond to strongly connected components of the action of $\Delta_M$ on the
relative $\mathscr{R}$-classes by left multiplication, and can be found
whenever the relative $\mathscr{R}$-classes can be. In many examples, the time
taken to find such relative $\mathscr{R}$-class and $\mathscr{J}$-class
representatives is negligible when compared to the overall time required to
compute the lattice of congruences, and in some examples there is a dramatic
reduction in the number of principal congruences that must be generated. For
example, if $M$ is the general linear monoid of $3\times 3$ matrices over the
finite field $\mathbb{F}_2$ of order $2$, then $|M| = 512$ and so $|M|(|M| - 1)
/ 2 = 130,816$. On the other hand, the numbers of relative $\mathscr{J}$- and
$\mathscr{R}$-classes of elements of $M\times M$ modulo $\Delta_M$ are $44$ and
$1,621$, $M$ has $6$ and $1,621$ principal 2-sided and right congruences,
respectively,  and the total number of 2-sided congruences is $7$.  Another
example: if $N$ is the monoid consisting of all $2\times 2$ matrices over the
finite field $\mathbb{F}_7$ with $7$ elements with determinant $0$ or $1$, then
$|N| = 721$ and so $|N|(|N| - 1) / 2 = 259,560$, but the numbers of relative
$\mathscr{J}$- and $\mathscr{R}$-classes are $36$ and $1,862$, there are $7$
and $376$ principal 2-sided and right congruences, respectively, and the total
number of 2-sided congruences is $10$. Of course, there are other examples
where there is no reduction in the number of principal congruences that must be
generated, and as such the time taken to compute the relative Green's classes
is wasted; see \cref{appendix-benchmarks} for a more thorough analysis.

One question we have not yet addressed is how to compute a (principal)
congruence from its generating pairs. For the purpose of computing the lattice
of congruences, it suffices to be able to compare congruences by containment.
There are a number of different approaches to this: such as the algorithm
suggested in~\cite[Algorithm 2]{Freese2008aa} and implemented
in~\cite{Araujo2022aa}, and that suggested in~\cite[Chapter 2]{Torpey2019aa}
and implemented in the \GAP package \Semigroups and the C++ library
\libsemigroups. The former is essentially a brute force enumeration of the
pairs belonging to the congruence, and congruences are represented by the
well-known disjoint sets data structure. The latter involves running two
instances of the Todd--Coxeter Algorithm in parallel; see~\cite[Chapter
2]{Torpey2019aa} and~\cite{Coleman2022aa} for more details.

We conclude this introduction with some comments about the relative merits and
de-merits of the different approaches outlined above, and we refer the reader
to~\cref{appendix-benchmarks} for some justification for the claims we are
about to make.

As might be expected, the runtime of the low-index congruence algorithm is
highly dependent on the input presentation, and it seems difficult (or
impossible) to predict what properties of a presentation reduce the runtime. We
define the \defn{length of a presentation} to be the sum of the lengths of the
words appearing in relations plus the number of generators.
On the one hand, for a fixed monoid $M$, long presentations appear to have an
adverse impact on the performance, but so too do very short presentations.
Perhaps one explanation for this is that a long presentation increases the cost
of processing each node in the search tree, while a short presentation does not
make it evident that certain branches of the tree contain no
solutions until many
nodes in the tree have been explored. If $M$ is finite, then it is possible to
find a presentation for $M$ using, for example, the Froidure-Pin
Algorithm~\cite{Froidure1997aa}. In some examples, the presentations produced
mechanically (i.e.\ non-human presentations) qualify as long in the preceding
discussion. In some examples presentations from the literature (i.e.\ human
presentations) work better than their non-human counterparts, but in other
examples they qualify as short in the preceding discussion, and are worse.
It seems that in many cases some experimentation is required to find a
presentation for which the low-index congruence algorithm works best. On the
other hand, in examples where the number of congruences is large, say in the
millions, running any implementation of the second algorithm is infeasible
because it requires too much space. Having said that, there are still some
relatively small examples where the low-index congruences algorithm is faster
than the implementations of the second algorithm in \Semigroups and in
\CREAM, and others where the opposite holds; see \cref{table-numbers}.

One key difference between the implementation of the low-index subgroup
algorithm in, say, \GAP, and the low-index congruence algorithm in
\libsemigroups is that the former finds conjugacy class representatives of
subgroups with index at most $n\in \N$. As far as the authors are aware, there
is no meaningful notion of conjugacy that can be applied to the low-index
congruences algorithm for semigroups and monoids in general.
Despite the lack of any optimizations for groups in the
implementation of the low-index congruence algorithm in
\libsemigroups, its performance is often better than or comparable to
that of the implementations of Sims' low-index subgroup algorithm in
\cite{Culler2025aa} and \GAP; see \cref{table-infinite} for more details.

The present paper is organised as follows: in \cref{section-prelims} we present
some preliminaries required in the rest of the paper; in
\cref{section-word-graphs} we state and prove some results related to actions,
word graphs, and congruences; in \cref{section-low-index} we state the low-index
congruence algorithm and prove that it is correct; in
\cref{section-applications-of-low-index-algo} we give numerous applications of
the low-index congruences algorithm; in \cref{section-joins-meets} we show how
to compute the joins and meets of congruences represented by word graphs; in
\cref{section-algo-2} we describe the algorithm based on~\cite{East2019aa} for
computing relative Green's relations. In \cref{appendix-benchmarks}, we
provide some benchmarks that compare the performance of the implementations in
\libsemigroups, \Semigroups, and \CREAM. Finally in
\cref{appendix-numbers} we present some tables containing
statistics about the lattices of congruences of some well-known families of
monoids.

\section{Preliminaries}\label{section-prelims}

In this section we introduce some notions that are required for the latter
sections of the paper.

Throughout the paper we use the symbol $\bot$ to denote an ``undefined'' value,
and note that $\bot$ is not an element of any set except where it is explicitly
included.

Let $S$ be a semigroup. An equivalence relation $\rho\subseteq S\times S$ is a
\defn{right congruence} if $(xs, ys)\in \rho$ for all $(x,y) \in \rho$ and all
$s\in S$. \defn{Left congruences} are defined analogously, and a
\defn{2-sided congruence} is both a left and a right congruence. We refer to
the number of classes of a congruence as its \defn{index}.

If $X$ is any set, and $\Psi:X\times S\to X$ is a function, then $\Psi$ is a
\defn{right action} of $S$ on $X$ if  $((x, s)\Psi, t)\Psi=(x, st)\Psi$ for all
$x\in X$ and for all $s, t \in S$. If in addition $S$ has an identity element
$1_S$ (i.e.\ if $S$ is a \defn{monoid}), we require $(x, 1_S)\Psi = x$ for all
$x\in X$ also. \defn{Left actions} are defined dually. In this paper we will
primarily be concerned with right actions of monoids.

If $M$ is a monoid, and $\Psi_0: X_0 \times M\to X_0$ and $\Psi_1: X_1\times M
\to X_1$ are right actions of $M$ on sets $X_0$ and $X_1$, then we say that
$\lambda: X_0\to X_1$ is a \defn{homomorphism of the right actions $\Psi_0$ and
$\Psi_1$} if
\[
  (x, s)\Psi_0\lambda = ((x)\lambda, s)\Psi_1
\]
for all $x\in X$ and all $s\in M$.
An \defn{isomorphism} of right actions is a bijective homomorphism.

Let $A$ be any alphabet and let $A ^ *$ denote the \defn{free monoid} generated
by $A$ (consisting of all words over $A$ with operation juxtaposition and
identity the empty word $\varepsilon$).
We define a \defn{word graph} $\Gamma = (V, E)$ over the alphabet $A$ to be
a digraph with set of nodes $V$ and edges $E\subseteq V \times A \times V$.
Word graphs are essentially finite state automata without initial or accept
states. More specifically, if $\Gamma = (V, E)$ is a word graph over
$A$, then for any
$\alpha \in V$ and any $Q_1\subseteq V$, we can define a finite state automaton
$(V, A, \alpha, \delta, Q_1)$ where: the state set is $V$; the alphabet is $A$;
the start state is $\alpha$; the transition function $\delta: V\times A \to V$
is defined by $(\alpha, a)\delta = \beta$ whenever $(\alpha, a, \beta)\in E$;
and $Q_1$ denotes the accept states.

If $(\alpha, a, \beta)\in E$ is an edge in a word graph $\Gamma$, then $\alpha$
is the \defn{source}, $a$ is the \defn{label}, and $\beta$ is the \defn{target}
of $(\alpha, a, \beta)$.
A word graph
$\Gamma$ is \defn{complete} if for every node $\alpha$ and every letter $a\in
A$ there is at least one edge with source $\alpha$ labelled by $a$. A word
graph $\Gamma = (V, E)$ is \defn{finite} if the sets of nodes $V$ and edges $E$
are finite. A word graph is \defn{deterministic} if for every node $\alpha\in
V$ and every $a\in A$ there is at most one edge with source $\alpha$ and label
$a$.

If $\alpha, \beta \in V$, then an \defn{$(\alpha, \beta)$-path} is a sequence of
edges $(\alpha_0, a_0, \alpha_{1}), \ldots, (\alpha_{n - 1}, a_{n - 1},
\alpha_{n})\in E$ where $\alpha_0 = \alpha$ and $\alpha_{n} = \beta$ and $a_0,
\ldots, a_{n - 1}\in A$; $\alpha$ is the \defn{source} of the path; the word
$a_0\cdots a_{n - 1}\in A ^ *$ \defn{labels} the path; $\beta$ is the
\defn{target} of the path; and the \defn{length} of the path is $n$. We say that
there exists an $(\alpha, \alpha)$-path of length $0$ labelled by $\varepsilon$
for all $\alpha \in V$. For $\alpha\in V$, and $u\in A^\ast$  we will
write $\alpha
\cdot_{\Gamma} u = \beta \in V$ to mean that $u$ labels a $(\alpha,
\beta)$-path in
$\Gamma$, and $\alpha \cdot_{\Gamma} u = \bot$ if $u$ does not label
a path with source
$\alpha$ in $\Gamma$. When there are no opportunities for ambiguity
we will omit the subscript $\Gamma$ from $\cdot_{\Gamma}$.
If $\alpha, \beta \in V$ and there is an
$(\alpha, \beta)$-path in $\Gamma$, then we say that $\beta$ is
\defn{reachable} from $\alpha$. If $\alpha$ is a node in a word graph $\Gamma$,
then the \defn{strongly connected component of $\alpha$} is the set of all
nodes $\beta$ such that $\beta$ is reachable from $\alpha$ and $\alpha$ is
reachable from $\beta$. If $\Gamma = (V, E)$ is a word graph and
$\mathfrak{P}(A ^ * \times A ^ *)$ denotes the power set of $A ^ * \times A ^
*$, then the
\defn{path relation of $\Gamma$} is the function $\pi_{\Gamma}: V
\to \mathfrak{P}(A ^ * \times A ^ *)$ defined by $(\alpha)\pi_{\Gamma}
= \set{(u, v)\in A ^ * \times A ^ *}{\alpha\cdot u \not= \bot\text{ and
}\alpha\cdot u = \alpha\cdot v}$. If $\Gamma$ is a complete word graph and
$\alpha$ is a node in $\Gamma$, then $(\alpha)\pi_{\Gamma}$ is a
right congruence on $A ^ *$. If $R\subseteq A ^ *\times A ^ *$, $\Gamma$ is a
word graph, and $\pi_{\Gamma}$ is the path relation of $\Gamma$, then we say
that $\Gamma$ is \defn{compatible with $R$} if
$\alpha\cdot u = \alpha\cdot v$ whenever $\alpha\cdot u \neq \bot$ and
$\alpha\cdot v \neq \bot$ for all $(u, v)\in R$ and for all $\alpha \in V$.
Equivalently, if $\Gamma$ is complete, then $\Gamma$ is compatible with $R$ if
and only if $R\subseteq (\alpha)\pi_{\Gamma}$ for every node $\alpha$ in
$\Gamma$. It is routine to verify that if a word graph $\Gamma$ is compatible
with $R$, then it is also compatible with the least (2-sided) congruence on $A ^
*$ containing $R$; we denote this congruence by $R ^{\#}$.

%%%%%%%%%%%%%%%%%%%%%%%%%%%%%%%%%%%%%%%%%%%%%%%%%%%%%%%%%%%%%%%%%%%%%%%%
\section{Word graphs and right congruences}\label{section-word-graphs}
%%%%%%%%%%%%%%%%%%%%%%%%%%%%%%%%%%%%%%%%%%%%%%%%%%%%%%%%%%%%%%%%%%%%%%%%

In this section we establish some fundamental results related to word graphs and
right congruences. It seems likely to the authors that the results presented in
this section are well-known; similar ideas occur in
~\cite{Jura1978aa},~\cite{Kilp2000aa},~\cite{Sims1994aa}, and probably
elsewhere. However, we did not find any suitable references that suit our
purpose here, particularly in \cref{section-standard-word-graphs}, and have
included some of the proofs of these results for completeness.
% TODO(later) remove this comment if we remove the proofs.

This section has three subsections: in \cref{subsection-right-congruences} we
describe the relationship between right congruences and word graphs; in
\cref{subsection-homomorphisms-of-word-graphs} we present some results about
homomorphisms between word graphs; and finally in
\cref{section-standard-word-graphs} we give some technical results about
standard word graphs.

\subsection{Right congruences}%
\label{subsection-right-congruences}

Suppose that $M$ is the monoid defined by the monoid presentation $\langle A
\mid R\rangle$ and $\Psi: V \times M\to V$ is a right action of $M$ on a set
$V$. Recall that $M$ is isomorphic to the quotient of the free monoid $A ^ *$
by $R ^{\#}$. If $\theta: A ^ * \to M$ is the surjective homomorphism with
$\ker(\theta) = R ^ {\#}$, then we can define a complete deterministic word
graph $\Gamma = (V, E)$ over the alphabet $A$ such that $(\alpha, a, \beta) \in
E$ whenever $(\alpha, (a)\theta)\Psi = \beta$. Conversely, if $\Gamma = (V, E)$
is a complete word graph that is compatible with $R$, then we define $\Psi: V
\times M \to V$ by
\begin{equation}\label{eq-action-from-word-graph}
  (\alpha, (w)\theta)\Psi = \beta
\end{equation}
whenever $w$ labels an $(\alpha, \beta)$-path in $\Gamma$. It is routine to
verify that the functions just described mapping a right action to a word
graph, and vice versa, are mutual inverses. For future reference, we record
these observations in the next proposition.

\begin{prop}\label{prop-word-graph-iff-action}
  Let $M$ be the monoid defined by the monoid presentation $\langle A \mid
  R\rangle$. Then there is a one-to-one correspondence between right actions of
  $M$ and complete deterministic word graphs over $A$ compatible with $R$.
\end{prop}

If $\Gamma_0= (V_0, E_0)$ and $\Gamma_1= (V_1, E_1)$ are word graphs over the
same alphabet $A$, then $\phi:V_0\to V_1$ is a \defn{word graph homomorphism}
if $(\alpha, a, \beta)\in E_0$ implies $((\alpha)\phi, a, (\beta)\phi)\in E_1$;
and we write $\phi: \Gamma_0\to \Gamma_1$. An \defn{isomorphism} of word graphs
$\Gamma_0$ and $\Gamma_1$ is a bijection $\phi: \Gamma_0\to \Gamma_1$ such that
both $\phi$ and $\phi ^ {-1}$ are homomorphisms. If $\phi: \Gamma_0 \to
\Gamma_1$ is a word graph homomorphism and $w \in A ^ *$ labels an $(\alpha,
\beta)$-path in $\Gamma_0$, then it is routine to verify that $w$ labels a
$((\alpha)\phi, (\beta)\phi)$-path in $\Gamma_1$.

The correspondence between right actions and word graphs given in
\cref{prop-word-graph-iff-action} also preserves isomorphisms, which we record
in the next proposition.

\begin{prop}
  Let $M$ be a monoid generated by a set $A$, let $\Psi_0$ and
  $\Psi_1$ be right actions of $M$, and
  let $\Gamma_0$ and $\Gamma_1$ be the word graphs over $A$
  corresponding to these
  actions. Then $\Psi_0$ and $\Psi_1$ are isomorphic actions if and only if
  $\Gamma_0$ and $\Gamma_1$ are isomorphic word graphs.
\end{prop}

\begin{proof}
  Assume that $\Gamma_0 = (V_0, E_0)$ and $\Gamma_1= (V_1, E_1)$ are isomorphic
  word graphs. Then there exists a bijection $\phi: V_0 \to V_1$. Since
  $\Gamma_0$ and $\Gamma_1$ are isomorphic word graphs, $w\in A ^ *$ labels an
  $(\alpha, \beta)$-path in $\Gamma_0$ if and only if $w$ labels a
  $((\alpha)\phi, (\beta)\phi)$-path in $\Gamma_1$. Hence $((\alpha,
  (w)\theta)\Psi_0)\phi = ((\alpha)\phi, (w)\theta)\Psi_1$ for all $\alpha \in
  V_0, w \in A^*$ and so $\phi$ is an isomorphism of the actions $\Psi_0$ and
  $\Psi_1$.

  Conversely, assume that $\Psi_0$ and $\Psi_1$ are isomorphic right actions of
  $M$ on sets $V_0$ and $V_1$, respectively. Then there exists a bijective
  homomorphism of right actions $\lambda: V_0 \to V_1$ such that $((\alpha,
  s)\Psi_0)\lambda = ((\alpha)\lambda, s)\Psi_1$ for all $\alpha \in V_0, s \in
  M$. We will prove that $\lambda: \Gamma_0 \to \Gamma_1$ is a word graph
  isomorphism. It suffices to show that $(\alpha, a, \beta) \in E_0$ if and
  only if $((\alpha)\lambda, a, (\beta)\lambda) \in E_1$. If $(\alpha, a,
  \beta) \in E_0$, then $(\alpha, (a)\theta)\Psi_0 = \beta$ and hence
  $((\alpha, (a)\theta)\Psi_0)\lambda = (\beta)\lambda$. Since $\lambda$ is an
  isomorphism of right actions $((\alpha, (a)\theta)\Psi_0)\lambda =
  ((\alpha)\lambda, (a)\theta)\Psi_1 = (\beta)\lambda$ and so
  $((\alpha)\lambda, a, (\beta)\lambda) \in E_1$. Similarly, it can be shown
  that if $((\alpha)\lambda, a, (\beta)\lambda) \in E_1$, then $(\alpha, a,
  \beta) \in E_0$ and hence $\lambda$ defines a word graph isomorphism.
\end{proof}

In \cref{section-low-index} we are concerned with enumerating the right
congruences of a monoid $M$ subject to certain restrictions, such as those
containing a given set $B\subseteq M \times M$ or those with a given number of
equivalence classes. If $M$ is the monoid defined by the presentation $\langle
A \mid R\rangle$ and $\rho$ is a right congruence on $M$, then the
function $\Psi:
M/\rho \times M \to M/\rho$ defined by $(x/\rho, y) \Psi = xy/\rho$ is a right
action of $M$ on $M/\rho$ where $x/\rho$ is the equivalence class of $x$ in
$\rho$ and $M/\rho = \set{x/\rho}{x\in M}$.  It follows by
\cref{prop-word-graph-iff-action} that $\rho$ corresponds to a complete
deterministic word graph $\Gamma$ over $A$ compatible with $R$. In particular,
the nodes of $\Gamma$ are the classes $M/\rho$, and the edges are $\{(x/\rho,
a, (x, a)\Psi): x\in M,\ a\in A\}$.

On the other hand, not every right action of a monoid is an action on the
classes of a right congruence. For example, if $S ^ 1$ is a $4 \times 4$
rectangular band with an adjoined identity, then two faithful right actions of
$S ^ 1$ are depicted in~\cref{fig-non-cyclic-monoid-action} with respect to the
generating set $\{a := (1, 1), b := (2, 2), c := (3, 3), d := (4, 4)\} \cup
\{1_{S^1}\}$. It can be shown that the action of $S^1$ on $\{0, \ldots, 6\}$
shown in \cref{fig-non-cyclic-monoid-action} corresponds to a right congruence
of $S^1$ but that the action of $S^1$ on $\{0, \ldots, 5\}$ does not.

\begin{figure}
  \centering

  \begin{tikzpicture}[scale=1.2]
    %\tikzstyle{every node}=[style=hole]
    \node[draw] (0) at (1, 2) {$0$};
    \node[draw] (1) at (2.5, 3) {$1$};
    \node[draw] (2) at (2.5, 1) {$2$};
    \node[draw] (3) at (-0.5, 1) {$3$};
    \node[draw] (4) at (1, 1) {$4$};
    \node[draw] (5) at (4, 1) {$5$};

    \draw[style=e, loop above] (0) to (0);
    \draw[style=a]             (0) to (2);
    \draw[style=b] (0) to (3);
    \draw[style=c] (0) to (4);
    \draw[style=d, bend left]  (0) to (5);

    \draw[style=e, loop above] (1) to (1);
    \draw[style=a, parallel_arrow_1]      (1) to (4);
    \draw[style=b, bend left]      (1) to (2);
    \draw[style=c, parallel_arrow_2]      (1) to (4);
    \draw[style=d, bend left] (1) to (5);

    \draw[style=e, loop above] (2) to (2);
    \draw[style=a, loop left]  (2) to (2);
    \draw[style=b, loop below] (2) to (2);
    \draw[style=c, parallel_arrow_1]            (2) to (5);
    \draw[style=d, parallel_arrow_2]            (2) to (5);

    \draw[style=e, loop right] (4) to (4);
    \draw[style=a, loop below] (4) to (4);
    \draw[style=b, bend left, out=45]  (4) to (3);
    \draw[style=c, loop, out=-25, in=-65, looseness=8] (4) to (4);
    \draw[style=d, bend right, out=-40] (4) to (3);

    \draw[style=e, loop above] (3) to (3);
    \draw[style=a, parallel_arrow_1]            (3) to (4);
    \draw[style=b, loop below] (3) to (3);
    \draw[style=c, parallel_arrow_2]            (3) to (4);
    \draw[style=d, loop left]  (3) to (3);

    \draw[style=e, loop right] (5) to (5);
    \draw[style=a, bend left]  (5) to (2);
    \draw[style=b, bend right] (5) to (2);
    \draw[style=c, loop below] (5) to (5);
    \draw[style=d, loop, out=-25, in=-65, looseness=8] (5) to (5);
  \end{tikzpicture}
  \qquad\qquad
  \centering
  \begin{tikzpicture}[scale=1.2]
    %\tikzstyle{every node}=[style=hole]
    \node[draw] (0) at (-0.75, 2.5) {$0$};
    %\node[draw] (1) at (2, 3) {$1$};
    \node[draw] (1) at (-3, 1) {$1$};
    \node[draw] (2) at (-1.5, 1) {$2$};
    \node[draw] (3) at (0, 1) {$3$};
    \node[draw] (4) at (1.5, 1) {$4$};
    \node[draw] (5) at (0, -0.5) {$5$};
    \node[draw] (6) at (1.5, -0.5) {$6$};

    \draw[style=e, loop above] (0) to (0);
    \draw[style=a] (0) to (1);
    \draw[style=b] (0) to (2);
    \draw[style=c] (0) to (3);
    \draw[style=d] (0) to (4);

    \draw[style=a, loop left]   (1) to (1);
    \draw[style=b, loop below]  (1) to (1);
    \draw[style=c, loop right]  (1) to (1);
    \draw[style=d, loop above]  (1) to (1);
    \draw[style=e, loop, out=-25, in=-65, looseness=8] (1) to (1);

    \draw[style=a, loop left]  (2) to (2);
    \draw[style=b, loop below] (2) to (2);
    \draw[style=c, loop right] (2) to (2);
    \draw[style=d, loop above] (2) to (2);
    \draw[style=e, loop, out=-25, in=-65, looseness=8] (2) to (2);

    \draw[style=e, loop above] (3) to (3);
    \draw[style=a, loop left] (3) to (3);
    \draw[style=b, parallel_arrow_1]  (3) to (5);
    \draw[style=c, loop right] (3) to (3);
    \draw[style=d, parallel_arrow_2] (3) to (5);

    \draw[style=e, loop above] (4) to (4);
    \draw[style=a, parallel_arrow_1]  (4) to (6);
    \draw[style=b, loop left] (4) to (4);
    \draw[style=c, parallel_arrow_2] (4) to (6);
    \draw[style=d, loop right]  (4) to (4);

    \draw[style=e, loop below] (5) to (5);
    \draw[style=a, loop left]  (5) to (5);
    \draw[style=b, loop right] (5) to (5);
    \draw[style=c, bend left] (5) to (3);
    \draw[style=d, bend right] (5) to (3);

    \draw[style=e, loop below] (6) to (6);
    \draw[style=a, loop left]  (6) to (6);
    \draw[style=b, bend left] (6) to (4);
    \draw[style=c, loop right] (6) to (6);
    \draw[style=d, bend right] (6) to (4);

  \end{tikzpicture}

  \caption{Two faithful right actions of a $4\times 4$ rectangular band with
    identity adjoined; ({\color{color0} $a = (1, 1)$}, {\color{color1} $b = (2,
      2)$}, {\color{color2} $c = (3, 3)$}, {\color{color3} $d = (4, 4)$} and
  {\color{color3} for the adjoined identity}).}
  \label{fig-non-cyclic-monoid-action}
\end{figure}
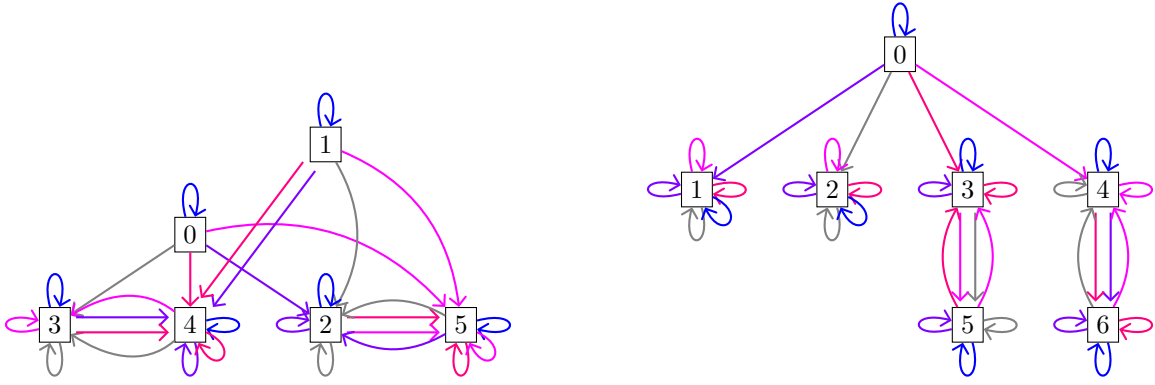

In a word graph $\Gamma$ corresponding to the action of a monoid $M$ on a right
congruence $\rho$, if $m\in M$, then there exist $a_0, \ldots, a_{n-1} \in A$
such that $m = a_0 \cdots a_{n - 1}$, and so $m$ labels the path from $1_M$ to
$m/\rho$ in $\Gamma$. In particular, every node in $\Gamma$ is reachable from
the node $1_M/\rho$. The converse statement, that every complete deterministic
word graph compatible with $R$ where every node is reachable from $1_M/\rho$
corresponds to a right congruence on $M$ is established in the next
proposition.

% It is not difficult to show that enumerating the right congruences of $M$ is
% the same as enumerating the corresponding word graphs that satisfy the
% additional property that each node is reachable from an ``initial'' node that
% corresponds to the congruence class of the identity $1_M$ of the monoid $M$
% and that we will be denoting by 0. We record this in the following results.

%\begin{prop}\label{congruence-to-graph} Suppose that $\rho$ is a right
%congruence of the monoid $M$ and that $\Gamma$ is the word graph corresponding
%to the right action $\Psi: M/\rho \times M \to M/\rho$ defined by $(x/\rho, y)
%\Psi = xy/\rho$. Then every node of $\Gamma$ is reachable from the node
%corresponding to $1_M/\rho$. \end{prop}

%The proof follows directly from the nature of the action of a monoid on the
%classes of any congruence by right multiplication. Similarly, we obtain the
%following result.

If $S$ and $T$ are semigroups, $R\subseteq S\times S$ is a binary relation and
$\theta: S\to T$ is a homomorphism. Then abusing our notation somewhat
we write
\[(R)\theta = \left\{((u)\theta, (v)\theta)\in T\times T : (u,v) \in
R\right\}.\]

\begin{prop}\label{graph-to-congruence}
  Let $M$ be the monoid defined by the monoid presentation $\langle
  A\mid R\rangle$,
  let $\theta : A ^ * \to M$ be the unique homomorphism with $\ker(\theta) = R ^
  {\#}$, and let $\Gamma = (V, E)$ be a word graph over $A$ that is complete,
  deterministic, compatible with $R$ and where every node is reachable from some
  $\alpha \in V$. Then $\rho = ((\alpha)\pi_{\Gamma})\theta$ is a right
  congruence on $M$, where $(\alpha)\pi_{\Gamma}$ is the path relation on
  $\Gamma$.
\end{prop}
\begin{proof}
  Suppose that $((u)\theta, (v)\theta)\in \rho$ for some $(u, v)\in
  (\alpha)\pi_{\Gamma}$ and that $w\in A^*$ is arbitrary. Then $\alpha\cdot u
  \neq \bot$ and $\alpha\cdot u = \alpha \cdot v$. Hence $\alpha\cdot uw =
  (\alpha\cdot u) \cdot w = (\alpha\cdot v) \cdot w = \alpha \cdot vw$. Since
  $\Gamma$ is complete, $\alpha\cdot uw \neq \bot$, and so $(uw, vw) \in
  (\alpha)\pi_{\Gamma}$. Therefore $((uw)\theta, (vw)\theta)\in \rho$, and
  $\rho$ is a right congruence.
\end{proof}

%It is not difficult to show that the actions of $M$ on $M /\rho$ and the
%action corresponding to $\Gamma$ are isomorphic; for the sake of completeness
%a proof is included in \cref{thm-graph-to-right-congruence}.

% \begin{prop}
%   Let $M$ be a monoid defined by a monoid presentation $\langle A
% \mid R\rangle$.
%   If $\rho$ is a right congruence on $M$ and $\Gamma$ is the word graph from
%   \cref{prop-word-graph-iff-action} corresponding to the action of $M$ on
%   $M/\rho$, then the right action corresponding to $\Gamma$ is isomorphic to
%   the right action of $M$ on $M/\rho$.
% \end{prop}

It follows from \cref{graph-to-congruence} that in order to enumerate the right
congruences of $M$ it suffices to enumerate the word graphs over $A$ that are
complete, deterministic, compatible with $R$ and where every node is reachable
from a node $0$. Of course, we only want to obtain every right congruence of
$S$ once, for which we require the next proposition. Henceforth we will suppose
that if $\Gamma = (V, E)$ is a word graph over $A$, then $V = \{0, \ldots, n -
1\}$ for some $n \geq 1$ and so, in particular, $0$ is always a node in
$\Gamma$. Moreover, we will assume that the node $0$ corresponds to the
equivalence class $1_M/\rho$ of the identity $1_M$ of $M$.

\begin{prop}\label{prop-same-congruence}
  Let $\Gamma_0$ and $\Gamma_1$ be word graphs over $A$ corresponding to right
  congruences of a monoid $M$ generated by $A$. Then $\Gamma_0$ and
  $\Gamma_1$ represent the
  same right congruence of $M$ if and only if there exists a word graph
  isomorphism $\phi: \Gamma_0 \to \Gamma_1$ such that $(0)\phi = 0$.
\end{prop}

\begin{proof}
  We denote the right congruences associated to $\Gamma_0$ and $\Gamma_1$ by
  $\rho_0$ and $\rho_1$, respectively.

  Suppose that there exists a word graph isomorphism $\phi: \Gamma_0 \to
  \Gamma_1$ such that $(0)\phi = 0$ and that the words $x, y\in A ^
  *$ label $(0,
  \alpha)$-paths in $\Gamma_0$. Then $x$ and $y$ label $(0, (\alpha)\phi)$-paths
  $\Gamma_1$ and so $\rho_{0}\subseteq \rho_{1}$. Since $\phi^{-1}: \Gamma_1
  \to \Gamma_0$ is a word graph isomorphism satisfying $(0)\phi^{-1} = 0$,
  it follows that $\rho_{1}\subseteq \rho_{0}$, and so $\rho_0 = \rho_1$, as
  required.

  Conversely, assume that $\Gamma_0 = (V_0, E_0)$ and $\Gamma_1=(V_1, E_1)$
  represent the same right congruence of $M$. We define a word graph
  homomorphism $\phi: \Gamma_0 \to \Gamma_1$ as follows. For every $\alpha\in
  V_0$, fix a word $w_{\alpha} \in A ^ *$ that labels a $(0, \alpha)$-path in
  $\Gamma_0$. We define $(\alpha)\phi$ to be the target of the path with source
  $0$ labelled by $w_{\alpha}$ in $\Gamma_1$. If $(\alpha, a, \beta)\in E_0$,
  then both $w_{\alpha}a$ and $w_{\beta}$ label $(0, \beta)$-paths in
  $\Gamma_0$, and so $(w_{\alpha}a , w_{\beta}) \in (0)\pi_{\Gamma_0} = \rho_0
  = \rho_1 = (0)\pi_{\Gamma_1}$. Hence $((\alpha)\phi, a, (\beta)\phi)$ is an
  edge of $\Gamma_1$. In addition, $|V_0| = |V_1|$ and it is clear that if
  $\alpha \neq \beta$, then $(\alpha)\phi \neq (\beta)\phi$ and hence $\phi$ is
  a bijection.
\end{proof}

It follows from \cref{prop-same-congruence} that to enumerate the right
congruences of a monoid $M$ defined by the presentation $\langle A
\mid R\rangle$
it suffices to enumerate the complete deterministic word graphs over $A$
compatible with $R$ where every node is reachable from 0 up to isomorphism. On
the face of it, this is not much of an improvement because, in general, graph
isomorphism is a difficult problem. However, word graph isomorphism, in the
context of \cref{prop-same-congruence}, is trivial by comparison.

If the alphabet $A = \{a_0, a_1, \ldots, a_{n - 1}\}$ and $u, v \in A ^ *$,
then we define $u \leq v$  if $|u| < |v|$ or $|u| = |v|$ and $u$ is less
than or equal to $v$ in the usual lexicographic order on $A ^ *$
arising from the
linear order $a_0 < a_1 < \cdots < a_{n - 1}$ on $A$. The order $\leq$ is
called the \defn{short-lex ordering} on $A ^ *$. If $u\neq v$ and $u\leq v$,
then we write $u< v$. If $\alpha$ is a node in a word graph $\Gamma$ and
$\alpha$ is reachable from $0$, then there is a $\leq$-minimum word
labelling a $(0, \alpha)$-path; we denote this path by $w_{\alpha}$.

The following definition is central to the algorithm presented in
\cref{section-low-index}.

\begin{de}\label{de-standard}
  A complete word graph $\Gamma = (V, E)$ over $A$ is \defn{standard}
  if the following
  hold:
  \begin{enumerate}[(i)]
    \item
      $\Gamma$ is deterministic;

    \item
      every node is reachable from $0\in V =  \{0, \ldots, |V|-1\}$;

    \item
      for all $\alpha, \beta \in V$, $\alpha< \beta$ if and only if $w_\alpha
      < w_\beta$.
  \end{enumerate}
\end{de}

\begin{prop}\label{prop-isomorphic-graphs}[cf. Proposition 8.1
  in~\cite{Sims1994aa}]
  Let $\Gamma_0$ and $\Gamma_1$ be standard complete word graphs over the same
  alphabet. Then there exists a word graph isomorphism $\phi: \Gamma_0
  \to\Gamma_1$ such that $(0)\phi = 0$ if and only if $\Gamma_0 = \Gamma_1$.
\end{prop}

By \cref{prop-isomorphic-graphs}, every complete deterministic word graph
$\Gamma$ in which every node is reachable from 0 is isomorphic to a unique
standard complete word graph. It is straightforward to compute this standard
word graph from the original graph by relabelling the nodes, for example, using
the procedure \texttt{SWITCH} from~\cite[Section 4.7]{Sims1994aa}. We refer to
this process as \defn{standardizing} $\Gamma$.

Another consequence of \cref{prop-same-congruence} and
\cref{prop-isomorphic-graphs} is that enumerating the right congruences of a
monoid $M$ defined by a presentation $\langle A \mid R\rangle$ is equivalent to
enumerating the standard complete word graphs over $A$ that are compatible with
$R$. As mentioned above, the right action of $M$ on the classes of a right
congruence is isomorphic to the right action represented by the corresponding
word graph $\Gamma$.  We record this in the following theorem.

\begin{thm}\label{thm-graph-to-right-congruence}
  Let $M$ be a monoid defined by a monoid presentation $\langle A
  \mid R\rangle$.
  Then there is a one-to-one correspondence between the right congruences of
  $M$ and the standard complete word graphs over $A$ compatible with $R$.

  If $\rho$ is a right congruence on $M$ and $\Gamma$ is the corresponding word
  graph, then the right actions of $M$ on $M/\rho$ and on $\Gamma$ are
  isomorphic; and $\rho =((0)\pi_{\Gamma})\theta$ where $\theta : A ^ * \to M$
  is the unique homomorphism with $\ker(\theta) = R ^ {\#}$ and
  $(0)\pi_{\Gamma}$ is the path relation on $\Gamma$.
\end{thm}

It is possible to determine a one-to-one correspondence between the right
congruences of a semigroup $S$ and certain word graphs arising from $S ^ 1$, as
a consequence of \cref{thm-graph-to-right-congruence}.

\begin{cor}\label{cor-semigroups-right}
  Let $S$ be a semigroup defined by a semigroup presentation $\langle A \mid
  R\rangle$. Then there is a one-to-one correspondence between the right
  congruences of $S$ and the complete  word graphs $\Gamma = (V,
  E)$ for $S ^ 1$ over $A$ compatible with $R$ such that $(v, a, 0)\not\in  E$
  for all $v\in V$ and all $a\in A$.

  If $\rho$ is a right congruence of $S$ and $\Gamma$ is the corresponding word
  graph for $S ^ 1$, then the right actions of $S$ on $S/ \rho$ and on
  $\Gamma\setminus \{0\}$ are isomorphic.
\end{cor}
\begin{proof}
  By \cref{thm-graph-to-right-congruence} it follows that there is a one-to-one
  correspondence between the right congruences of $S ^ 1$ and the complete
  word graphs $\Gamma = (V, E)$ for $S ^ 1$ over $A$ compatible
  with $R$. In addition, there is a one-to-one correspondence between the right
  congruences of $S$ and the right  congruences $\rho$ of $S ^ 1$
  such that $1_M/
  \rho = x/\rho$ if and only if $x = 1_M$. Since $\rho$ and $(0)\pi_{\Gamma}$
  coincide it follows that there is a one-to-one correspondence
  between the right
  congruences $\rho$ of $S ^ 1$ such that $1_M/ \rho = x/\rho$ if and only if $x
  = 1_M$ and the complete  word graphs $\Gamma = (V, E)$ for $S ^ 1$
  over $A$ compatible with $R$ such that $(v, a, 0)\not\in  E$ for all $v\in V$
  and all $a\in A$. The argument to prove that the right actions of $S$ on
  $S/\rho$ and on $\Gamma \setminus \{0\}$ are isomorphic is identical to the
  argument in the proof of \cref{thm-graph-to-right-congruence}.
\end{proof}

Given \cref{thm-graph-to-right-congruence}, we will refer to the standard
complete word graph over $A$ compatible with $R$ corresponding to a given right
congruence $\rho$ as the \defn{word graph of $\rho$}.

We require the following lemma.

\begin{lemma}\label{lem-containing}
  Let $S$ and $T$ be semigroup, let $R \subseteq S \times S$, and let $\theta: S
  \to T$ be a homomorphism. Then the following hold:
  \begin{enumerate}[\rm (i)]
    \item
      if $\ker(\theta)\subseteq R$, $R$ is transitive and $u, v\in S$ are such
      that $(u)\theta = (u')\theta$ and $(v)\theta = (v')\theta$ for some $(u',
      v') \in R$, then $(u, v) \in R$;
    \item\label{lem-rho-preserves-least-congruence}
      if $\theta$ is surjective and $\ker(\theta)$ is contained in the least
      (left, right, or 2-sided) congruence $\rho$ on $S$ containing $R$, then
      the least (left, right, or 2-sided) congruence on $T$
      containing $(R)\theta$  is $(\rho)\theta$.
  \end{enumerate}
\end{lemma}

\cref{lem-containing}(i) can be reformulated as follows.
\begin{cor}\label{lem-rho-theta-rho-ker}
  If $\ker(\theta) \subseteq R$ and $R$ is transitive, then
  $((u)\theta, (v)\theta)\in (R)\theta$ if and only if
  $(u, v) \in R$
  for all $u, v\in S$.
\end{cor}

A convenient consequence of the correspondence between right congruences
$\rho$ and their word graphs $\Gamma$ allows us to determine a set of generating
pairs for $\rho$ from $\Gamma$. This method is similar to Stalling's method for
finding a generating set of a subgroup of a free group from its associated
coset graph, see, for example,~\cite[Proposition~6.7]{Kapovich2002aa}.

\begin{lemma}\label{lem-generating-pairs}
  Let $M$ be a monoid defined by a monoid presentation $\langle A
  \mid R\rangle$,
  let $\rho$ be a right congruence of $M$, and let $\Gamma = (V, E)$ be the
  word graph of $\rho$.
  Then
  \(
    \left\{(w_\alpha a, w_\beta) \in A ^ *\times A ^ * : (\alpha, a, \beta)
    \in E \right\}
  \)
  generates the path relation $(0)\pi_{\Gamma}$ as a right
  congruence.

  In particular, if $\theta: A ^ * \to M$ is the natural homomorphism with
  $\ker(\theta) = R ^ \#$, then
  \[
    \left\{((w_\alpha a)\theta, (w_\beta)\theta)
    \in M\times M : (\alpha, a, \beta) \in E \right\}
  \]
  generates $\rho$ as a right congruence.
\end{lemma}
\begin{proof}
  We set $W = \set{w_\alpha}{\alpha\in V}$ where $w_{\alpha}$ is the short-lex
  minimum word labelling any $(0, \alpha)$-path in $\Gamma$ for every $\alpha
  \in V$.

  We denote by $\sigma$ the right congruence on $A ^ *$ generated by $S :=
  \left\{(w_\alpha a, w_\beta) \in A ^ *\times A ^ * : (\alpha, a, \beta) \in
  E\right\}$. If $(w_{\alpha} a, w_\beta) \in S$, then since $(\alpha, a,
  \beta) \in E$, it follows that both $w_{\alpha}a$  and $w_{\beta}$ label $(0,
  \beta)$-paths in $\Gamma$. In particular, $S \subseteq (0)\pi_{\Gamma}$ and
  so, since $\sigma$ is the least right congruence containing $S$ and
  $(0)\pi_{\Gamma}$ is a right congruence, $\sigma \subseteq (0)\pi_{\Gamma}$
  also.

  For the converse inclusion, we will show that $(u, w_\alpha) \in \sigma$ for
  every $u\in A^\ast$ that labels a $(0, \alpha)$-path in $\Gamma$. It will
  follow from this that $(0)\pi_\Gamma \subseteq \sigma$.

  Suppose that $u \in A^\ast$ labels a $(0, \alpha)$-path in $\Gamma$. We write
  $u = w_{\beta}v$ where  $w_{\beta}$ is the longest prefix of $u$ belonging to
  $W$. We proceed by induction on $|v|$. If $|v| = 0$, then $u = w_\beta$, and
  so $\beta = \alpha$. Hence $u = w_\alpha$ and so $(u, w_\alpha)\in \sigma$ by
  reflexivity. This establishes the base case.

  Suppose that for some $k\geq 0$, $(u, w_{\alpha})\in \sigma$ for all $u\in A
  ^ *$ such that $u$ labels a $(0, \alpha)$-path and $u = w_{\beta} v$ where
  $|v| \leq k$. Let $u = w_\beta v$ for some $v\in A ^*$ with $|v| = k+1 > 0$.
  Then we may write $v = av_0$ for some $a\in A$ and $v_0\in A^\ast$ with
  $|v_0| = k$. Since $\Gamma$ is complete, there is an edge $(\beta, a,
  \gamma)\in E$ for some $\gamma \in V$. Hence $(w_\beta a, w_\gamma)\in S$ by
  definition.

  If we set $u_0 = w_\gamma v_0$, then $u_0$ also labels a $(0, \alpha)$-path
  in $\Gamma$. Again we write $u_0 = w_\delta v_1$ where $w_{\delta}$ is the
  maximal prefix of $u_0$ in $W$ and $v_1\in A ^*$. Since $w_\gamma$ is a
  prefix of $u_0$ and $w_{\delta}$ is maximal, $w_{\gamma}$ is a (not
  necessarily proper) prefix of $w_{\delta}$. In particular, $|w_\delta| \geq
  |w_\gamma|$  and so $|v_1| \leq |v_0| = k$. Hence by induction $(u_0,
  w_{\alpha})\in \sigma$. In addition $(u, u_0) = (w_\beta a v_0, w_\gamma v_0)
  \in \sigma$, since $(w_{\beta}a, w_{\gamma})\in S \subseteq \sigma$ and
  $\sigma$ is a right congruence. Therefore by transitivity $(u, w_\alpha)\in
  \sigma$, as required.

  Since $R ^ {\#} = \ker(\theta) \subseteq (0)\pi_{\Gamma} = \sigma$ (we have
    that $R ^ {\#}\subseteq(0)\pi_{\Gamma}$ since $\Gamma$ is complete and
  compatible with $R$), we can apply
  \cref{lem-containing}\ref{lem-rho-preserves-least-congruence} to show that
  $(S)\theta = \left\{((w_\alpha a)\theta, (w_\beta)\theta) \in M\times M :
  (\alpha, a, \beta) \in E \right\}$ generates $\rho$ as a right congruence. It
  follows that $(\sigma)\theta = ((0)\pi_{\Gamma})\theta = \rho$ is generated by
  $(S)\theta$.
\end{proof}

From this point in the paper onwards, for the purposes of describing the
time or space complexity of some claims, we assume that we are using the RAM
model of computation. The following corollary is probably well-known.
% TODO(later) reference?

\begin{cor}\label{cor-generating-pairs}
  Let $M$ be a monoid defined by a monoid presentation $\langle A\mid
  R\rangle$, and let $\rho$ be a right congruence on $M$.  If $\rho$ has finite
  index $n\in \N$, then $\rho$ is finitely generated as a right congruence and
  a set of generating pairs for $\rho$ can be determined in $O(n^2 |A|)$ time
  from the word graph associated to $\rho$.
\end{cor}
\begin{proof}
  This follows immediately since the generating set for $\rho$ given in
  \cref{lem-generating-pairs} has size $n|A|$,
  and it can be found by a breadth first traversal of the word graph. The
  breadth first traversal of the word graph has time complexity $O(n |A|)$.
  To store the generating pairs $(u, v)\in A ^ * \times A ^ *$ for $\rho$
  requires at most $O(n ^ 2|A|)$ space and time, yielding the stated time
  complexity.
\end{proof}

%%%%%%%%%%%%%%%%%%%%%%%%%%%%%%%%%%%%%%%%%%%%%%%%%%%%%%%%%%%%%%%%%%%%%%%%%%%
\subsection{Homomorphisms of word graphs}%
\label{subsection-homomorphisms-of-word-graphs}
%%%%%%%%%%%%%%%%%%%%%%%%%%%%%%%%%%%%%%%%%%%%%%%%%%%%%%%%%%%%%%%%%%%%%%%%%%%
In this section we state some results relating homomorphisms of word graphs
and containment of the associated congruences.

\begin{lemma}\label{unique-homomorphism}
  If $\Gamma_0 = (V_0, E_0)$ and $\Gamma_1= (V_1, E_1)$ are deterministic word
  graphs on the same alphabet $A$ such that every node in $\Gamma_0$ and
  $\Gamma_1$ is reachable from $0\in V_0$  and $0 \in
  V_1$, respectively, then there exists at most one word graph homomorphism
  $\phi: \Gamma_0\to \Gamma_1$ such that $(0)\phi = 0$.
\end{lemma}

\begin{proof}
  Suppose that $\phi_0, \phi_1: \Gamma_0\to \Gamma_1$ are word graph
  homomorphisms such that $(0)\phi_0 = (0)\phi_1 = 0$ and $\phi_0 \neq \phi_1$.
  Then there exists a node
  $\alpha\in V_0$ such that $(\alpha)\phi_0\neq (\alpha)\phi_1$. If $w$ is any
  word labelling a $(0, \alpha)$-path, then, since word graph homomorphisms
  preserve paths, $w$ labels $(0, (\alpha)\phi_0)$- and $(0,
  (\alpha)\phi_1)$-paths in $\Gamma_1$, which contradicts the assumption that
  $\Gamma_1$ is deterministic.
\end{proof}

Suppose that $\Gamma_0 = (V_0, E_0)$ and $\Gamma_1= (V_1, E_1)$ are word
graphs. At various points in the paper it will be useful to consider the
\defn{disjoint union} $\Gamma_0\sqcup \Gamma_1$ of $\Gamma_0$ and $\Gamma_1$,
which has set of nodes, after appropriate relabelling, $V_0\sqcup V_1$ and
edges $E_0\sqcup E_1$. If $\alpha\in V_i$, then we will write
$\alpha_{\Gamma_i}\in V_0\sqcup V_1$ if it is necessary to distinguish which
graph the node belongs to. We will also abuse notation by assuming that $V_0,
V_1 \subseteq V_0\sqcup V_1$ and $E_0, E_1\subseteq E_0\sqcup E_1$.

\begin{cor}\label{corollary-homomorphism}
  If $\Gamma_0 = (V_0, E_0)$, $\Gamma_1= (V_1, E_1)$, and $\Gamma_2 = (V_2,
  E_2)$ are  word graphs over the same alphabet, and every node in each word
  graph $\Gamma_i$ is reachable from $0_{\Gamma_i}\in V_i$ for $i\in \{0, 1,
  2\}$, then there is at most one word graph homomorphism  $\phi:
  \Gamma_0\sqcup \Gamma_1\to \Gamma_2$ such that $(0_{\Gamma_0})\phi =
  (0_{\Gamma_1})\phi = 0_{\Gamma_2}$.
\end{cor}
\begin{proof}
  Suppose that $\phi_0, \phi_1: \Gamma_0\sqcup \Gamma_1 \to \Gamma_2$ are
  word graph homomorphisms such that $(0_{\Gamma_0})\phi = (0_{\Gamma_1})\phi
  = 0_{\Gamma_2}$. Then $\phi_i|_{\Gamma_j}: \Gamma_j \to \Gamma_2$ is a word
  graph homomorphism for $i = 0, 1$ and $j = 0, 1$. Hence, by
  \cref{unique-homomorphism}, $\phi_0|_{\Gamma_0} = \phi_1|_{\Gamma_0}$ and
  $\phi_0|_{\Gamma_1} = \phi_1|_{\Gamma_1}$, and so $\phi_0 = \phi_1$.
\end{proof}

The final lemma in this subsection relates word graph homomorphisms and
containment of the associated congruences.

\begin{lemma}\label{lemma-kernel-right-invariant}
  Let $\Gamma_0 = (V_0, E_0)$ and $\Gamma_1 = (V_1, E_1)$ be word graphs over
  an alphabet $A$ representing congruences $\rho$ and $\sigma$ on a monoid $M =
  \genset{A}$. Then there exists a word graph homomorphism $\phi : \Gamma_0 \to
  \Gamma_1$ such that $(0_{\Gamma_0})\phi = 0_{\Gamma_1}$ if and only if $\rho
  \subseteq \sigma$.
\end{lemma}
\begin{proof}
  ($\Rightarrow$)
  If $\phi : \Gamma_0 \to \Gamma_1$ is a word graph homomorphism such that
  $(0_{\Gamma_0})\phi = 0_{\Gamma_1}$, and $w\in A ^*$
  labels a $(0, \alpha)$-path in $\Gamma_0$, then $w$ labels a $(0,
  (\alpha)\phi)$-path in $\Gamma_1$. If $(x, y) \in \rho$, then $x$
  and $y$ label
  $(0, \alpha)$-paths in $\Gamma_0$ and so $x$ and $y$ label $(0,
  (\alpha)\phi)$-paths in $\Gamma_1$. Hence $(x, y) \in \sigma$ and so  $\rho
  \subseteq \sigma$.

  ($\Leftarrow$) Conversely, assume that $\rho \subseteq \sigma$. We
  define $\phi
  : \Gamma_0 \to \Gamma_1$ as follows. For every $\alpha\in V_0$, we fix a
  word $v_{\alpha}\in A ^ *$ labelling a $(0, \alpha)$-path. Such a path exists
  for every $\alpha\in V_0$ because every node is reachable from $0$. We define
  $(\alpha)\phi = 0\cdot_{\Gamma_1} v_{\alpha}$
  for all $\alpha \in V_0$. If $(\alpha, a, \beta)\in
  E_0$, then $v_{\alpha}a$ and $v_{\beta}$ both label $(0, \beta)$-paths, and so
  $(v_{\alpha}a, v_{\beta})\in \rho$. Since $\rho \subseteq \sigma$, it follows
  that $(v_{\alpha}a, v_{\beta})\in \sigma$ and so $v_{\alpha}a$ and
  $v_{\beta}$ both label $(0, (\beta)\phi)$-paths in $\Gamma_1$. In particular,
  $((\alpha)\phi, a, (\beta)\phi)$ is an edge of $\Gamma_1$, and $\phi$ is a
  homomorphism.
\end{proof}

%%%%%%%%%%%%%%%%%%%%%%%%%%%%%%%%%%%%%%%%%%%%%%%%%%%%%%%%%%%%%%%%%%%%%%%%%%%
\subsection{Standard word graphs}%
\label{section-standard-word-graphs}
%%%%%%%%%%%%%%%%%%%%%%%%%%%%%%%%%%%%%%%%%%%%%%%%%%%%%%%%%%%%%%%%%%%%%%%%%%%

In this section we prove some essential results about standard word graphs.

Suppose that $\Gamma = (V, E)$ is a word graph where $V = \{0, \ldots, m - 1\}$
for some $n\in \N$ and $E\subseteq V\times A \times V$ where $A = \{a_0 <
\cdots < a_{k-1}\}$. The ordering on $V$ and $A$ induces the lexicographic
ordering on $V\times A$ and $V\times A\times V$, and the latter orders the
edges of any word graph $\Gamma = (V, E)$. We write $<$ for any and all of
these orders.

We defined standard complete word graphs in \cref{de-standard}, we
now extend this
definition to incomplete word graphs $\Gamma = (V, E)$ by adding the following
requirement:
\begin{enumerate}[(i)]
    \addtocounter{enumi}{3}
  \item\label{de-standard-incomplete}
    if $(\alpha, a)\in V\times A$ is a missing edge in $\Gamma$, $(\alpha, a) <
    (\beta, b)$ for some $(\beta, b)\in V\times A$, and
    there exists $\gamma\in V$ such that $(\beta, b, \gamma)\in E$, then
    $w_\gamma \neq w_\beta b$.
\end{enumerate}

An edge $(\alpha, a, \beta)\in E$ is a \defn{short-lex defining edge} in
$\Gamma$ if $w_\beta = w_\alpha a$.

\begin{lemma}\label{prop-short-lex-path-properties}
  Let $\Gamma = (V, E)$ be a deterministic word graph over an alphabet $A$,
  let $\alpha, \beta \in V$, and let
  \[
    (0, b_0, \alpha_1) = (\alpha_0, b_0, \alpha_1),\ldots, (\alpha_{n-1},
    b_{n-1}, \alpha_n)= (\alpha_{n-1}, b_{n-1}, \alpha)
  \]
  be an $(\alpha, \beta)$-path labelled by $w = b_0\cdots b_{n-1} \in A ^*$. If
  $w$ is the short-lex least word labelling any $(\alpha, \beta)$-path in
  $\Gamma$, then the following hold:
  \begin{enumerate}[\rm (i)]
    \item
      $\alpha_i = \alpha_j$ if and only if $i=j$, hence the $(\alpha,
      \beta)$-path labelled by $w$ does not contain duplicate edges;

    \item
      $b_i\cdots b_{j-1}$ is the short-lex least word labelling any $(\alpha_i,
      \alpha_j)$-path for all $i, j \in \{0, \ldots, n\}$ such that $i < j$.
  \end{enumerate}
\end{lemma}

\begin{lemma}\label{prop-standard-short-lex-path-properties}
  Let $\Gamma = (V, E)$ be a standard word graph over the alphabet
  $A$, let $\alpha \in V$, and let
  \[
    (0, b_0, \alpha_1)=(\alpha_0, b_0, \alpha_1),\ldots, (\alpha_{n-1}, b_{n-1},
    \alpha_n)=(\alpha_{n - 1}, b_{n - 1}, \alpha)
  \]
  be a $(0, \alpha)$-path labelled by $w = b_0\cdots b_{n-1}\in A ^ *$. If $w$
  is the short-lex least word labelling any $(0, \alpha)$-path in
  $\Gamma$, then:
  \begin{enumerate}[\rm (i)]
    \item every edge $(\alpha_i, b_i, \alpha_{i+1})$ is a short-lex defining
      edge;
    \item $\alpha_0 < \alpha_1 < \cdots < \alpha_n$.
  \end{enumerate}
\end{lemma}
\begin{proof}
  For (i), let $w_{\alpha_i}$ be the short-lex least word labelling a $(0,
  \alpha_i)$-path for each $i\in \{0, \ldots, n\}$. By
  \cref{prop-short-lex-path-properties}(ii), $w_{\alpha_i} = b_0\cdots
  b_{i-1}$. Hence $w_{\alpha_{i+1}} = w_{\alpha_i}b_i$ and so $(\alpha_i, b_i,
  \alpha_{i+1})$ is a short-lex defining edge as required.

  For (ii), since $|w_{\alpha_i}| < |w_{\alpha_{i+1}}|$, it follows that
  $w_{\alpha_i} < w_{\alpha_{i + 1}}$ and so by \cref{de-standard}(iii),
  $\alpha_i < \alpha_{i+1}$ for every $i$.
\end{proof}

If $\Gamma = (V, E)$ is an incomplete word graph over $A$, then we refer to
$(\alpha, a)\in V \times A$ as a \defn{missing edge} if $(\alpha, a,
\beta)\not\in E$ for all $\beta\in V$.  Recall that the missing edges
are ordered lexicographically according to the orders on $V$ and $A$
as defined at the start of \cref{section-standard-word-graphs}.

\begin{lemma}%
  \label{lemma-minimum-missing-edge-maximum-word}
  Let $\Gamma = (V, E)$ be an incomplete standard word graph over the
  alphabet $A$, and let
  $(\alpha, a)\in V \times A$ be a missing edge. Then $w_\gamma < w_\alpha
  a$ for all $\gamma \in V$.
\end{lemma}
\begin{proof}
  If $\gamma = 0$, then $w_\gamma = \varepsilon < w_{\alpha}a$ as required.
  Assume that $\gamma> 0$. Then $|w_{\gamma}| \geq 1$ and, in particular, there
  is at least one edge on the $(0, \gamma)$-path labelled by $w_\gamma$. Let
  $(\beta, b, \gamma)$ be the last such edge. By
  \cref{prop-short-lex-path-properties}(ii), $w_\gamma = w_\beta b$, and, by
  \cref{de-standard}(iv), $(\beta, b)\leq (\alpha, a)$. If $(\beta, b) =
  (\alpha, a)$, then $(\alpha, a)$ is not a missing edge. Hence $(\beta, b)
  < (\alpha, a)$, and so, either $\beta < \alpha$; or $\alpha = \beta$ and $a <
  b$.
  If $\beta < \alpha$, then, by \cref{de-standard}(iii), $w_\beta < w_\alpha$
  and so $w_\gamma = w_\beta b < w_\alpha a$, as required. If $\beta = \alpha$
  and $b < a$, then  $w_{\gamma} = w_{\beta}b = w_\alpha b < w_\alpha a$.
\end{proof}

\begin{lemma}\label{lemma-label-preservation}
  Let $\Gamma = (V, E)$ be an incomplete standard word graph over the
  alphabet $A$, let $(\alpha,
  a)$ be a missing edge, let $\beta \in V$, and let $e=(\alpha, a,
  \beta)$. If $\Gamma^\prime = (V, E\cup \{e\})$, then
  the following hold:
  \begin{enumerate}[\rm(i)]
    \item
      $w_\gamma = w_\gamma^\prime$ for all $\gamma\in V$ where $w_\gamma^\prime$
      is the short-lex least word labelling any $(0, \gamma)$-path in
      $\Gamma^\prime$;
    \item
      \label{lemma-standard-adding-internal-edge-preserved}
      $\Gamma'$ is standard.
  \end{enumerate}
\end{lemma}
\begin{proof}
  \textbf{(i).}
  Since every path in $\Gamma$ is also a path in $\Gamma^\prime$,
  $w^\prime_\gamma \leq w_\gamma$ for all $\gamma\in V$. Let $\gamma\in V$ be
  arbitrary. If the $(0, \gamma)$-path labelled by $w^\prime_\gamma$ in
  $\Gamma^\prime$ does not contain the newly added edge $e$, then
  $w^\prime_\gamma$ also labels a $(0, \gamma)$-path in $\Gamma$ and so
  $w_\gamma = w^\prime_\gamma$.

  Seeking a contradiction suppose that $(0, \gamma)$-path labelled by
  $w^\prime_\gamma$ in $\Gamma$ contains the newly added edge $e$. Then we can
  write $w^\prime_\gamma = uav$ where $u$ labels a $(0, \alpha)$-path and $v$
  labels a $(\beta, \gamma)$-path in $\Gamma^\prime$. By
  \cref{prop-short-lex-path-properties}(ii), $u$ labels the short-lex least
  $(0, \alpha)$-path in $\Gamma^\prime$ and so $u = w^\prime_\alpha$.
  Since $(\alpha, a)$ was a missing edge, $\Gamma^\prime$ is
  deterministic, so by
  \cref{prop-short-lex-path-properties}(i) in $\Gamma^\prime$, the
  $(0, \alpha)$-path labelled by
  $u$ does not contain the edge $e$. Hence $w^\prime_\alpha$ also labels a $(0,
  \alpha)$-path in $\Gamma$ and so $w^\prime_\alpha = w_\alpha$. By
  \cref{prop-short-lex-path-properties}(ii) again, $w_\beta^\prime = w_\alpha
  a$. But now by \cref{lemma-minimum-missing-edge-maximum-word} in $\Gamma$,
  $w_\beta^\prime = w_\alpha a > w_\beta$, which is a contradiction.

  Hence the $(0, \gamma)$-path labelled by $w^\prime_\gamma$ in $\Gamma$ does
  not contain $e$ for any $\gamma \in V$, and so $w_\gamma = w^\prime_\gamma$.

  \textbf{(ii).}
  Clearly, parts (i) and (ii) of \cref{de-standard} hold in $\Gamma'$. Part
  (iii) of \cref{de-standard} holds by part (i) of this lemma.

  To show that \cref{de-standard}(iv) holds, suppose that $(\gamma, c)$ is a
  missing edge of $\Gamma^\prime$ and $(\delta, d, \zeta)\in E\cup \{e\}$ for
  some $\gamma, \delta, \zeta \in V$ and $c, d \in A$ such that $(\gamma, c) <
  (\delta, d)$. Note that every missing edge of $\Gamma'$ is a missing edge of
  $\Gamma$ also. In particular, $(\gamma, c)$ is a missing edge of $\Gamma$.

  If $(\delta, d, \zeta)\neq e$, then, by \cref{de-standard}(iv) applied to
  $\Gamma$,  $w_{\zeta} \neq w_{\delta}d$. But part (i) implies that
  $w_{\zeta}' = w_{\zeta}$ and $w_{\delta}' = w_{\delta}$, and so $w_{\zeta}'
  \neq w_{\delta}'d$ also.  In particular, $\Gamma'$ satisfies
  \cref{de-standard}(iv) in this case.

  Applying \cref{lemma-minimum-missing-edge-maximum-word} to $\Gamma$ and
  $(\alpha, a)$ yields $w_\beta < w_\alpha a$. So, if $(\delta, d, \zeta) = e =
  (\alpha, a, \beta)$, then $w_{\zeta} = w_{\beta} < w_{\alpha}a =
  w_{\delta}d$. Applying part (i) as in the previous case implies that
  $w_{\zeta}' < w_{\delta}'d$.
\end{proof}

The final lemma in this section shows that the analogue of
\cref{lemma-label-preservation}(ii) holds when the target of the missing edge
is defined to be a new node.

% Let $\Gamma = (V, E)$ be a standard word graph. We say that a pair $(\alpha,
% a)\in V\times A$ is the \emph{minimum missing edge} if $(\alpha, a)$ is the
% lexicographically least missing edge of $\Gamma$. Note that the
% minimum missing
% edge is unique if it exists.

\begin{lemma}\label{lemma-standard-plus-min-missing-is-standard}
  Let $\Gamma = (V, E)$ be an incomplete standard word graph over the
  alphabet $A$, let $(\alpha,
  a)\in V \times A$ be the shortlex least missing edge in $\Gamma$,
  and let $e = (\alpha, a, |V|)$. Then $\Gamma^\prime = (V\cup
  \{|V|\}, E \cup \{e\})$ is standard.
\end{lemma}
\begin{proof}
  The word graph $\Gamma'$ is deterministic since $(\alpha, a)$ is a missing
  edge in $\Gamma$. Hence $\Gamma'$ satisfies \cref{de-standard}(i). Since
  $\Gamma$ is standard, every node in $V$ is reachable from $0$ in $\Gamma$ and
  $\Gamma'$. In particular, $\alpha$ is reachable from $0$ in $\Gamma'$ and
  hence so too is $|V|$. Thus \cref{de-standard}(ii) holds for $\Gamma'$.

  We set $V' = V \cup \{|V|\}$ and $E' = E\cup \{e\}$, so that $\Gamma' = (V',
  E')$. As in \cref{lemma-label-preservation}(i), for every $\gamma\in
  V^\prime$, we denote by $w^\prime_\gamma$ the short-lex least word labelling
  any $(0, \gamma)$-path in $\Gamma^\prime$. There are no edges in $E^\prime$
  with source $|V|$ and $|V|\not\in V$, and so if $\gamma \in V$, then no $(0,
  \gamma)$-path in $\Gamma^\prime$ contains the newly added edge $e$. Therefore
  $w_\gamma = w^\prime_\gamma$ for each $\gamma\in V$. Since $(\alpha, a,
  |V|)$ is the only edge with target $|V|$, and from
  \cref{prop-short-lex-path-properties}(ii), $w^\prime_{|V|} = w_{\alpha}'a =
  w_\alpha a$. But then \cref{lemma-minimum-missing-edge-maximum-word} implies
  that $w_\gamma' = w_{\gamma} < w_\alpha a = w_{|V|}'$ for all $\gamma\in V$.
  So, if $\gamma, \delta \in V'$ and $\gamma < \delta$, then either $\gamma,
  \delta\in V$ or $\delta = |V|$. In both cases, $w_{\gamma}' <
  w_{\delta}'$ and so \cref{de-standard}(iii) holds for $\Gamma'$.

  To establish that \cref{de-standard}(iv) holds for $\Gamma'$, suppose that
  $(\gamma, c)$ is a missing edge of $\Gamma^\prime$ and $(\delta, d, \zeta)\in
  E^\prime$ for some $\gamma, \delta, \zeta \in V$ and $c, d \in A$
  with $(\gamma,
  c) < (\delta, d)$.
  There are three cases to consider:
  \begin{enumerate}

    \item
      $(\gamma, c)$ is a missing edge in $\Gamma$ and $(\delta, d, \zeta)\in E$;

    \item
      $(\gamma, c)$ is a missing edge in $\Gamma$ and $(\delta, d,
      \zeta)\notin E$;

    \item
      $(\gamma, c)$ is not a missing edge in $\Gamma$.
  \end{enumerate}

  If (a) holds, then $w_{\zeta} \neq w_{\delta}d$ by \cref{de-standard}(iv)
  applied to $\Gamma$. But \cref{lemma-label-preservation} implies that
  $w_{\zeta}' = w_{\zeta}$ and $w_{\delta}'\neq w_{\delta}$ and so $w_{\zeta}'
  \neq w_{\delta}'d$ in this case.

  We conclude the proof by showing that neither (b) nor (c) holds.

  If (b) holds, then $(\delta, d, \zeta) = e= (\alpha, a, |V|)$. But $(\alpha,
  a)$ is the least missing edge of $\Gamma$, so $(\delta, d) =
  (\alpha, a) \leq (\gamma, c)$,
  which contradicts the assumption that $(\gamma, c) < (\delta, d)$. Hence (b)
  does not hold.

  If (c) holds, then $\gamma = |V|$. Since there are no edges with source $|V|$
  in $\Gamma^\prime$, it follows that $\delta \in V$ and so $\delta < |V|$.
  But then $(\gamma, c) = (|V|, c) > (\delta, d)$, which again contradicts the
  assumption that $(\delta, d) > (\gamma, c)$.
\end{proof}

Combining \cref{lemma-label-preservation}(ii) and
\cref{lemma-standard-plus-min-missing-is-standard} we obtain the following
corollary.

\begin{cor}
  Let $\Gamma = (V, E)$ be an incomplete standard word graph over the
  alphabet $A$, let $(\alpha, a)$
  be the short-lex least missing edge in $\Gamma$, and let $\beta \in V\cup
  \{|V|\}$. Then $\Gamma^\prime = (V\cup \{\beta\}, E\cup\{(\alpha,
  a, \beta)\})$
  is also standard.
\end{cor}

\section{Algorithm 1: the low-index right congruences
algorithm}\label{section-low-index}

Throughout this section we suppose that: $\langle A \mid R\rangle$ is a fixed
finite monoid presentation defining a monoid $M$; and that $n\in \mathbb{Z} ^
+$ is fixed. The purpose of this section is to describe a procedure for
iterating through the right congruences of $M$ with at most $n$ congruence
classes.
This procedure is based on the Todd--Coxeter Algorithm (see for
example~\cite{Coleman2022aa},~\cite{Jura1978aa} or~\cite{Todd1936aa})
and is related to
Sims' ``low-index'' algorithm for computing subgroups of finitely presented
groups described in Chapter 5.6 of~\cite{Sims1994aa}.

As shown in \cref{thm-graph-to-right-congruence}, there is a bijective
correspondence between complete standard word graphs with at most $n$ nodes
that are compatible with $R$, and the right
congruences of $M$ with index at most $n$. As we hope to demonstrate, the key
advantage of this correspondence is that word graphs are inherently
combinatorial objects which lend themselves nicely to various enumeration
methods. The algorithm we describe in this section is a more or less classical
backtracking algorithm, or depth-first search.

This section is organised as follows. We begin with a brief general description
of backtracking algorithms and refining functions
in~\cref{subsection-backtrack}. For
a more detailed overview of backtracking search methods
see~\cite[Section 7.2.2]{Knuth2022aa}. In
\cref{subsection-specific-search-tree} we construct a specific search
tree for the
problem at hand, whose nodes are standard word graphs; in
\cref{subsec-refining-functions} we introduce various refining
functions that improve the
performance of finding right congruences.

%%%%%%%%%%%%%%%%%%%%%%%%%%%%%%%%%%%%%%%%%%%%%%%%%%%%%%%%%%%%%%%%%%%%%%%%
\subsection{Backtracking search and refining
functions}\label{subsection-backtrack}
%%%%%%%%%%%%%%%%%%%%%%%%%%%%%%%%%%%%%%%%%%%%%%%%%%%%%%%%%%%%%%%%%%%%%%%%

We start with the definition of the search space where we are performing the
backtracking search. To do so we require the notion of a digraph $\T = (\V, \E)$
where $\V$ is the set of nodes, and $\E\subseteq \V\times \V$ is the set of
edges. We denote such digraphs using blackboard fonts to distinguish them from
the word graphs defined above.

A \defn{multitree} is a digraph $\T = (\V, \E)$ such that for all $\vv, \w\in
\V$ there exists at most one directed path from $\vv$ to $\w$ in $\T$. For a
node $\vv\in \V$ we write $\Reach(\T, \vv)$ to denote the set of all
nodes reachable from $\vv$ in $\T$.

Given a multitree $\T = (\V, \E)$ and a (possibly infinite) set $\X\subseteq
\V$, we say that $\T$ is a \defn{search multitree for $\X$} if there exists
$\vv\in \V$ such that $\X\subseteq \Reach(\vv)$. We refer to any such node $\vv$
as a \defn{root node} of $(\T, \X)$. Recall that the symbol $\bot$ is used to
mean ``undefined'' and does not belong to $\V$.

\begin{de}\label{de-refining-func}
  A function $f:
  \V\cup\{\bot\}\to \V\cup\{\bot\}$ is a \defn{refining function for
  $\X$} if the
  following hold for all $\vv\in \V$:
  \begin{enumerate}[(i)]
    \item $f(\bot) = \bot$,
    \item if $f(\vv) = \bot$, then $\Reach(\vv)\cap \X = \varnothing$,
    \item if $f(\vv) \neq \bot$, then $\Reach(\vv) \cap \X =
      \Reach(f(\vv))\cap \X$.
  \end{enumerate}
\end{de}
For example, the identity function $\id: \V\cup\{\bot\}\to \V\cup\{\bot\}$ is a
refining function for subset of $\V$.

If $\vv \in \V$, then we refer to the set
\[
  \Kids(\vv) =\{\w\in \V : (\vv, \w)\in\E\}
\]
as the \defn{children} of $\vv$ in $\T = (\V, \E)$. Given algorithms for
computing the refining function $f$, testing membership in $\X$, and
determining the children $\Kids(\vv)$ of any $\vv\in \V$, the backtracking
algorithm $\texttt{BacktrackingSearch}_\X(f, \vv)$ outputs the set
$\Reach(\vv)\cap \X$ for any $\vv\in \V$. As a consequence
$\texttt{BacktrackingSearch}_\X(f, \vv) = \X$ if $\vv$ is any root node of $(\T,
\X)$. Pseudocode for the algorithm $\texttt{BacktrackingSearch}_\X$ is given
in~\cref{alg-backtrack}.

\begin{algorithm}
  \caption{- $\texttt{BacktrackingSearch}_{\X}$}\label{alg-backtrack}
  \textbf{Input:} A refining function $f:\V\cup \{\bot\} \to \V\cup
  \{\bot\}$ for $\X$ and a node $\vv\in \V$. \\
  \textbf{Output:} $\Reach(\vv)\cap \X$.
  \begin{algorithmic}[1]
    \State $S\gets \varnothing$
    \State $\vv\gets f(\vv)$
    \If{$\vv \not = \bot$}
    \If{$\vv\in \X$}
    \State $S\gets S \cup \{\vv\}$
    \EndIf
    \For{$\w\in \Kids(\vv)$}
    \State $S\gets S \cup \texttt{BacktrackingSearch}_{\X}(f, \w)$
    \Comment{Due to the multitree property, this is a disjoint union.}
    \EndFor
    \EndIf
    \State \Return $S$
  \end{algorithmic}
\end{algorithm}
Of course, if $\X$ is infinite, then $\texttt{BacktrackingSearch}_\X$ will not
terminate. In practice, if $\X$ is finite but $\T$ is infinite, some care is
required when choosing a refining function $f:\V \cup\{\bot\} \to \V\cup
\{\bot\}$ to ensure that $\texttt{BacktrackingSearch}_\X(f, \vv)$ terminates.
On the other hand if $f(\vv) \neq \bot$ for only finitely many $\vv \in \V$,
then clearly, $\texttt{BacktrackingSearch}_\X$ will terminate.

$\texttt{BacktrackingSearch}_\X$
can be modified to simply count the number of elements in $\X$, to apply any
function to each element of $\X$ as it is discovered, or to search for an
element of $\X$ with a particular property by modifying line 5 and 8.

The following properties of search multitrees and refining functions will be
useful later:
\begin{prop}\label{prop-refiner-properties}
  Let $\T = (\V, \E)$ be a search multitree for $\X\subseteq \V$. Then
  \begin{enumerate}[\rm (i)]
    \item
      If $\mathbb{Y}\subseteq \X$, then $\T$ is also a search multitree
      for $\mathbb{Y}$;

    \item
      If $\mathbb{Y}, \mathbb{Z} \subseteq\X$ and $f_{\mathbb{Y}}$ and
      $f_{\mathbb{Z}}$ are refining functions for $\mathbb{Y}$ and $\mathbb{Z}$
      respectively, then $f_{\mathbb{Y}}\circ f_{\mathbb{Z}}$ is a
      refining function for
      $\mathbb{Y}\cap \mathbb{Z}$.
  \end{enumerate}
\end{prop}

\subsection{The search multitree of standard word
graphs}\label{subsection-specific-search-tree}

In this section, we describe the specific search multitree $\T = (\V, \E)$
required for the low-index congruences algorithm.

We define $\V$ to be the set of all standard word graphs over a fixed finite
alphabet $A$ and we define $\E\subseteq \V\times \V$ to consist of the edges
$(\Gamma, \Gamma^\prime)\in\E$ if and only if $\Gamma = (V, E)$, $\beta \in
V\cup \{|V|\}$, $(\alpha, a)$ is the short-lex least missing edge in $\Gamma$,
and $\Gamma^\prime = (V\cup \{\beta\}, E\cup \{(\alpha, a, \beta)\})$. Since
every word graph over $A$ is finite by definition, the set $\V$ is
countably infinite.

We write $\Gamma=(V, E)\subseteq\Gamma' = (V', E')$
if $V\subseteq V'$ and $E\subseteq E'$.

\begin{lemma}\label{prop-specific-multitree}
  The digraph $\T$ is a multitree.
\end{lemma}
\begin{proof}
  Suppose that $\Gamma_0, \ldots, \Gamma_m$ and $\Gamma_0^\prime, \ldots,
  \Gamma_n^\prime$  are paths in $\T$ such that $\Gamma_0 = \Gamma_0'$ and
  $\Gamma_m = \Gamma_n'$. From the definition of $\T$ it
  follows that $\Gamma_0\subseteq \cdots\subseteq \Gamma_m$ and
  $\Gamma_0^\prime\subseteq \cdots\subseteq \Gamma_n^\prime$. Seeking a
  contradiction suppose that $i\in\N$ is the least value such that
  $\Gamma_{i+1} \neq \Gamma_{i+1}'$. If $(\alpha, a)$ is the least missing
  edge in $\Gamma_i = (V_i, E_i) = \Gamma_i'$, then $\Gamma_{i+1} = (V_i\cup
  \{\beta\}, E_i\cup \{(\alpha, a, \beta)\})$ and $\Gamma_{i + 1}' = (V_i\cup
  \{\beta'\}, E_i\cup \{(\alpha, a, \beta')\})$ for some $\beta, \beta'\in
  V_i\cup \{|V_i|\}$ with $\beta \neq \beta'$.
  It follows that $(\alpha, a, \beta), (\alpha, a, \beta') \in \Gamma_m =
  \Gamma_n'$, and so $\Gamma_m$ is not deterministic, and hence not standard,
  which is a contradiction.
\end{proof}

We denote the set of all complete standard word graphs over $A$ by $\X$. Note
that, by \cref{thm-graph-to-right-congruence}, the word graphs in $\X$ are in
bijective correspondence with the right congruences of the free monoid $A ^ *$;
see~\cref{appendix-numbers}. We do not use this correspondence explicitly.

We
will now show that every $\Gamma\in \X$ is reachable from the trivial word
graph $\Xi = (\{0\}, \varnothing)$ in $\T$, so that $\T$ is a search
multitree for $\X$.

\begin{lemma}\label{prop-specific-search-tree}
  Let $\Gamma=(V, E)\in \X$ be any complete standard word graph over $A$. Then
  there exists a sequence of standard word graphs
  \[
    \Xi = \Gamma_0, \Gamma_1, \ldots, \Gamma_n = \Gamma
  \]
  such that $(\Gamma_i, \Gamma_{i + 1})$ is an edge in $\T$ for every $i\in
  \{0, \ldots, n - 1\}$. In particular, $\X \subseteq \Reach(\Xi)$, and so
  $\T$ is a search multitree for $\X$.
\end{lemma}
\begin{proof}
  Suppose that we have defined $\Gamma_0, \ldots, \Gamma_i$ for some $0 \leq i
  < n$ such that $\Gamma_0\subsetneq \cdots \subsetneq \Gamma_i
  \subseteq \Gamma$
  and $(\Gamma_j, \Gamma_{j + 1})$ is an edge in $\T$ for every $j\in
  \{0, \ldots, i - 1\}$.
  Let $(\alpha, a)\in V \times A$ be the least missing edge in $\Gamma_i =
  (V_i, E_i)$.
  Since $\Gamma$ is deterministic, there exists $\beta\in V$ such that
  $(\alpha, a, \beta) \in E$. We define $\Gamma_{i + 1} = (V_i\cup \{\beta\},
  E_i\cup \{(\alpha, a, \beta)\})$. Clearly, $(\Gamma_i, \Gamma_{i + 1})\in \E$
  by definition and $\Gamma_i \subsetneq \Gamma_{i + 1}\subseteq \Gamma$.

  Since $\Gamma$ is finite, and the $\Gamma_i$ form a strictly increasing
  sequence of subsets of $\Gamma$, it follows that the sequence of $\Gamma_i$
  is finite.
\end{proof}

If $\Gamma=(V, E)$ is an incomplete standard word graph with least missing edge
$(\alpha, a)\in V\times A$, then the children of $\Gamma$ in $\T$ are:
\[
  \Kids(\Gamma) = \set{(V\cup \{\beta\}, E\cup \{(\alpha, a, \beta)\})}{\beta\in
  \{0, \ldots, |V|\}}.
\]
Clearly, since $\Kids(\Gamma)$ is finite, it can be computed in linear time in
$|V|$. Also it is possible to check if $\Gamma$ is complete in constant time
by checking whether $|E| = |A||V|$. Hence we can check whether $\Gamma$
belongs to $\X$ in constant time.

We conclude this subsection with some comments about the implementational
issues related to $\texttt{BacktrackingSearch}_\X(f, \Gamma)$ for some refining
function $f$ of $\X$ and word graph $\Gamma\in \V$.
It might appear that to
iterate over $\Kids(\Gamma)$ in line 7 of \cref{alg-backtrack}, it is necessary
to copy $\Gamma$ with the appropriate edge added, so that the
recursive call to $\texttt{BacktrackingSearch}_\X$ in line 8 does not modify
$\Gamma$. However, this approach is extremely memory inefficient, requiring
memory proportional to the size of the search tree. This is especially bad when
$\texttt{BacktrackingSearch}_{\X}$ is used to count the word graphs satisfying
certain criteria or used to find a word graph satisfying a particular property.
We briefly outline how to iterate over $\Kids(\Gamma)$ by modifying $\Gamma$
inplace, which requires no extra memory (other than that needed to
store the additional edge). If the maximum number of nodes in any word graph
$\Gamma$ that will be encountered during the search is known beforehand, then
counting and random sampling of $\X$ within the search multitree $\T$ can be
performed with constant space.

To do this, the underlying datastructure used to store $\Gamma$ must support
the following operations: retrieving the total number of edges, adding an edge
(with a potentially new node as its target), removing the most recently added
edge (and any incident nodes that become isolated). It is possible to implement
a datastructure where each of these operations takes constant time and space,
this is the approach used in  \libsemigroups. Of course, the refining functions
$f$ may also modify $\Gamma$ inplace.

Given such a datastructure and refining functions, we can then perform
the loop in lines 7-9 of \cref{alg-backtrack} as follows:
\begin{enumerate}[(1)]

  \item
    Let $k = |E|$ be the total number of edges of $\Gamma$, let
    $(\alpha, a)$ be the least missing edge of $\Gamma$ and let $\beta = 0$.

  \item
    Add the edge $(\alpha, a, \beta)$ to $\Gamma$.

  \item
    Set $S\gets S\cup \texttt{BacktrackingSearch}_\X(f, \Gamma)$,
    noting that the recursive call takes the modified $\Gamma$ as input.

  \item
    If $|E| > k$, then repeatedly remove the most recently added edge
    from $\Gamma$ until $|E| = k$.

  \item
    Increment $\beta$. If $\beta > |V|$, then terminate. Otherwise go to Step 2.
\end{enumerate}
The word graph $\Gamma$ is equal to one of its children after the edge is added
in Step 2. After Step 4, $\Gamma$ is restored to its original state before
the recursive call was made. Note that we cannot just remove the last added
edge, as the refining function $f$ may have added extra edges to $\Gamma$ in
the recursive call, and these extra edges are not removed in the recursive call.

% JDM: I don't think we need the following paragraph...
% For the remainder of this section as well as for all of
% \cref{section-applications-of-low-index-algo} we assume the above modification
% of \texttt{Backtrackingsearch} for iterating over $\Kids(\Gamma)$ and that all
% refining functions which modify $\Gamma$ in place. However, for notational
% convenience and simplicity of proofs we will still specify the
% refining functions as if they are
% creating a new word graph $\Gamma^\prime$, where appropriate.

%%%%%%%%%%%%%%%%%%%%%%%%%%%%%%%%%%%%%%%%%%%%%%%%%%%%%%%%%%%%%%%%%%%%%%%%
\subsection{Refining functions for standard word graphs}%
\label{subsec-refining-functions}
%%%%%%%%%%%%%%%%%%%%%%%%%%%%%%%%%%%%%%%%%%%%%%%%%%%%%%%%%%%%%%%%%%%%%%%%

We denote the set of complete standard word graphs over $A$
\begin{itemize}
  \item with at most $n$ nodes by $\X_n$;
  \item compatible with $R\subseteq A ^ * \times A ^ *$ by $\X_R$.
\end{itemize}
By  \cref{thm-graph-to-right-congruence}, the word graphs in
$\X_{n}\cap \X_R$ are precisely the word graphs of
the right congruences of the monoid defined by $\langle A \mid R\rangle$ with
index at most $n$.

In this section we define the refining functions $\texttt{AtMost}_n$ and
$\texttt{IsCompatible}_R$ for $\X_n$ and $\X_R$, respectively. It follows from
\cref{prop-refiner-properties}(ii) that $\texttt{AtMost}_n\circ
\texttt{IsCompatible}_R$ is a refining function for $\X_{n}\cap
\X_R$.  We also define two further refining functions for $\X_R$ that try to
reduce the number of word graphs (or equivalently nodes in the search
multitree) visited by $\texttt{BacktrackingSearch}_{\X_R}$; we will
say more about this later.

The first refining function is
$\texttt{AtMost}_n:\V\cup\{\bot\}\to \V\cup\{\bot\}$ for any
$n\in \N$ defined by
\[
  \texttt{AtMost}_n(\Gamma) =
  \begin{cases}
    \Gamma & \text{if } \Gamma=(V,E) \text{ and } |V|\leq n\\
    \bot & \text{otherwise}
  \end{cases}
\]
for every $\Gamma\in \V\cup \{\bot\}$ (the set of standard word graphs over $A$
and $\bot$).  It is routine to verify that
$\texttt{AtMost}_n$  is a refining function for $\X_n$.

% \begin{lemma}\label{lemma-at-most-is-refiner}
% $\texttt{AtMost}_n$  is a refining function for $\X_n$.
% \end{lemma}
% \begin{proof}
% Clearly, $\texttt{AtMost}_n$ returns $\bot$ if $\Gamma = \bot$, and
% so \cref{de-refining-func}(i) holds.
%
%   If $\Gamma = (V, E)$ is a standard word graph (i.e.\ $\Gamma\in \V$) and
%   $\Gamma'=(V', E')$ is reachable in $\T$ from $\Gamma$, then $|V'|\geq |V|$.
%   Hence $\texttt{AtMost}_n(\Gamma) = \bot$ if and only if
% $\Reach(\Gamma)\cap \X_n =
%   \varnothing$. Hence \cref{de-refining-func}(ii) holds.
%
%   Finally, $\texttt{AtMost}_n(\Gamma) = \Gamma$ if
% $\texttt{AtMost}_n(\Gamma)\neq \bot$, hence
% \cref{de-refining-func}(iii) holds.
% \end{proof}

If $\Gamma=(V, E)$ is a standard word graph, then checking whether $\Gamma\in
\X$ (i.e. $\Gamma$ is complete) can be done in constant time, and checking that
$\Gamma\in \X_n$ (i.e. that $|V|\leq n$) also has constant time complexity.
Moreover, $\texttt{AtMost}_n(\Gamma) \neq \bot$ for only finitely many standard
word graphs $\Gamma$ over $A$, and thus $\texttt{BacktrackingSearch}_{\X_n}(
\texttt{AtMost}_n, \vv)$ terminates and outputs $\X_n$.

%and so $\texttt{BacktrackingSearch}(\X_n,
%\texttt{AtMost}_n\circ \texttt{IsCompatible}_R, \vv)$ returns
% $\X_{n}\cap \X_R$ also.

% We then focus on word graphs over the alphabet $A$ that are also compatible
% with a set of relations $R$, and give a multiple refiners for this set.

It is possible to check if a, not necessarily standard, word graph $\Gamma$
belongs to $\mathbb{X}_{n,R}$ in linear time in the length of the presentation
$\langle A \mid R\rangle$.  Again since $\X_{n}\cap \X_R\subseteq \X$, $\T$ is a
multisearch tree for $\X_{n}\cap \X_R$. This gives us an immediate, if rather
inefficient, method for computing all the right congruences with index at most
$n$ of a given finitely presented monoid: simply run
$\texttt{BacktrackingSearch}_{\X_n}(\texttt{AtMost}_n, \Xi)$ to
obtain $\X_n$, and
then check each word graph in $\X_n$ for compatibility with $R$.

This method would explore just as many nodes of the search tree $\mathbb{T}$
for the free monoid $\langle a, b \mid \rangle$ as it would for the trivial
monoid $\langle a, b \mid a = 1, b = 1\rangle$. On the other hand, when
considering the trivial monoid, as soon as we define an edge leading to a node
other than $0$, both of the relations $a = 1$ and $b = 1$ are violated and
hence there is no need to consider any of the descendants of the current node
in the multitree. So that we can take advantage of this observation, we define
the function $\texttt{IsCompatible}_{R}:\V\cup\{\bot\}\to
\V\cup\{\bot\}$ by
\begin{equation*}
  \texttt{IsCompatible}_{R}(\Gamma) =
  \begin{cases}
    \Gamma & \text{if } \Gamma \in \V \text{ and } \Gamma \text{ is
    compatible with } R\\
    \bot & \text{otherwise}
  \end{cases}
\end{equation*}

Recall from the definition, if $(u, v) \in R$, then  $\Gamma = (V, E)$ is
compatible with $(u, v)$ if $\alpha \cdot u = \alpha \cdot v$ for every
$\alpha\in V$ such that $\alpha\cdot u\neq \bot$ and $\alpha\cdot v\neq \bot$.
The non-existence of a path with source $\alpha$ labelled by $u$ or $v$,
however, does not make $\Gamma$ incompatible with $(u, v)$. So it
may be possible to extend $\Gamma$ so that $u$ and $v$ do label paths with
source $\alpha$ and common target $\beta$. Hence $\texttt{IsCompatible}_R$ does
not return $\bot$ just because some relation word does not label a path from
some node in $\Gamma$. It is straightforward to verify that
$\texttt{IsCompatible}_R$ is a refining function for $\mathbb{X}_R$.

% \begin{prop}\label{prop-is-comp-is-refiner}
%  $\texttt{IsCompatible}_R$ is a refining function for $\mathbb{X}_R$.
% \end{prop}
%
% \begin{proof}
%   Clearly, $\texttt{IsCompatible}_R(\bot) = \bot$  by definition, and so
%   \cref{de-refining-func}(i) holds.
%
%   Otherwise if $\Gamma = (V, E) \in \V$ is such that
%   $\texttt{IsCompatible}_R(\Gamma) = \bot$, then there exist $\alpha, \beta,
%   \gamma\in V$ and a relation $(u, v)\in R$ such that $u$ and $v$ label
%   $(\alpha, \beta)$- and $(\alpha, \gamma)$-paths, respectively, and $\beta
%   \neq \gamma$.  If $\Gamma'$ is any descendent of $\Gamma$ in $\T$, then
%   the paths labelled by $u$ and $v$ in $\Gamma$ are also paths in
% $\Gamma'$. In
%   particular, $\Gamma'$ is not compatible with $R$, and so $\Gamma'\not\in
%   \X_R$. In other words, $\Reach(\Gamma) \cap \X_R = \varnothing$, and so
%   \cref{de-refining-func}(ii) holds.
%
%   Finally, if $\texttt{IsCompatible}_R(\Gamma) \neq \bot$, then
% $\Gamma \in \V$
%   and $\texttt{IsCompatible}_R(\Gamma) = \Gamma$. In particular,
%   $\Reach(\texttt{IsCompatible}_R(\Gamma))\cap \X_R = \Reach(\Gamma) \cap
%   \X_R$ and \cref{de-refining-func}(iii) holds.
% \end{proof}

By \cref{prop-refiner-properties}(ii),
$\texttt{IsCompatible}_R\circ\texttt{AtMost}_n$ is a refining
function for $\X_{n}\cap \X_R$. Therefore the output of
$\texttt{BacktrackingSearch}_{\X_n\cap \X_R}(
  \texttt{IsCompatible}_R\circ\texttt{AtMost}_n,
\Xi)$ is $\X_n\cap \X_R$.
For comparison, the number of standard word graphs visited by
$\texttt{BacktrackingSearch}_{\X_n\cap \X_R}( \texttt{AtMost}_n,
\Xi)$ where $n = 4$, for
the $3$-generated plactic monoid (with the standard presentation,
see~\cite{Knuth1970aa,Lascoux1981aa}) is $3,556,169$. On the other
hand, the number for
$\texttt{BacktrackingSearch}_{\X_n\cap
\X_R}(\texttt{IsCompatible}_R\circ\texttt{AtMost}_n, \Xi)$ is
$29,800$.
This example illustrates that it can be significantly faster to check
for compatibility with $R$ at every node in the search multitree $\T$
rather than first finding $\X_n$ and then checking for compatibility.
The extra time spent per word graph (or node in the search multitree)
checking compatibility with $R$ is negligible in comparison to the
saving achieved in this example.

The refiner $\texttt{IsCompatible}_R$ can be improved. Consider the situation
where $(u, v)\in R$ with $v = v_1b$ for some word $v_1\in A^\ast$ and letter
$b\in A$. If there is a node $\alpha\in V$ such that $\alpha \cdot u \neq \bot$
and $\alpha\cdot v_1 \neq \bot$, but $\alpha\cdot v= \bot$, then $(\alpha\cdot
v_1, b)$ is a missing edge in $\Gamma$. There is only one choice $\alpha\cdot u$
for the target of this missing edge which will not break compatibility with $R$.
This situation is shown diagrammatically in \cref{fig-one-away-from-compatible}.

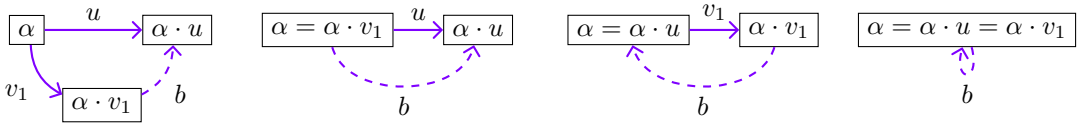
\begin{figure}[H]
  \centering
  \begin{tikzpicture}
    \begin{scope}
      \node[draw] (alpha) at (0,0) {$\alpha$};
      \node[draw] (beta) at (2,0)  {$\alpha \cdot u$};
      \node[draw] (gamma) at (1,-1) {$\alpha \cdot v_1$};
      \draw[style=a](alpha) to node [midway,above]{$u$} (beta);
      \draw[style=a, bend right](alpha) to node [midway,below
      left]{$v_1$} (gamma);
      \draw[dashed, style=a, bend right](gamma) to node [midway,below
      right]{$b$} (beta);
    \end{scope}
    \begin{scope}[xshift=4cm]
      \node[draw] (alpha) at (0,0) {$\alpha=\alpha\cdot v_1$};
      \node[draw] (beta) at (2,0) {$\alpha\cdot u$};
      \draw[style=a] (alpha) to node [midway,above]{$u$} (beta);
      \draw[dashed, style=a, bend right=75] (alpha) to node
      [midway,below]{$b$} (beta);
    \end{scope}
    \begin{scope}[xshift=8cm]
      \node[draw] (alpha) at (0,0) {$\alpha=\alpha\cdot u$};
      \node[draw] (gamma) at (2,0) {$\alpha \cdot v_1$};
      \draw[style=a] (alpha) to node [midway,above]{$v_1$} (gamma);
      \draw[dashed, style=a, bend left=75] (gamma) to node
      [midway,below]{$b$} (alpha);
    \end{scope}
    \begin{scope}[xshift=12.5cm]
      \node[draw] (alpha) at (0,0) {$\alpha=\alpha\cdot u=\alpha\cdot v_1$};
      \path [dashed, style=a, loop below](alpha) to node
      [midway,below]{$b$} (alpha);
    \end{scope}
  \end{tikzpicture}
  \caption{Illustration of a node $\alpha$ being one letter away from being
    compatible with the relation $(u, v_1b)$, including degenerate cases when
  one or both of $u, v_1$ are empty words.}
  \label{fig-one-away-from-compatible}
\end{figure}

So, if $\Gamma'\in \Reach(\Gamma)\cap \X_{n}\cap \X_R$, then $(\gamma, b,
\beta)$ must be an edge in $\Gamma'$. Of course, it is not guaranteed that
$(\gamma, b, \beta)$ is such that $(\gamma, b)$ is the least missing edge of
any descendent of $\Gamma$ in the search multitree $\T$. However by
\cref{lemma-label-preservation}\ref{lemma-standard-adding-internal-edge-preserved},
if $\Gamma = (V, E)$, then $\beta \in V$ and so $\Gamma^\prime = (V, E \cup
\{(\gamma, b, \beta)\})$ is standard.

We now improve the refining function $\texttt{IsCompatible}_R$ by adding the
ability to define edges along paths that are one letter away from fully
labelling a relation word in the manner described above. We denote this new
refining function for $\X_R$ by $\texttt{MakeCompatible}_R$ and define it in
\cref{alg-make-compatible}.

\begin{algorithm}
  \caption{- $\texttt{MakeCompatible}_R$}
  \label{alg-make-compatible}
  \textbf{Input:} A standard word graph $\Gamma = (V, E)\in \V$ or
  $\Gamma= \bot$.
  \\
  \textbf{Output:} A standard word graph $\Gamma'\in \Reach(\Gamma)$ or $\bot$.
  \begin{algorithmic}[1]
    \If{$\Gamma = \bot$}
    \State \Return $\bot$
    \EndIf
    \State $E' := E$
    \For{$\alpha \in V$ and $(u, v) \in R$}
    \If{$u=u_1a$, $\alpha\cdot u_1 = \beta \in V$, $\beta \cdot a = \bot$,
    and $\alpha \cdot v = \gamma\in V$}
    \State  $E'\gets E' \cup \{(\beta, a, \gamma)\}$
    \ElsIf{$v=v_1b$, $\alpha\cdot v_1 = \beta \in V$, $\beta \cdot b = \bot$,
    and $\alpha \cdot u = \gamma\in V$}
    \State  $E'\gets E' \cup \{(\beta, b, \gamma)\}$
    \ElsIf{$\alpha\cdot u \neq \alpha\cdot v$}
    \State \Return $\bot$
    \EndIf
    \EndFor
    \State \Return $\Gamma' = (V, E')$
  \end{algorithmic}
\end{algorithm}

\begin{lemma}
  $\texttt{MakeCompatible}_R$ is a refining function for $\X_R$.
\end{lemma}
\begin{proof}
  We verify the conditions in \cref{de-refining-func}.

  Clearly, $\texttt{MakeCompatible}_R(\bot) =\bot$ and so
  \cref{de-refining-func}(i) holds. Let $\Gamma = (V, E)\in \V$.
  If $\Gamma \neq \bot$, then the execution of the algorithm
  constructs a sequence of word graphs $\Gamma = \Gamma_0, \Gamma_1,
  \ldots, \Gamma_n =
  \Gamma^\prime$ such that $\Gamma_{i}$ is obtained from $\Gamma_{i-1}$ by
  adding an edge in line 7 or line 9 of \cref{alg-make-compatible}, where
  $\Gamma^\prime$ is the final word graph constructed before returning in either
  line 11 or 14. In the final iteration of the for loop in line 5, $\Gamma'$ is
  the output word graph, and so the conditions in lines 6 and 8 do not hold
  because no more edges are added to $\Gamma^\prime$. The condition
  in line 10 will only hold if
  $\Gamma^\prime$ is not compatible with $R$. Therefore
  $\texttt{MakeCompatible}_R(\Gamma)= \bot$ if and only if
  $\texttt{IsCompatible}_R(\Gamma^\prime)=\bot$ and
  $\texttt{MakeCompatible}_R(\Gamma)= \Gamma^\prime$ otherwise. If
  $\Reach(\Gamma)\cap \X_R = \Reach(\Gamma^\prime)\cap \X_R$, then
  \cref{de-refining-func}(ii) and (iii) both hold since
  $\texttt{IsCompatible}_R$ is a refining function for $\X_R$.

  It suffices to establish that $\Reach(\Gamma)\cap \X_R =
  \Reach(\Gamma^\prime)\cap \X_R$ when $n=1$. The claim can then be established
  for $n>1$ by straightforward induction.
  We may
  also assume without loss of generality $\Gamma'$ is obtained from $\Gamma$ in
  line 9 of \cref{alg-make-compatible}. The other case when an edge is added
  in line 7 of \cref{alg-make-compatible} is dual.

  If we add an edge to $\Gamma$ in line 9, then the condition in line 8 of
  \cref{alg-make-compatible} holds. Therefore, there exist $(u, v)\in R$,
  $v_1\in A^\ast$, $b\in A$, and $\alpha, \beta, \gamma \in V$ such that the
  following hold: $\alpha\cdot v = \bot$; $\alpha \cdot v_1 = \gamma$ for some
  $\gamma\in V$; $\alpha \cdot u = \beta$ for some $\beta \in V$; and
  $\Gamma^\prime = (V, E \cup \{(\gamma, b, \beta)\})$.

  We must show that $\Reach(\Gamma)\cap \X_R = \Reach(\Gamma^\prime)\cap \X_R$.
  Since $\Gamma^\prime$ contains $\Gamma$, it is clear that $\Reach(\Gamma)\cap
  \X_R \supseteq \Reach(\Gamma^\prime)\cap \X_R$.

  Let $\Gamma'' \in \Reach(\Gamma)\cap \X_R$.   It suffices to show that
  $\Gamma'' \in \Reach(\Gamma')$. Since $\Gamma''$ is complete, there exists
  $\delta \in V''$ such that $(\gamma, b, \delta)\in E^{\prime\prime}$. Since
  $u$ labels an $(\alpha, \beta)$-path in $\Gamma$, $u$ also labels such a path
  in $\Gamma''$. Likewise, $v_1$ labels an $(\alpha, \gamma)$-path in
  $\Gamma''$. Hence $v = v_1b$ labels an $(\alpha, \delta)$-path in $\Gamma''$.
  Since $(u, v)\in R$ and $\Gamma''$ is compatible with $R$, it follows that
  $\beta = \delta$ and so $(\gamma, b, \beta) \in E^{\prime\prime}$.

  By the definition of $\Reach(\Gamma)$ there exists a sequence of word graphs
  \[
    \Gamma = \Gamma_0, \Gamma_1, \ldots, \Gamma_m = \Gamma''
  \]
  and of edges $e_1=(\alpha_1, a_1, \beta_1), e_2= (\alpha_2, a_2, \beta_2),
  \ldots, e_m = (\alpha_m, a_m, \beta_m)$ such that $(\alpha_i, a_i)$ is the
  least missing edge of $\Gamma_{i-1}$ and $\Gamma_i = (V\cup \{\beta_1,
  \ldots, \beta_i\}, E\cup\{e_1, \ldots, e_i\})$ for every $i\in\{1,\ldots,
  m\}$.

  If $j\in \{1, \ldots, m\}$ is such that $e_j = (\alpha_j, a_j, \beta_j) =
  (\gamma, b, \beta)$, then we consider the sequence of word graphs
  \begin{align*}
    \Gamma_0' & = (V\cup \{\beta_j\}, E\cup \{e_j\}) = \Gamma', \\
    \Gamma_1' & = (V\cup \{\beta_j, \beta_1\}, E\cup \{e_j, e_1\}), \\
    & \vdots \\
    \Gamma_{j - 1}' & = (V\cup \{\beta_j, \beta_1, \ldots,\beta_{j-1}\}, E\cup
    \{e_j, e_1, \ldots, e_{j -1}\}) = \Gamma_j,
  \end{align*}
  Suppose the least missing edge in $\Gamma_0' = \Gamma'$ is $(\alpha'_1,
  a_1')$.
  Every missing edge of $\Gamma_0'$ is also a missing edge of
  $\Gamma$, and so $(\alpha'_1, a_1')\geq  (\alpha_1, a_1)$, since
  $(\alpha_1, a_1)$ is the least missing edge in $\Gamma$.
  Since $\Gamma$ and $\Gamma_0'$ differ by the single edge $e_j$, it follows
  that either $j = 1$ and $(\alpha'_1, a_1') > (\alpha_1, a_1) = e_j$;
  or $j > 1$ and $(\alpha'_1, a_1') = (\alpha_1, a_1)$.
  Similarly, if $i < j$, then by the same argument, the least missing edge in
  $\Gamma_i'$ is $(\alpha_i, a_i)$ for every $i$.
  Therefore
  \[
    \Gamma' =\Gamma_0', \ldots, \Gamma_{j - 1}' = \Gamma_j, \Gamma_{j + 1},
    \ldots, \Gamma_m = \Gamma''
  \]
  is a path in $\T$ and so $\Gamma''\in \Reach(\Gamma')$, as required.
\end{proof}

It is possible that $\texttt{MakeCompatible}_R^2(\Gamma) \neq
\texttt{MakeCompatible}_R(\Gamma)$.
To ensure that we add as many edges to $\Gamma$ as possible we could keep track
of whether $\texttt{MakeCompatible}_{R}$ adds any edges to its input, and
run \cref{alg-make-compatible} again until no more edges are added. We denote
this algorithm by $\texttt{MakeCompatibleRepeatedly}_R$; see
\cref{alg-make-compatible-repeatedly} for pseudocode.

\begin{algorithm}
  \caption{- $\texttt{MakeCompatibleRepeatedly}_R$}
  \label{alg-make-compatible-repeatedly}
  \textbf{Input:} A word graph $\Gamma = (V, E)\in \V$ or $\Gamma = \bot$.
  \\
  \textbf{Output:} A word graph or $\bot$.
  \begin{algorithmic}[1]
    \If{$\Gamma = \bot$}
    \State \Return $\bot$
    \EndIf
    \State $\texttt{EdgesAdded} := \textrm{True}$
    \While{\texttt{EdgesAdded}}
    \State $\texttt{EdgesAdded}\gets \textrm{False}$
    \State $\Gamma' \gets \texttt{MakeCompatible}_R(\Gamma)$
    \If{$\Gamma\neq \bot$ and $\Gamma'\neq \Gamma$}
    \State $\texttt{EdgesAdded}\gets \textrm{True}$
    \EndIf
    \State $\Gamma \gets \Gamma^\prime$
    \EndWhile
    \State \Return $\Gamma$
  \end{algorithmic}
\end{algorithm}

That $\texttt{MakeCompatibleRepeatedly}_R$ is a refining function for
$\mathbb{X}_R$
follows by repeatedly applying
\cref{prop-refiner-properties}(ii).
Therefore both $\texttt{MakeCompatible}_R\circ \texttt{AtMost}_n$
and $\texttt{MakeCompatibleRepeatedly}_R\circ \texttt{AtMost}_n$ are refining
functions for $\X_{R}\cap \X_n$.

While $\texttt{MakeCompatibleRepeatedly}_R$ is more computationally expensive
than $\texttt{MakeCompatible}_R$, every edge added to $\Gamma$ by
$\texttt{MakeCompatibleRepeatedly}_R$ reduces the number of nodes in $\T$ that
must be traversed in $\texttt{BacktrackingSearch}_{\X_R}(f, \Xi)$
where $f$ is $\texttt{MakeCompatibleRepeatedly}_R$ by a factor of at least
$|V|+1$. This tradeoff seems quite useful in practice as can be seen in
\cref{appendix-benchmarks}.

Although in line 5 of $\texttt{MakeCompatible}_R$ we loop over all nodes
$\alpha$ and all relations in $R$, in practice this is not necessary. Clearly,
if the word graph $(V, E)$ is compatible with $R$, then we only need to follow
those paths labelled by relations that include any new edges. A
technique for doing just this is given in~\cite[Section 7.2]{Coleman2022aa},
and it is this that is implemented in \libsemigroups.
A comparison of the refining functions for $\X_n\cap\X_R$ presented
in this section is given in \Cref{table-plactic-nodes-visited}

\begin{table}[H]
  \centering
  \begin{tabular}[t]{r|rrrr}
    $n$ & 1 & 2 & 3 & 4 \\\hline
    $\texttt{AtMost}_n$
    & 6 & 165 & 15,989 & 3,556,169\\
    $\texttt{IsCompatible}_R\circ\texttt{AtMost}_n$
    & 6 & 120 & 1,680 & 29,800\\
    $\texttt{MakeCompatible}_R\circ\texttt{AtMost}_n$
    & 6 & 75 & 723 & 6,403\\
    $\texttt{MakeCompatibleRepeatedly}_R\circ\texttt{AtMost}_n$
    & 6 & 75 & 695 & 6,145
  \end{tabular}
  \caption{Number of word graphs visited by
    $\texttt{BacktrackingSearch}_{\X_n\cap \X_R}(f, \Xi)$ for
    the $3$-generated plactic monoid where $f$ ranges over the refining
  functions of $\X_{n}\cap\X_R$ presented in \Cref{subsec-refining-functions}.}
  \label{table-plactic-nodes-visited}
\end{table}

%%%%%%%%%%%%%%%%%%%%%%%%%%%%%%%%%%%%%%%%%%%%%%%%%%%%%%%%%%%%%%%%%%%%%%%%
\section{Applications of Algorithm 1}%
\label{section-applications-of-low-index-algo}
%%%%%%%%%%%%%%%%%%%%%%%%%%%%%%%%%%%%%%%%%%%%%%%%%%%%%%%%%%%%%%%%%%%%%%%%

In this section we describe a number of applications of the low-index right
congruences algorithm. Recall that $M$ is a monoid defined by the presentation
$\langle A \mid R\rangle$. Each application essentially consists of defining a
refining function $f$ so that $\texttt{BacktrackingSearch}_{\X_R\cap \X_n}(f,
\Xi)$ returns a particular subset of right congruences. In brief these subsets
are:
\begin{enumerate}

  \item
    left congruences;

  \item
    2-sided congruences;

  \item
    right congruences including, or excluding, a given subset of $A ^ * \times
    A ^ *$;

  \item
    2-sided congruences $\rho$ such that the quotient $M /\rho$ is a finite
    group;

  \item
    non-trivial right Rees congruences when $M$ has decidable word problem; and

  \item
    right congruences $\rho$ where the right action of $M$ on the nodes of the
    word graph of $\rho$ is faithful.
\end{enumerate}
with index up to $n\in \N$. A further application of (c) provides a practical
algorithm for solving the word problem in finitely presented residually finite
monoids; this is described in \cref{subsec-mckinsey}. Each application
enumerated above is described in a separate subsection below.

A further application of the implementations of the algorithms described in
this section is to reproduce, and extend, some of the results
from~\cite{Bailey2016aa}.

%%%%%%%%%%%%%%%%%%%%%%%%%%%%%%%%%%%%%%%%%%%%%%%%%%%%%%%%%%%%%%%%%%%%%%%%%%%
\subsection{Left congruences}\label{subsection-left-congs}
%%%%%%%%%%%%%%%%%%%%%%%%%%%%%%%%%%%%%%%%%%%%%%%%%%%%%%%%%%%%%%%%%%%%%%%%%%%
In Sections~\ref{section-word-graphs} and~\ref{section-low-index}, we only
considered right congruences, and, in some sense, word graphs are inherently
``right handed''.  It is possible to state an analogue of
\cref{thm-graph-to-right-congruence} for left congruences, for which we require
the following notation. If $w = b_0\cdots b_{m - 1}\in A^*$, then we write
$\tilde{w}$ to denote the reverse $b_{m - 1} \cdots b_0$ of $w$. If $R
\subseteq A^* \times A^*$ is arbitrary, then we denote by $\tilde{R}$ the set
of relations containing $(\tilde{u}, \tilde{v})$ for all $(u, v) \in R$.

An analogue of \cref{prop-word-graph-iff-action} holds for left actions. More
specifically, if $\Psi: M \times V\to V$ is a left action of a monoid $M$ on a
set $V$, and $\theta: A ^ * \to M$ is as above, then the corresponding word
graph is $\Gamma = (V, E)$ where $(\alpha, a, \beta) \in E$ whenever
$((a)\theta, \alpha)\Psi = \beta$. Conversely, if $\Gamma = (V, E)$ is a word
graph, then we define a left action $\Psi: M \times V \to V$ by
\[
  ((w)\theta, \alpha)\Psi = \beta
\]
whenever $\tilde{w}$ labels an $(\alpha, \beta)$-path in $\Gamma$.

\begin{thm}\label{thm-graph-to-left-congruence}
  Let $M$ be a monoid defined by a monoid presentation $\langle A
  \mid R\rangle$.
  Then there is a one-to-one correspondence between the left congruences of $M$
  and the standard complete word graphs over $A$ compatible with $\tilde{R}$.

  If $\rho$ is a left congruence on $M$ and $\Gamma$ is the corresponding word
  graph, then the left actions of $M$ on $M/\rho$ (by left multiplication) and
  on $\Gamma$ are isomorphic; and $\rho =(\{(\tilde{u}, \tilde{v}) :
  (u,v)\in (0)\pi_{\Gamma}\})\theta$ where $\theta : A ^ * \to M$
  is the unique homomorphism with $\ker(\theta) = R ^ {\#}$ and
  $(0)\pi_{\Gamma}$ is the path relation on $\Gamma$.
\end{thm}

It follows from \cref{thm-graph-to-left-congruence} that the left congruences
of a monoid can be enumerated using the same method for enumerating right
congruences applied to $\tilde{R}$. Some care is required here, in particular,
since the corresponding word graphs are associated to right congruences on the
dual of the original monoid (defined by the presentation $\langle
A\mid\tilde{R}\rangle$), rather than to left congruences on $M$.

%%%%%%%%%%%%%%%%%%%%%%%%%%%%%%%%%%%%%%%%%%%%%%%%%%%%%%%%%%%%%%%%%%%%%%%%
\subsection{2-sided congruences}\label{subsec-2-sided-congs}
%%%%%%%%%%%%%%%%%%%%%%%%%%%%%%%%%%%%%%%%%%%%%%%%%%%%%%%%%%%%%%%%%%%%%%%%

% NOTE to selves: what follows is technically wrong because the word
% graph of a left congruence
% corresponds to left-multiplication not right multiplication.
%At this point, it might be natural to suspect that the 2-sided congruences on
%a monoid $M$ correspond to the standard complete word graphs over
%$A$ compatible with both $R$ and $\tilde{R}$. If $\Gamma$ is a word graph
%corresponding to a 2-sided congruence of $M$, then $\Gamma$ must necessarily
%be compatible with both $R$ and $\tilde{R}$; but this condition is not
%sufficient. For example, the partition monoid $P_2$ of degree $2$ is self-dual,
%and so $P_2$ is defined by a presentation $\langle A \mid R\rangle$ with $R =
%\tilde{R}$. But $P_2$ has 13 2-sided congruences and 105 right congruences,
%and so although all 105 of the corresponding word graphs are compatible with
%$R$ and $\tilde{R}$, only 13 of them correspond to 2-sided congruences.

The word graph corresponding to a 2-sided congruence $\rho$ of $M$
is just the right Cayley graph of $M / \rho$ with respect to $A$.
Therefore characterizing $2$-sided congruences is equivalent to
characterizing Cayley graphs.
The corresponding question for groups --- given a word graph $\Gamma=(V, E)$,
determine whether $\Gamma$ is the Cayley graph of a group --- was
investigated in~\cite[Theorem~8.14]{Kapovich2002aa}. An important necessary
condition, in our notation, is that $(0)\pi_\Gamma = (\alpha)\pi_\Gamma$ for
all $\alpha\in V$. In some sense, this condition states that the automorphism
group of $\Gamma$ is transitive. We will next show that the corresponding
condition for monoids is that $(0)\pi_\Gamma \subseteq (\alpha)\pi_\Gamma$ for
all $\alpha\in V$. This condition for monoids is, in the same sense as for
groups, equivalent to the statement that for every node of the Cayley word
graph there is an endomorphism mapping the identity to that node.

\begin{lemma}\label{lem-2-sided-pi-inclusion}
  Let $M$ be the monoid defined by the monoid presentation $\langle A \mid
  R\rangle$, and let $\rho$ be a right congruence on $M$ with word
  graph $\Gamma = (V, E)$. Then $\rho$ is a 2-sided congruence if and only if
  $(0)\pi_\Gamma \subseteq (\alpha)\pi_\Gamma$ for all $\alpha \in V$.
\end{lemma}
\begin{proof}
  For every $\alpha\in V$, we denote by $w_\alpha$ the short-lex minimum word
  labelling any $(0,\alpha)$-path in $\Gamma$. We also denote by $\theta: A^\ast
  \to M$ the unique surjective homomorphism with $\ker(\theta) = R^\#$.
  Recall that, since $\Gamma$ is compatible with $R$, $\ker(\theta)\subseteq
  (0)\pi_{\Gamma}$ and so, by \cref{lem-rho-theta-rho-ker}, $(u, v)\in
  (0)\pi_{\Gamma}$ if and only if $((u)\theta, (v)\theta) \in
  ((0)\pi_{\Gamma})\theta$. Also, by \cref{thm-graph-to-right-congruence},
  $((0)\pi_{\Gamma})\theta = \rho$.

  For the forward implication, suppose that $\rho$ is a 2-sided congruence
  and that $(u, v)\in (0)\pi_\Gamma$. Then $((u)\theta, (v)\theta)\in
  ((0)\pi_{\Gamma})\theta = \rho$. Since $\rho$ is a 2-sided congruence,
  \[
    ((w_\alpha)\theta\cdot (u)\theta, (w_\alpha)\theta\cdot
    (v)\theta) = ((w_\alpha
    u)\theta, (w_\alpha v)\theta) \in \rho
  \]
  for all $\alpha \in V$. In particular, again by
  \cref{lem-rho-theta-rho-ker}, $(w_\alpha u, w_\alpha
  v)\in (0)\pi_\Gamma$. Thus $0\cdot w_\alpha u = 0\cdot w_\alpha
  \neq \bot$. Since $0\cdot w_\alpha = \alpha$, it follows that
  $\alpha \cdot u = \alpha \cdot v$ and so $(u,
  v)\in (\alpha)\pi_\Gamma$ as required.

  For the converse implication, assume that $(0)\pi_\Gamma \subseteq
  (\alpha)\pi_\Gamma$ for all $\alpha\in V$. This implies that $\Gamma$
  is compatible with $(0)\pi_\Gamma$ and so $(0)\pi_\Gamma^\# \subseteq
  (\alpha)\pi_\Gamma$ for all $\alpha\in V$. Therefore
  $(0)\pi_\Gamma = (0)\pi_\Gamma^\#$ is a 2-sided congruence. Since $\Gamma$
  is also compatible with $R$, $\ker(\theta) = R^\# \subseteq (0)\pi_\Gamma =
  (0)\pi_\Gamma^\#$. So applying
  \cref{lem-containing}\ref{lem-rho-preserves-least-congruence}
  yields:
  \[\rho = (0)\pi_\Gamma\theta = \left((0)\pi_\Gamma^\#\right)\theta =
  \left(((0)\pi_\Gamma)\theta\right)^\# = \rho^\#,\]
  and so $\rho$ is a 2-sided congruence.
\end{proof}

We can further refine \cref{lem-2-sided-pi-inclusion} by using the generating
pairs of \cref{lem-generating-pairs} to yield a computationally testable
condition as follows.

\begin{thm}\label{thm-2-sided-condition}
  Let $M$ be the monoid defined by the monoid presentation $\langle A \mid
  R\rangle$, and let $\rho$ be a right congruence on $M$ with word
  graph $\Gamma = (V, E)$. Then $\rho$ is a 2-sided congruence if and only if
  $\Gamma$ is compatible with $\left\{(w_\alpha a, w_\beta) : (\alpha, a,
  \beta) \in E \right\}$ where $w_{\alpha} \in A ^ *$ is the short-lex minimum
  word labelling any $(0,\alpha)$-path in $\Gamma$.
\end{thm}
\begin{proof}
  Let $\theta : A^\ast \to M$ be the unique surjective homomorphism with
  $\ker(\theta) = R^\#$ and let $X_{\Gamma}$ denote the set
  \[
    \left\{(w_\alpha a, w_\beta) : (\alpha, a, \beta) \in E \right\}.
  \]

  ($\Rightarrow$)
  If $\rho$ is a 2-sided congruence, then by \cref{lem-2-sided-pi-inclusion},
  $(0)\pi_\Gamma\subseteq (\alpha)\pi_\Gamma$ for all $\alpha\in V$.
  The relation
  $X_\Gamma$ is contained in $(0)\pi_\Gamma$ by definition. Therefore
  $X_\Gamma\subseteq
  (0)\pi_{\Gamma} \subseteq (\alpha)\pi_\Gamma$ for all $\alpha\in V$. Hence
  $\Gamma$ is compatible with $X_\Gamma$ as required.

  ($\Leftarrow$)
  Assume that $\Gamma$ is compatible with $X_\Gamma$. Then $X_\Gamma^\#
  \subseteq (0)\pi_\Gamma$ and so $(X_\Gamma^\#)\theta \subseteq
  ((0)\pi_{\Gamma})\theta = \rho$.
  Since $\ker(\theta) = R^\# \subseteq (0)\pi_\Gamma$ and $(0)\pi_{\Gamma}$ is
  the least right congruence containing $X_{\Gamma}$, it follows
  that $\ker(\theta) \subseteq X_\Gamma^\#$. Hence
  $(X_{\Gamma}^{\#})\theta = ((X_{\Gamma})\theta)^{\#}$ by
  \cref{lem-containing}\ref{lem-rho-preserves-least-congruence}.
  Therefore $((X_{\Gamma})\theta)^{\#} = (X_\Gamma^\#)\theta \subseteq \rho$.
  But $\rho$ is generated as a right congruence by $(X_{\Gamma})\theta$ by
  \cref{lem-generating-pairs} and so $((X_{\Gamma})\theta)^{\#} \subseteq \rho
  \subseteq ((X_{\Gamma})\theta)^{\#}$ giving equality throughout. In
  particular, $\rho$ is a 2-sided congruence, as required.
\end{proof}

In \cref{thm-2-sided-condition} we showed there is a bijection between the
2-sided congruences of the monoid $M$ defined by $\langle A \mid R\rangle$ and
the complete standard word graphs $\Gamma$ compatible with both $R$ and the set
$X_{\Gamma} = \{(w_\alpha a, \beta) : (\alpha, a, \beta)\in E\}$. We denote by
$\Y_{R}$ the set of complete standard word graphs corresponding to 2-sided
congruences of $M$. Recall from \cref{cor-generating-pairs} that we can compute
$X_\Gamma$ from $\Gamma = (V, E)$ in $O(|V|^2|A|)$ time. We can also verify
that a given word graph is compatible with $X_{\Gamma}$ in $O(m\cdot |V|)$
where $m$ is the sum of the lengths of the words occurring in $X_{\Gamma}$,
using, for example,  $\texttt{IsCompatible}_{X_\Gamma}$.

% TODO(later) mention that the data structure used to only follow paths that
% go through new edges can't be used so easily here. Details about incremental
% datastructure for this

We define the function
$\texttt{TwoSidedMakeCompatibleRepeatedly}_R : \V\cup\{\bot\} \to
\V\cup\{\bot\}$ by
\[
  \texttt{TwoSidedMakeCompatibleRepeatedly} (\Gamma) =
  \begin{cases}
    \bot & \Gamma = \bot \\
    \texttt{MakeCompatibleRepeatedly}_{X_\Gamma}(\Gamma)
    & \text{otherwise}
  \end{cases}
\]
for all $\Gamma \in \V\cup \{\bot\}$.

\begin{lemma}
  $\texttt{TwoSidedMakeCompatibleRepeatedly}$ is a refining function for $\Y_R$.
\end{lemma}
\begin{proof} % TODO(later): remove the proof if pressed for space
  On superficial inspection, it might seem that
  $\texttt{TwoSidedMakeCompatibleRepeatedly}$ is a refining function because
  $\texttt{MakeCompatibleRepeatedly}_{X_\Gamma}$ is a refining function.
  However, this does not follow immediately because the set $X_\Gamma$ is
  dependent on the input word graph $\Gamma$.

  Clearly, $\texttt{TwoSidedMakeCompatibleRepeatedly}(\bot) = \bot$ so
  \cref{de-refining-func}(i) holds.

  If $\Gamma \subseteq \Gamma^\prime$ are deterministic
  word graphs, then $X_\Gamma\subseteq X_{\Gamma^\prime}$ and so
  $\X_{X_\Gamma} \supseteq \X_{X_{\Gamma^\prime}}$. If additionally
  $\Gamma^\prime\in \Y_R$, then $\Gamma^\prime \in \X_{X_{\Gamma^\prime}}
  \subseteq \X_{X_\Gamma}$ by \cref{thm-2-sided-condition}. Hence,
  $\Reach(\Gamma)\cap \Y_R \subseteq \Reach(\Gamma)\cap \X_{X_\Gamma}$. If
  $\texttt{MakeCompatibleRepeatedly}_{X_\Gamma}(\Gamma) = \bot$, then
  $\Reach(\Gamma)\cap\X_{X_\Gamma} = \varnothing$ by
  \cref{de-refining-func}(ii) applied to
  $\texttt{MakeCompatibleRepeatedly}_{X_\Gamma}$. Hence $\Reach(\Gamma)\cap
  \Y_R = \varnothing$ also and so \cref{de-refining-func}(ii) holds.

  If $\texttt{MakeCompatibleRepeatedly}_{X_\Gamma}(\Gamma) = \Gamma^\prime\neq
  \bot$, then $\Reach(\Gamma)\cap\X_{X_\Gamma} = \Reach(\Gamma^\prime)\cap
  \X_{X_\Gamma}$ by \cref{de-refining-func}(iii) applied to
  $\texttt{MakeCompatibleRepeatedly}_{X_\Gamma}$. Since $\Reach(\Gamma^\prime)
  \subseteq \Reach(\Gamma)$ and $\Reach(\Gamma)\cap\Y_R \subseteq
  \Reach(\Gamma)\cap \X_{X_{\Gamma}}$, it follows that $\Reach(\Gamma)\cap \Y_R
  = \Reach(\Gamma^\prime)\cap\Y_R$, and so \cref{de-refining-func}(iii) holds.
\end{proof}

It follows that $\texttt{BacktrackingSearch}_{\Y_R\cap \X_n}(f, \Xi)$
where $f$ is
the refining function $\texttt{TwoSidedMakeCompatibleRepeatedly}\circ
\texttt{MakeCompatibleRepeatedly}_R\circ \texttt{AtMost}_n$ returns $\Y_R\cap
\X_n$ the set of word graphs of the 2-sided congruences of the monoid defined
by $\langle A\mid R\rangle$ with index at most $n$.
As a practical comparison, the number of word graphs visited by
$\texttt{BacktrackingSearch}_{\X_n\cap \Y_R}( f,
\Xi)$ where $n = 6$ and $f=
\texttt{MakeCompatibleRepeatedly}_R\circ\texttt{AtMost}_n$ for
the $3$-generated plactic monoid is $662,550$. On the other hand, the
number of word graphs visited when using the refining function $f =
\texttt{TwoSidedMakeCompatibleRepeatedly}\circ
\texttt{MakeCompatibleRepeatedly}_R\circ \texttt{AtMost}_n$ is only $37,951$.

As an example, in \cref{table-2-sided} we compute the number of 2-sided
congruences with index at most $n$ of the free monoid $A ^ *$ when
$n$ and $|A|$ are not too large. For example,  we compute the number of
2-sided congruences of $A ^ *$ up to index $22$, $14$, $10$ and $9$, when $|A|
= 2, 3, 4$, and $5$, respectively.

%%%%%%%%%%%%%%%%%%%%%%%%%%%%%%%%%%%%%%%%%%%%%%%%%%%%%%%%%%%%%%%%%%%%%%%%
\subsection{Congruences including or excluding a relation}%
\label{subsec-include-exclude-refining-functions}
%%%%%%%%%%%%%%%%%%%%%%%%%%%%%%%%%%%%%%%%%%%%%%%%%%%%%%%%%%%%%%%%%%%%%%%%

Given two elements $x$ and $y$ of the monoid $M$ defined by the finite
presentation $\langle A\mid R\rangle$, we might be interested in finding finite
index right congruences containing $(x, y)$ or not containing $(x, y)$. Suppose
that $x, y\in M$ and $u,v\in A^\ast$ are such that $(u)\theta = x, (v)\theta =
y$ where $\theta: A ^ * \to M$ is the unique homomorphism with $\ker(\theta) = R
^ {\#}$. By~\cref{thm-graph-to-right-congruence}, if $\rho$ is a right
congruence of $M$, then $(x, y)\in\rho$ if and only if $\alpha \cdot u =
a\alpha\cdot v\neq \bot$ in the word graph $\Gamma$ of $\rho$.

For $u, v\in A ^*$, we denote by $\mathbb{X}_{(u, v)}$ the set of complete
standard word graphs such that $0\cdot u = 0\cdot v$. Similarly, we denote by
$\mathbb{X}_{\overline{(u, v)}}$  the set of complete standard word graphs such
that $0\cdot u \neq 0 \cdot v$.

We also define refining functions $\texttt{Include}_{(u, v)}$ and
$\texttt{Exclude}_{(u, v)}$ by
\begin{align*}
  \texttt{Include}_{(u, v)}(\Gamma) &=
  \begin{cases}
    \bot & \textrm{ if } \Gamma = \bot \text{ or } 0\cdot u\neq \bot,
    0\cdot v\neq \bot \text{ and } 0\cdot u \neq 0 \cdot v \\
    \Gamma & \text{ otherwise}
  \end{cases}\\
  \texttt{Exclude}_{(u, v)}(\Gamma) &=
  \begin{cases}
    \bot & \text{ if } \Gamma = \bot \text{ or } 0\cdot u\neq \bot,
    0\cdot v\neq \bot \text{ and } 0\cdot u = 0 \cdot v \\
    \Gamma & \text{ otherwise}
  \end{cases}
\end{align*}
It is routine to verify that $\texttt{Include}_{(u, v)}$ and
$\texttt{Exclude}_{(u, v)}$ are refining functions for $\mathbb{X}_{(u,
v)}$ and $\mathbb{X}_{\overline{(u, v)}}$, respectively.

Composing these refining functions with $\texttt{AtMost}_n$ and any of the
refining functions for one or 2-sided congruences from
\cref{section-low-index} and \cref{subsec-2-sided-congs} allows us to find
one or 2-sided congruences of $\langle A\mid R\rangle$ with index at most $n$
that include or exclude a given relation.

%%%%%%%%%%%%%%%%%%%%%%%%%%%%%%%%%%%%%%%%%%%%%%%%%%%%%%%%%%%%%%%%%%%%%%%%
\subsection{McKinsey's algorithm}\label{subsec-mckinsey}
%%%%%%%%%%%%%%%%%%%%%%%%%%%%%%%%%%%%%%%%%%%%%%%%%%%%%%%%%%%%%%%%%%%%%%%%

A monoid $M$ is \defn{residually finite} if for all $s, t\in M$ with $s \neq t$
there exists a finite monoid $M^\prime$ and homomorphism $\phi : M\to M^\prime$
such that $(s)\phi \neq (t)\phi$. In ~\cite{McKinsey1943aa}, McKinsey gave an
algorithm for deciding the word problem in finitely presented residually finite
monoids. McKinsey's Algorithm in~\cite{McKinsey1943aa} is, in fact, more
general, and can be applied to residually finite universal algebras.

McKinsey's algorithm relies on two semidecision procedures --- one for testing
equality in a finitely presented monoid and the other for testing inequality in
a finitely generated residually finite monoid. It is well-known (and easy to
show) that testing equality is semidecidable for every finitely presented
monoid.

Suppose that $M$ is finitely presented by $\langle A\mid R\rangle$ and that $M$
is residually finite. There are only finitely many finite monoids $M'$ of every
size, and only finitely many possible functions from $A$ to $M'$. If $f: A \to
M'$ is any such function, then it is possible to verify that $f$ extends to a
homomorphism $\phi: M \to M'$ by checking that $M'$ satisfies the (finite set
of) relation $R$. Clearly, if $s, t\in M$ and $(s)\phi \neq (t)\phi$ for some
$\phi$, then, since $\phi$ is a function, $s\neq t$.  It follows that it is
theoretically possible to verify that $s\neq t$ by looping over the finite
monoids $M'$, the functions $f:A \to M'$, and for every $f$ that extends to a
homomorphism $\phi: M \to M'$, testing whether $(s)\phi \neq (t)\phi$.
Thus testing inequality in a finitely presented residually finite monoid $M$ is
also semidecidable.

McKinsey's algorithm proceeds by running semidecision algorithms for
testing equality and inequality in parallel; this is guaranteed to terminate,
and so the word problem for finitely presented residually finite monoids $M$ is
decidable in theory. In practice, checking for equality in a finitely presented
monoid $M$ with presentation $\langle A \mid R\rangle$ can be done somewhat
efficiently by performing a backtracking search in the space of all elementary
sequences over $R$. On the other hand, the semidecision procedure given above
for checking inequality is extremely inefficient. For example, the number of
monoids of size at most $10$ up to isomorphism and anti-isomorphism is
$52,993,098,927,712$; see~\cite{Distler2009aa}.

The low-index congruences algorithm provides a more efficient algorithm for
deciding inequality in a finitely presented residually finite monoid by
utilizing the $\texttt{Exclude}_{(u, v)}$ refining function for some $u, v\in A
^*$. The set $\X_{\overline{(u, v)}}\cap \Y_{R, n}$ consists of exactly the
2-sided congruences on $\langle A \mid R\rangle$ with index at most $n$ such
that $(u)\theta\neq (v)\theta$ (where $\theta: A ^* \to M$ is the natural
homomorphism). Hence $(u)\theta \neq (v)\theta$ in $M$ if and only if
$\X_{\overline{(u,v)}}\cap \Y_{R, n}\neq \{\bot\}$ for some $n$.
Therefore $\texttt{BacktrackingSearch}_{\X_{\overline{(u,v)}}\cap \Y_{R, n}}(
\texttt{Exclude}_{(u, v)}, \Xi)$ can be used to implement McKinsey's algorithm
with a higher degree of practicality.

%%%%%%%%%%%%%%%%%%%%%%%%%%%%%%%%%%%%%%%%%%%%%%%%%%%%%%%%%%%%%%%%%%%%%%%%
\subsection{Congruences defining groups}%
\label{subsec-group-congruences}
%%%%%%%%%%%%%%%%%%%%%%%%%%%%%%%%%%%%%%%%%%%%%%%%%%%%%%%%%%%%%%%%%%%%%%%%

We say that a 2-sided congruence $\rho$ on $M$ is a \defn{group congruence}
if the quotient monoid $M/\rho$ is a group. A 2-sided congruence $\rho$ is a
group congruence if and only if for every $x \in M$ there exists $y\in M$ such
that $(xy,1_M)\in \rho $ and $(yx, 1_M)\in \rho$ where $1_M\in M$ is the
identity element. If $M$ is generated by $A$, then $\rho$ is a group congruence
if and only if for every $x \in A$ there exists $y\in M$ such that $(xy,1_M)\in
\rho $ and $(yx, 1_M)\in \rho$. We say that a word graph $\Gamma$ is
\defn{injective} if for all $\beta\in V$ and $a\in A$ there is at most one edge
in $E$ with target $\beta$ and label $a$. This is the dual of the definition of
determinism. We can decide if a finite word graph $\Gamma$ corresponds to a
group congruence as follows.

\begin{thm}
  Let $M$ be the monoid defined by the monoid presentation $\langle A \mid
  R\rangle$, and let $\rho$ be a finite index 2-sided congruence on $M$ with
  word graph $\Gamma = (V, E)$. Then $\rho$ is a group congruence if
  and only if $\Gamma$ is injective.
\end{thm}
\begin{proof}
  Let $\theta: A^\ast \to M$ be the unique surjective homomorphism with
  $\ker(\theta) = R^\#$.

  $(\Rightarrow)$ Let $(\beta, a, \alpha), (\gamma, a, \alpha)\in E$ for some
  $\alpha, \beta, \gamma\in V$ and $a\in A$. Since $\rho$ defines a group
  there exists $y\in M$ such that $((a)\theta \cdot y, 1_M)\in  \rho$. Since
  $\rho$ is a 2-sided congruence, $((w_\beta)\theta\cdot (a)\theta \cdot
  y, (w_\beta)\theta)\in \rho$ and so $((w_\beta a)\theta\cdot y,
  (w_\beta)\theta) \in \rho$ and similarly $((w_\gamma a)\theta \cdot y,
  (w_\gamma)\theta) \in \rho$. But $w_\beta a$ and $w_{\gamma}a$ both label
  $(0,\alpha)$-paths in $\Gamma$, and so $((w_\beta a)\theta, (w_\gamma
  a)\theta)\in\rho$. Hence by transitivity $((w_\beta)\theta,
  (w_\gamma)\theta)\in  \rho$. Then, by \cref{lem-rho-theta-rho-ker} and
  \cref{thm-graph-to-right-congruence}, $(w_\beta, w_\gamma)\in
  (0)\pi_\Gamma$ and so $\beta = \gamma$. We have shown that $\Gamma$ is
  injective.

  $(\Leftarrow)$ Suppose that $a\in A$. Since the set $\{a^n : n\in\N_0\}$ is
  infinite but $\Gamma$ has only finitely many nodes, it follows from the
  pigeonhole principle that there exists $\alpha\in V$ and $i, j\in \N_0$
  with $i < j$ such that $a^i$ and $a^j$ both label $(0,\alpha)$-paths in
  $\Gamma$. Assume that $i$ is the least such value. If $i>0$, then there
  exist $\beta, \gamma \in V$ such that $a^{i-1}$ and $a^{j-1}$ label $(0,
  \beta)$- and $(0,\gamma)$-paths respectively. It follows that $(\beta, a,
  \alpha), (\gamma, a, \alpha)\in E$ and so by injectivity $\beta=\gamma$. In
  particular, $a ^ {i - 1}$ and $a ^ {j - 1}$ both label $(0, \beta)$-paths,
  and this contradicts the minimality of $i$.

  Therefore $i = 0$ and so $(\varepsilon, a^j)\in (0)\pi_\Gamma$. Hence, by
  \cref{thm-graph-to-right-congruence}, $(1_M, (a)\theta^j)\in \rho$. In
  particular, if $y = (a)\theta^{j-1}$, then $((a)\theta\cdot y, 1_M),
  (y\cdot (a)\theta, 1_M)\in\rho$. Since $\theta$ is surjective, and $a\in A$
  was arbitrary, it follows that $\rho$ is a group congruence.
\end{proof}

It is possible to verify if a given word graph over $A$ is injective, or not.
In particular, in the representation used in \libsemigroups, this can be
verified in time linear in $|A||V|$. Furthermore, if $\Gamma$ is not injective,
then neither is any descendent of $\Gamma$ in the search multitree $\T$. Hence
the following function is a refining function for the set of all word graphs
corresponding to group congruences:
\[
  \texttt{IsInjective}(\Gamma) =
  \begin{cases}
    \Gamma & \text{ if } \Gamma \text{ is an injective word graph}\\
    \bot & \text{ otherwise.}
  \end{cases}
\]
Composing $\texttt{IsInjective}$, $\texttt{AtMost}_n$, and any of the refining
functions for $\mathbb{Y}_R$ of word graphs corresponding to 2-sided
congruences of $\langle A\mid R\rangle$, this gives us a method for computing
all group congruences with index at most $n$ of the monoid presented by
$\langle A \mid R\rangle$.

\subsection{Rees congruences}\label{subsec-rees-congs}
%%%%%%%%%%%%%%%%%%%%%%%%%%%%%%%%%%%%%%%%%%%%%%%%%%%%%%%%%%%%%%%%%%%%%%%%

In this section we describe how to use the low-index right congruences
algorithm to compute Rees congruences, i.e. those arising from ideals. A
related algorithm for finding low-index Rees congruences is given
in~\cite{JuraIdeals1} and~\cite{JuraIdeals2}. Like the low-index congruences
algorithm, Jura's Algorithm in~\cite{JuraIdeals1} and~\cite{JuraIdeals2} also
uses some aspects of the Todd--Coxeter Algorithm. The algorithm presented in
this section is distinct from Jura's Algorithm. In general, the problem of
computing the finite index ideals of a finitely presented monoid is
undecidable; see~\cite{JuraIdeals2} and~\cite[Theorem 5.5]{Ruskuc1998aa}.
However, if the word problem happens to be decidable for a finitely presented
monoid, then so too is the problem of computing the finite index ideals of that
monoid.

Given a right ideal $I$ of a monoid $M$, the \defn{right Rees congruence} of
$I$ is $\rho_I = \{(x, y) \in M\times M : x = y \textrm{ or } x, y\in I\}$. The
trivial congruence $\Delta_M$ is a right Rees congruence if and only if $M$ has
a right zero; and the trivial congruence has finite index if and only if $M$ is
finite. As such, we will restrict ourselves, in this section, to considering
only non-trivial finite index right Rees congruence.

Let $\Gamma=(V, E)$ be a standard word graph of a right congruence of $M$ and
let $\theta: A^\ast \to M$ be the unique homomorphism with $\ker(\theta) =
R^\#$. We call a node $\omega\in V$ a \defn{sink} if $(\omega, a, \omega)\in E$
for all $a\in A$. We say that a sink $\omega$ is \defn{non-trivial} if there
exists an edge $(\alpha, a, \omega)\in E$ such that $(w_\alpha a)\theta \neq
(w_\omega)\theta$, where as usual $w_\alpha\in A ^ *$ is the short-lex least
word labelling a $(0, \alpha)$-path in $\Gamma$.

If $\Gamma$ is compatible with the relations $R$ defining $M$, $\rho$ is the
right congruence of any complete standard word graph compatible with $R$ that
contains $\Gamma$, and $\alpha\in V$ is a non-trivial sink, then the
equivalence class of $\rho$ on $M$ corresponding to $\alpha$ contains at least
2 elements: $(w_{\alpha}a)\theta$ and $(w_{\omega})\theta$. In particular,
$\rho$ is non-trivial, which explains why we called $\alpha$ a non-trivial
sink.

We give a criterion for deciding if a complete standard word graph corresponds
to a non-trivial right Rees congruence in the next theorem.

\begin{thm}%
  \label{thm-rees-congruence-condition}
  Let $M$ be the monoid defined by the monoid presentation $\langle A \mid
  R\rangle$, let $\rho$ be a right congruence on $M$ with word graph $\Gamma
  = (V, E)$, and let $\theta: A^\ast \to M$ be the unique homomorphism with
  $\ker(\theta) = R^\#$. Then $\rho$ is a non-trivial right Rees congruence
  if and only if the following conditions hold:
  \begin{enumerate}[\rm (i)]
    \item there exists a unique non-trivial sink $\omega \in V$;
    \item if $(\alpha, a, \beta)\in E$ and $\beta\neq \omega$, then
      $(w_\alpha a)\theta = (w_\beta)\theta$.
  \end{enumerate}
\end{thm}
\begin{proof}
  ($\Rightarrow$)
  Let $I$ be a right ideal of $M$ such that $1 < |I|$ and $I\neq M$, and let
  $\rho = \rho_I$ be the corresponding non-trivial right Rees congruence with
  complete standard word graph $\Gamma = (V, E)$. If $u\in A^\ast$ is such
  that $(u)\theta\in I$, then since $\Gamma$ is complete, $0\cdot u = \omega$
  for some $\omega \in V$.

  If $(v)\theta\in I$ for some $v\in A^\ast$, then $((u)\theta, (v)\theta)\in
  \rho$ since $\rho$ is a right Rees congruence. By
  \cref{lem-rho-theta-rho-ker}, it follows that $(u, v)\in (0)\pi_\Gamma$,
  and so $0\cdot v = \omega$ also. Conversely,
  if $(u, v)\in (0)\pi_\Gamma$, then $((u)\theta, (v)\theta)\in \rho$ and
  hence $(v)\theta\in I$. Hence $w\in A^\ast$ labels a $(0,\omega)$-path in
  $\Gamma$ if and only if $(w)\theta \in I$.

  In particular $(w_\omega)\theta \in I$ and if $a\in A$ is arbitrary, then
  $(w_\omega)\theta\cdot (a)\theta = (w_\omega a)\theta \in I$, and so
  $w_\omega a$ also labels a $(0,\omega)$-path in $\Gamma$. It follows that
  $(\omega, a, \omega)\in  E$ for all $a\in A$ and $\omega$ is a sink.
  To show that (i) holds, it remains to show that $\omega$ is non-trivial and
  unique.

  To show that $\omega$ is non-trivial, consider the set
  \[
    W = \set{w\in A ^ *}{(w)\theta\in I\text{ and }(w)\theta\neq
    (w_{\omega})\theta}.
  \]
  Since $|I| > 1$ and $\theta$ is surjective, this set is non-empty. We set $v$
  to be the short-lex least word in $W$. Since $0\cdot v = 0\cdot w_{\omega} =
  \omega$, and $w_{\omega}$ is the short-lex least such word, and $v \neq
  w_{\omega}$, it follows that $v > w_{\omega} \geq \varepsilon$. Hence $v = v_1
  a$ for some $v_1\in A^\ast$ and some $a\in A$. If $v_1$ labels a $(0,
  \alpha)$-path in $\Gamma$ for some $\alpha\in V$, then $(\alpha, a, \omega)\in
  E$ and so it suffices to show that  $(w_\alpha a)\theta \neq
  (w_{\omega})\theta$. If $(v_1)\theta=(w_\alpha)\theta$, then $(w_\alpha
  a)\theta = (v_1a)\theta = (v) \theta$. Since $v\in W$, it would follow that
  $(w_{\alpha}a)\theta = (v)\theta \not= (w_{\omega})\theta$ and $\omega$ is
  non-trivial. Hence it suffices to show that $(v_1)\theta=(w_\alpha)\theta$.

  If $\alpha \neq \omega$, then $0\cdot w_{\alpha}\neq \omega$, and
  so $(w_\alpha)\theta \not \in I$. But $(w_\alpha, v_1) \in
  (0)\pi_{\Gamma}$ implies that $((w_\alpha)\theta, (v_1)\theta)\in\rho$ and so
  $(w_\alpha)\theta = (v_1)\theta$ since $\rho$ is a Rees congruence, as
  required. If $\alpha = \omega$, then $0\cdot v_1 = \omega$ and so
  $(v_1)\theta \in I$. Since $v_1 < v$ and $v$ is the least element of $W$, it
  follows that $v_1\not\in W$ and so $(v_1)\theta = (w_{\omega})\theta =
  (w_{\alpha})\theta$, as required. We have shown that $\omega$ is non-trivial.

  To establish the uniqueness of $\omega$, let $\omega^\prime\in V$ be a
  non-trivial sink. By the definition of non-trivial sinks, there exists
  $(\alpha, a, \omega^\prime)\in E$ such that $(w_\alpha a)\theta \neq
  (w_{\omega^\prime})\theta$. Since $0\cdot w_{\alpha}a = 0\cdot w_{\omega'}=
  \omega'$, it follows that $((w_\alpha a)\theta, (w_{\omega^\prime})\theta)\in
  \rho$. If $(w_{\omega'})\theta\not \in I$ or $(w_{\alpha}a)\theta\not\in I$,
  then, since $\rho$ is a Rees congruence, $(w_{\alpha}a)\theta =
  (w_{\omega^\prime})$, which is a contradiction. Hence $w_{\omega'}\theta\in
  I$ and so $\omega = \omega'$.

  To show that (ii) holds, let $(\alpha, a, \beta)\in E$ and $\beta\neq
  \omega$. It follows that neither $0\cdot w_\alpha a, 0\cdot w_\beta \neq
  \omega$, and so $(w_\alpha a)\theta, (w_\beta)\theta\not \in I$. On the other
  hand $(w_\alpha a, w_\beta)\in (0)\pi_\Gamma$ implies $((w_\alpha a)\theta,
  (w_\beta)\theta) \in \rho$ and so $(w_\alpha a)\theta = (w_\beta)\theta$
  again since $\rho$ is a right Rees congruence.

  ($\Leftarrow$)
  Let $\omega$ be the unique node satisfying condition (i) and let
  \[
    I = \{(u)\theta\in M \ :\ u\in A ^*,\ 0\cdot u = \omega \}.
  \]
  If $\omega = 0$, then $I = M$, and $\rho_I = M \times M$ is a non-trivial Rees
  congruence. Hence we may suppose that $\omega \neq 0$.

  By assumption $(\omega, a, \omega) \in E$ for all $a\in A$ and so
  $(\varepsilon, v)\in (\omega)\pi_\Gamma$ for all $v\in A^\ast$. Hence if
  $u\in A ^ *$ labels a $(0,\omega)$-path in $\Gamma$, then so does $uv$ for
  all $v\in A^\ast$ and so $(uv)\theta = (u)\theta\cdot (v)\theta\in I$ for all
  $(u)\phi\in I$ and $v\in A^\ast$. Since $\theta$ is surjective, this implies
  that $I$ is a right ideal of $M$. It suffices by
  \cref{thm-graph-to-right-congruence} to show that $\rho_I =
  ((0)\pi_\Gamma)\theta$.

  Suppose that $(x, y)\in \rho_I$. If $x, y\in I$, then there exist $u,v\in
  A^\ast$ such that $x = (u)\theta$ and $y=(v)\theta$. Hence by the definition
  of $I$, both $u$ and $v$ label $(0, \omega)$-paths in $\Gamma$. This implies
  $(u, v)\in (0)\pi_\Gamma$ and so $(x, y)\in ((0)\pi_\Gamma)\theta$.
  Otherwise, if $x = y$, then $(x, y)\in ((0)\phi_\Gamma)\theta$ by reflexivity
  since $((0)\pi_\Gamma)\theta$ is a right congruence. Hence $\rho_I\subseteq
  ((0)\pi_\Gamma)\theta$.

  For the converse, suppose that $((u)\theta, (v)\theta) \in
  (0)\pi_{\Gamma}\theta$ for some $u, v\in A ^*$ such that $(u)\theta,
  (v)\theta\not\in I$. We proceed by induction on $\max\{|u|, |v|\}$. If
  $\max\{|u|, |v|\} = 0$, then $u = v= \varepsilon$ and so $((u)\theta ,
  (v)\theta)\in \rho_I$  by reflexivity.

  Suppose that for some $n \geq 1$ and for all $u, v\in A ^ *$ with $|u|, |v| <
  n$, $((u)\theta, (v)\theta) \in (0)\pi_{\Gamma}\theta$ implies $((u)\theta,
  (v)\theta)\in \rho_I$. Let $u, v\in A ^ *$ be such that $\max\{|u|, |v|\} = n
  > 0$ and $((u)\theta, (v)\theta) \in (0)\pi_{\Gamma}\theta$. Without loss of
  generality there are two cases to consider: when $v = \varepsilon$; and when
  $u\neq \varepsilon$ and $v\neq \varepsilon$.

  If $v= \varepsilon$, then $u\neq \varepsilon$ by assumption. Hence we can
  write $u = u_1a$ for some $u_1\in A ^ *$ and $a\in A$. If $0\cdot u_1 =
  \alpha\in V$, then $(u_1, w_{\alpha})\in (0)\pi_{\Gamma}$ and so $|u_1| \geq
  |w_{\alpha}|$. In particular, $|u_1|, |w_{\alpha}| \leq |u_1| < |u| = n$ and
  so by induction $((u_1)\theta, (w_{\alpha})\theta)\in \rho_I$. Thus
  $((u)\theta, (w_\alpha a)\theta) \in \rho_I$. Since $(\alpha, a, 0)\in
  E$ and $0\neq \omega$,
  it follows by (ii) that $((w_\alpha a)\theta, (w_0)\theta) = ((w_\alpha
  a)\theta, (\varepsilon)\theta)\in \rho$. Hence $((u)\theta, (v)\theta) =
  ((u)\theta, (\varepsilon)\theta) \in \rho_I$ by transitivity.

  If $u\neq \varepsilon$ and $v\neq \varepsilon$, then
  we can write $u = u_1a$ and $v = v_1b$ for some $u_1, v_1\in A ^ *$ and $a,
  b\in A$. If $0\cdot u_1 = \alpha$ and $0\cdot v_1 = \beta$, then
  $(u_1, w_{\alpha}), (v_1, w_{\beta}) \in (0)\pi_{\Gamma}$. Since $u_1 \geq
  w_{\alpha}$, it follows that $|u_1|, |w_{\alpha}| < |u| \leq  n$. Similarly
  $|v_1|, |w_{\beta}| < |v| \leq n$. Hence, by induction, $((u_1)\theta,
  (w_{\alpha})\theta), ((v_1)\theta, (w_{\beta})\theta)\in \rho_I$ and so
  $((u)\theta, (w_{\alpha}a)\theta), ((v)\theta, (w_{\beta}b)\theta)\in\rho_I$.
  If $0\cdot u = \gamma = 0\cdot v$, then by (ii) applied to $(\alpha, a,
  \gamma), (\beta, b, \gamma)\in E$, it follows that $(w_{\alpha}a)\theta =
  (w_{\gamma})\theta = (w_{\beta}b)\theta$. Thus, by transitivity,
  $((u)\theta, (v)\theta) \in \rho_I$, and the proof is complete.
\end{proof}

Unlike in the previous subsections, the conditions of
\cref{thm-rees-congruence-condition} can only be tested if a method for solving
the word problem in the monoid $\langle A \mid R\rangle$ is known. Given an
algorithm for solving the word problem, the non-triviality of a sink in
\cref{thm-rees-congruence-condition}(i) and the condition $(w_\alpha a)\theta =
(w_\beta)\theta$ in \cref{thm-rees-congruence-condition}(ii) can both be
verified computationally.

Let $\Z_R$ be the set of all standard complete word graphs corresponding to
non-trivial right Rees congruences on the monoid presented by $\langle A \mid
R\rangle$. The function $\texttt{IsRightReesCongruence}_R$ is defined in
\cref{alg-is-right-rees-congruence}. We will show that the function
$\texttt{IsRightReesCongruence}_R$, is a refining function for $\Z_R$ in
\cref{lemma-rees-refiner}.

\begin{algorithm}
  \caption{- $\texttt{IsRightReesCongruence}_R$}
  \label{alg-is-right-rees-congruence}
  \textbf{Input:} A word graph $\Gamma = (V, E)\in\V$ or $\Gamma = \bot$.
  \\
  \textbf{Output:} A word graph or $\bot$.
  \begin{algorithmic}[1]
    \If{$\Gamma = \bot$}
    \State \Return $\bot$
    \EndIf
    \State $\omega := \bot$
    \For{$(\alpha, a, \beta)\in E$}
    \If{$(w_\alpha a)\theta \neq (w_\beta)\theta$}
    \Comment{Since $\beta$ violates condition (ii) of
      \cref{thm-rees-congruence-condition}, it must be the case that
    $\beta = \omega$}
    \If{$\omega = \bot$}
    \State $\omega\gets \beta$
    \ElsIf{$\omega \neq \beta$}
    \State \Return $\bot$ \Comment{Multiple possibilities for
    $\omega$ detected, contradicting uniqueness}
    \EndIf
    \EndIf
    \EndFor
    \State $E^\prime := E$
    \If{$\omega\neq \bot$}
    \For{$a\in A$}
    \If{$\omega\cdot a = \bot$}
    \State $E^\prime \gets E^\prime \cup\{(\omega, a, \omega)\}$
    \ElsIf{$\omega\cdot a \neq \omega$}
    \State \Return $\bot$ \Comment{$\omega$ is not a sink, so cannot
    satisfy condition (i) of \cref{thm-rees-congruence-condition}}
    \EndIf
    \EndFor
    \EndIf
    \State \Return $\Gamma^\prime = (V, E^\prime)$
  \end{algorithmic}
\end{algorithm}

\begin{lemma}\label{lemma-rees-refiner}
  $\texttt{IsRightReesCongruence}_R$ is a refining function for $\Z_R$.
\end{lemma}
\begin{proof}
  We verify \cref{de-refining-func}. Clearly,
  $\texttt{IsRightReesCongruence}_R(\bot) = \bot$ and so
  \cref{de-refining-func}(i) holds.

  Suppose that $\Gamma\neq \bot$ and
  $\texttt{IsRightReesCongruence}_{R}(\Gamma)= \bot$. Then
  $\texttt{IsRightReesCongruence}_{R}$ returns in line 10 or line 20 of
  \cref{alg-is-right-rees-congruence}.  If $\texttt{IsRightReesCongruence}_{R}$
  returns $\bot$ in line 10, then there exist $(\alpha, a, \beta), (\alpha', a',
  \beta') \in E$ such that $(w_{\alpha}a)\theta \neq (w_{\beta})\theta$ and
  $(w_{\alpha'}a')\theta \neq (w_{\beta})\theta$. If any descendent of
  $\Gamma$ has a unique non-trivial sink $\omega$, then $\omega \neq \beta$ or
  $\omega\neq \beta'$. In particular, \cref{thm-rees-congruence-condition}(ii)
  does not hold, and so $\Reach(\Gamma) \cap \Z_R =
  \varnothing$, and so \cref{de-refining-func}(ii) holds.

  If $\texttt{IsRightReesCongruence}_{R}$ returns $\bot$ in line 20, then there
  exists a unique $\omega \in V$ and $(\alpha, a, \omega)\in E$ such that
  $(w_{\alpha}a)\theta\neq (w_{\omega})\theta$. The node $\omega$ is the unique
  node with this property for every descendent of $\Gamma$ also.
  But, by the condition of line 19, $\omega$ is not a sink in $\Gamma$, and
  hence no descendent of $\Gamma$ contains a unique non-trivial sink. In other
  words, $\Reach(\Gamma)\cap \Z_{R} =\varnothing$, and so
  \cref{de-refining-func}(ii) holds.

  Finally assume that $\texttt{IsRightReesCongruence}_{R}$ returns
  $\Gamma^\prime = (V, E^\prime)$ in line 24 of
  \cref{alg-is-right-rees-congruence}. We must show that $\Reach(\Gamma') \cap
  \Z_R = \Reach(\Gamma)\cap \Z_R$. Since $\Gamma'\in \Reach(\Gamma)$, certainly
  $\Reach(\Gamma')\cap \Z_R \subseteq \Reach(\Gamma)\cap \Z_R$. Suppose that
  $\Gamma''\in \Reach(\Gamma)\cap \Z_R$. The if statement in line 15 implies
  that $\omega$ has been assigned a value in $\Gamma$ and therefore it is the
  unique non-trivial sink of every descendent of $\Gamma$ in $\Z_R$ including
  $\Gamma''$. Since edges are only added to $E^\prime$ in line 18, it follows
  that $\Gamma'' \in \Reach(\Gamma')\cap \Z_R$ and so
  \cref{de-refining-func}(iii) holds.
\end{proof}

It follows from \cref{lemma-rees-refiner} that
$\texttt{BacktrackingSearch}_{\Z_R\cap \X_n}(f, \Xi) = \Z_R\cap\X_n$ where
$f$ is  $\texttt{IsRightReesCongruence}_R\circ
\texttt{MakeCompatibleRepeatedly}_R \circ \texttt{AtMost}_n$.

If $I$ is a right ideal of $M$, then the Rees congruence $\rho_I$ is a 2-sided
congruence if and only if $I$ is a 2-sided ideal of $M$. Hence $\Z_R\cap \Y_R$
is the set of all standard complete word graphs corresponding to Rees
congruences by 2-sided ideals. Combining the criteria of
\cref{thm-rees-congruence-condition} and \cref{thm-2-sided-condition} we can
computationally check if a word graph belongs to $\Z_R\cap \Y_R$. Therefore
$\texttt{BacktrackingSearch}_{\Z_R\cap \Y_R\cap \X_n}(f, \Xi) = \Z_R\cap\Y_R\cap
\X_n$ where $f$ is $\texttt{IsRightReesCongruence}_R\circ
\texttt{TwoSidedMakeCompatibleRepeatedly}\circ
\texttt{MakeCompatibleRepeatedly}_R \circ
\texttt{AtMost}_n$.

% TODO: compute these numbers
%As an example, in \cref{appendix-numbers} we compute the number of Rees
%congruences of the free monoid $A ^ *$ and free semigroup $A ^ +$ when $|A|
%\leq ??$. These numbers reproduce and extend those given
%in~\cite[Appendix B]{Bailey2016aa}.

% TODO(maybe later) specific reference when avaliable
% TODO(maybe later) some example number of what we can compute for
% the free monoid +
% semigroup, when available.
% TODO(maybe later) numbers of nodes searched when applying the
% refining functions above.

%%%%%%%%%%%%%%%%%%%%%%%%%%%%%%%%%%%%%%%%%%%%%%%%%%%%%%%%%%%%%%%%%%%%%%%%
\subsection{Congruences representing faithful actions}
%%%%%%%%%%%%%%%%%%%%%%%%%%%%%%%%%%%%%%%%%%%%%%%%%%%%%%%%%%%%%%%%%%%%%%%%

% TODO:(maybe never) change X to \Omega, check rest of paper
% TODO:(maybe never) change x, y to s,t in rest of paper

If $\Psi: X\times M\to X$ is a (right) monoid action, then every $s\in M$
induces a transformation $\Psi_s: X \to X$ defined by $x \mapsto (x, s)\Psi$.
We say that $\Psi$ is \defn{faithful} if $\Psi_s = \Psi_t$ implies that $s = t$
for all $s, t\in M$. Of course, if $\Psi$ is a faithful right action of $M =
\genset{A}$, then $M$ and $\genset{\Psi_a\mid a \in A}$ are isomorphic monoids.
Several recent papers have studied transformation representations of
various classes of semigroups and monoids;
see~\cite{Araujo2009},~\cite{Cameron2023aa},~\cite{East2024aa},~and~\cite{Margolis2023aa}
and the references therein. The implementations of the algorithms in
this paper were crucial in~\cite{East2024aa} and can be used to
verify, in finitely many cases, some of the results in~\cite{Margolis2023aa}.

Recall that to every complete deterministic word graph
$\Gamma = (V, E)$ compatible with $R$ we associate the right action $\Psi :
V\times M \to V$ given by $(\alpha, (w)\theta)\Psi = \alpha \cdot w$ for all
$\alpha\in V, w\in A^\ast$, where $\theta: A^\ast \to M$ is the unique
homomorphism with $\ker(\theta) = R^\#$.

We require the following theorem.

\begin{thm}[cf. Proposition 1.2 in~\cite{Tully1961aa} or Chapter I, Proposition
  5.24 in~\cite{Kilp2000aa}]\label{thm-faithful}
  Let $M$ be a monoid and let $\rho$ be a right congruence on $M$. Then the
  action of $M$ on $M/\rho$ by right multiplication is faithful if and only if
  the only 2-sided congruence contained in $\rho$ is trivial.
\end{thm}

A 2-sided congruence $\rho$ on a monoid $M$ is \defn{principal} if there exists
$(s, t) \in M \times M$ such that $s\neq t$ and $\rho$ is the least 2-sided
congruence containing $(s, t)$, i.e.\ $\rho = \{(s, t)\}^{\#}$. We also refer
to the pair $(s, t)$ in the preceding sentence as the \defn{generating pair} of
the principal congruence $\rho$. Note that if $\rho$ is a principal 2-sided
congruence, then $\rho\neq \Delta_M$ by definition.

Clearly, every non-trivial 2-sided congruence contains a principal 2-sided
congruence. If $\rho$ is a right congruence not containing any
\textit{principal} 2-sided congruences, then $\rho$ contains no non-trivial
2-sided congruences.  Hence, by \cref{thm-faithful}, $M$ acts faithfully on
$M/\rho$ by right multiplication. This argument establishes one implication of
the next theorem.

\begin{thm}\label{thm-faithful-cong-condition}
  Let $M$ be the monoid defined by the monoid presentation $\langle
  A \mid R\rangle$, let $\theta: A ^ * \to M$ be the natural surjective
  homomorphism,  let $\rho$ be a right congruence on $M$, let $\Gamma
  = (V, E)$ be the word graph of $\rho$, and let $\Psi: V\times M \to V$
  be the associated right action. If $P\subseteq A^* \times A^*$ is such that
  $(P)\theta$ contains a generating pair for every principal 2-sided
  congruence of $M$, then $\Psi$ is faithful if and only if
  for all $(u, v)\in P$ there exists some $\alpha\in V$ with $\alpha
  \cdot u \neq \alpha \cdot v$.
\end{thm}
\begin{proof}
  ($\Leftarrow$)
  If $(u, v)\in P$ is arbitrary, then by assumption there exists some
  $\alpha\in V$ such that $\alpha\cdot u \neq \alpha\cdot v$. If
$\{(u)\theta, (v)\theta)\}^{\#} \subseteq \rho$, then $((w_\alpha
u)\theta, (w_\alpha v)\theta)\in \rho$ and by
\cref{lem-rho-theta-rho-ker} we have that $(w_\alpha u, w_\alpha
v)\in (0)\pi_\Gamma$. But then $\alpha \cdot u  = 0\cdot w_\alpha u =
0\cdot w_\alpha v = \alpha \cdot v$, a contradiction.
Hence $\rho$ contains no
principal 2-sided congruences, and this implication follows by the
argument before the theorem.

($\Rightarrow$)
We prove the contrapositive.
Assume that there exists $(u, v)\in P$ such that
$\alpha \cdot u = \alpha \cdot v$ for all $\alpha \in V$. Then by
definition $(\alpha, (u)\theta)\Psi = (\alpha,
(v)\theta)\Psi$ for all $\alpha\in V$. In other words, $\Psi_{(u)\theta} =
\Psi_{(v)\theta}$, and  since $(u)\theta \neq (v)\theta$, $\Psi$ is not
faithful.
\end{proof}

The condition of \cref{thm-faithful-cong-condition} can only be tested if it is
possible to compute $P$. On the other hand, only finite monoids can have a
faithful action on a finite set. Therefore it is only meaningful to look for
faithful finite index congruences of finite monoids, in which case if the set
$P$ can be computed, then the condition of \cref{thm-faithful-cong-condition}
can be tested. In practice there are two ways of computing a set $P$ of
generating pairs for the principal congruences of $M$. The first is
computational (such as the approach described in \cref{section-algo-2}); the
second is mathematical: if a description of the lattice of congruences of a
particular monoid $M$ is known (such as those given in~\cite{East2019aa}), then
we can obtain a set $P$ from this description directly.

Let $\texttt{IsFaithfulRightCongruence}_P :
\V\cup\{\bot\}\rightarrow \V\cup\{\bot\}$ be given by
\[
\texttt{IsFaithfulRightCongruence}_P(\Gamma) =
\begin{cases}
  \bot & \text{if }\Gamma = \bot\\
  \bot &\text{if }\Gamma = (V, E) \text{ and there exists } (u,v)\in P\\
  & \text{ such that }
  \alpha \cdot u = \alpha\cdot v\neq \bot \text{ for all }
  \alpha\in V\\
  \Gamma &\text{otherwise}
\end{cases}
\]
where $P$ is a set of generating pairs for the principal congruences
of $M$. That $\texttt{IsFaithfulRightCongruence}_P$ is a refining
function for the set of standard word graphs corresponding to faithful right
congruences can be easily verified.

The primary application of faithful right congruences is in finding small
transformation representations of finite monoids. Of course, every such monoid
$M$ has a transformation representation of degree $|M|$; we refer to this as
the \defn{right regular representation} of $M$. However, computing this
representation requires computing the action of $M$ on itself by right
multiplication, which requires all of the elements of $M$ to be stored in
memory. This is only feasible when $|M|$ is relatively small; see~\cite[Section
1]{East2019aa} for a more detailed discussion. If a presentation for $M$ is
known, then iterating through the faithful right congruences of index at most
$|M|$ using $\texttt{BacktrackingSearch}$ we can find the smallest
transformation
representation arising from a right congruence. This is often quite slow in
practice especially when $|M|$ is large. This method can be improved by
modifying $\texttt{BacktrackingSearch}$ to return the first faithful right
congruence with index at most $|M|$. If a faithful right congruence  of index
$n$ is known, then we call $\texttt{BacktrackingSearch}$ to find the first
faithful right congruence with index at most $n - 1$. If no such faithful right
congruence exists, then $n$ is the minimum index of any faithful right
congruence on $M$. Otherwise, if a faithful right congruence with index $m < n$
is found, then we repeat this process. This is implemented in
\libsemigroups, and has been successfully employed to find transformation
representations of relatively small degree for several classes of monoids where
the best previous known bounds were $|M|$; see~\cite{East2024aa} for
more details.

We provide two examples where the algorithm presented in this section
returns minimal transformation representations.

\begin{ex}
Let $S_{\sigma}$ be the Rees matrix semigroup over the symmetric group of degree
$4$ with matrix:
\[
  \begin{bmatrix}
    \sigma & \id \\
    \id & \id
  \end{bmatrix},
\]
where $\id$ is the identity permutation, and $\sigma\in\{(1\
2), (1\ 2\ 3), (1\ 2\ 3\ 4)\}$. Then the minimum degree transformation
representation of $S_{(1\ 2)}$, $S_{(1\ 2\ 3)}$, and $S_{(1\ 2\ 3\ 4)}$
found by the algorithm described in this section are $6$, $7$, and $8$. This
agrees with the minimum degree transformation representation
from~\cite[Theorem 2.19]{Margolis2023aa}.
\end{ex}

\begin{ex}
The \defn{opposite} $S ^ {op}$ of a semigroup $S$ has multiplication $*$ defined
by
\[
  x * y = yx
\]
where the multiplication on the right hand side of the equality is the
multiplication in $S$. It is shown in~\cite[Theorem 2.2]{Margolis2023aa}
that the minimum degree transformation of the opposite $T_n^{op}$ of the
full transformation monoids $T_n$ of degree $n$ is $2 ^ n$; and this agrees
with the output of the algorithm in this section.
\end{ex}

As expected, the algorithm in this section does not always return the minimum
degree transformation representation, since such a representation does not
always correspond to an action on the classes of a right congruence; see, for
example,~\cref{fig-non-cyclic-monoid-action} and~\cite[Table 3.1 and Theorem
3.9]{Cameron2023aa}.

In~\cite{Schein1992aa}, the minimal degree partial permutation representation of
an arbitrary inverse semigroup $S$ is given in terms of the minimal degrees of
certain subgroups of $S$. This is implemented in the \Semigroups package for
\GAP. In the example of the dual symmetric inverse monoid $I_n^*$, the output of
the algorithm described in this section as an application of the low-index
congruences algorithm when $n \leq 6$ agrees with the theoretical
minimum found in~\cite{Maltcev2007aa} of $2 ^ n - 1$,
and with the output of Schein's algorithm from~\cite{Schein1992aa}.

Of course, as a consequence of \cref{prop-word-graph-iff-action} and
\cref{thm-graph-to-right-congruence}, not every faithful action arises from a
faithful right congruence. However, since \cref{prop-word-graph-iff-action}
does give a bijection between word graphs and right actions, it may be possible
to extend the backtracking search over standard word graphs given in
this paper to
cover a wider class of word graphs, and adapt the criterion given in
\cref{thm-rees-congruence-condition} to word graphs representing faithful
actions in general. This would give an algorithm for finding minimal
transformation representations of finite monoids.

% TODO add some numbers

% TODO(later or never):
% -- maximal congruences (check there are no non-trivial congruences)
% -- lattice height (using the approach in the lattice heights paper)
% -- lattice width (no real idea how to do this)

%%%%%%%%%%%%%%%%%%%%%%%%%%%%%%%%%%%%%%%%%%%%%%%%%%%%%%%%%%%%%%%%%%%%%%%%
\section{Meets and joins for congruences represented by word
graphs}\label{section-joins-meets}
%%%%%%%%%%%%%%%%%%%%%%%%%%%%%%%%%%%%%%%%%%%%%%%%%%%%%%%%%%%%%%%%%%%%%%%%

Although \texttt{AllRightCongruences} can be used to find all of the 1-sided
congruences of a monoid, to compute the lattice of such congruences we require
a mechanism for computing joins or meets of congruences represented by word
graphs.   In this section we will show that the Hopcroft-Karp
Algorithm~\cite{Hopcroft1971aa}, for checking whether two finite state
automata recognise the same language, can be used to determine the join of two
congruences represented by word graphs. We will also show that the standard
construction of an automaton recognising the union of two regular languages can
be used to compute the meet of two congruences represented by word graphs.

\subsection{The Hopcroft-Karp Algorithm for joins}

In this section we present \cref{join} which can be used to compute the join of
two congruences represented by word graphs. This is a slightly modified version
of the Hopcroft-Karp Algorithm for checking whether two finite state
automata recognise the same language. The key difference between \cref{join}
and the Hopcroft-Karp Algorithm is that the inputs are word graphs rather than
automata. As mentioned in the introduction a word graph is essentially an
automaton without initial or accept states. The initial states of the automata
that form the input to the Hopcroft-Karp Algorithm are only used at the
beginning of the algorithm (when the pair consisting of the start states of the
two automata are pushed into the stack).  In our context, the node $0$ will
play the role of the start state. Similarly the accept states are only used at
the last step of the algorithm. As such the only difference between the
Hopcroft-Karp Algorithm described in~\cite{Hopcroft1971aa}
(and~\cite{Norton2009aa}) and the version here are: the inputs, the values used
to initialise the stack, and the return value. The time complexity of
\texttt{JoinWordGraphs} is $O(|A|n)$ where $n = \max\{|V_0|, |V_1|\}$.

\cref{join} makes use of the well-known disjoint sets data structure
(originating in~\cite{Galler1964aa}, see also~\cite[Chapter 21]{Cormen2009aa})
which provides an efficient representation of a partition of $V_0\sqcup V_1$.
We identify the disjoint sets data structure for a partition $\kappa$ of
$V_0\sqcup V_1$ and the partition itself. If $\kappa$ denotes a disjoint sets
data structure, then we require the following related functions:
\begin{description}
\item[\texttt{Union}$(\kappa, \alpha, \beta)$:] unites the parts of the
  partition $\kappa$ containing $\alpha, \beta\in V_0\sqcup V_1$;
\item[\texttt{Find}$(\kappa, \alpha)$:] returns a canonical
  representative of the part
  of the partition $\kappa$ containing $\alpha$.
\end{description}

Note that $\alpha, \beta \in V_0\sqcup V_1$ belong to the same part of the
partition represented by our disjoint sets data structure if and only if
$\texttt{Find}(\alpha) = \texttt{Find}(\beta)$. If $\kappa$ is an equivalence
relation on the nodes of a word graph $\Gamma = (V, E)$, then we define the
\defn{quotient $\Gamma/\kappa$ of  $\Gamma$ by $\kappa$} to be the word graph
(after an appropriate relabelling) with nodes
$\set{\alpha/\kappa}{\alpha \in V}$
and edges
$\set{(\alpha /\kappa, a, \beta/\kappa)}{(\alpha, a, \beta) \in E}$.
Note that $\Gamma/\kappa$ is not necessarily deterministic.

\begin{algorithm}
\caption{- \texttt{JoinWordGraphs}}
\label{join}
\textbf{Input:} Word graph $\Gamma_0= (V_0, E_0)$ and $\Gamma_1 = (V_1, E_1)$
representing right congruences $\rho_0$ and $\rho_1$ of the monoid $M$.\\
\textbf{Output:} The word graph of the join $\rho_0\vee \rho_1$ of
$\rho_0$ and $\rho_1$.
\begin{algorithmic}[1]
  \State Let $\kappa$ be the disjoint sets data structure for $\Delta_{V_0\sqcup
  V_1} = \set{(\alpha, \alpha)}{\alpha \in V_0\sqcup V_1}$
  \State Push $(0_{\Gamma_0}, 0_{\Gamma_1})$ onto the stack $S$
  \State \texttt{Union}$(\kappa, 0_{\Gamma_0}, 0_{\Gamma_1})$
  \While{$S \neq \varnothing$}
  \State Pop $(\alpha_0, \alpha_1) \in V_0\times V_1$ from the stack
  \For{$a \in A$}
  \State Let $\alpha_0'\in V_0$ and $\alpha_1'\in V_1$ be such that
  $(\alpha_0, a, \alpha_0')\in E_0$ and $(\alpha_1, a, \alpha_1')\in E_1$.
  \State Let $\gamma_0 =\texttt{Find}(\kappa, \alpha_0')$ and $\gamma_1 =
  \texttt{Find}(\kappa, \alpha_1')$
  \If{$\gamma_0\neq \gamma_1$}
  \State Push $(\gamma_0, \gamma_1)$ onto the stack $S$
  \State \texttt{Union}$(\kappa, \gamma_0, \gamma_1)$
  \EndIf
  \EndFor
  \EndWhile
  \State \Return $(\Gamma_0 \sqcup \Gamma_1) / \kappa$.
\end{algorithmic}
\end{algorithm}

In order to prove the correctness of~\texttt{JoinWordGraphs} we need the
following definition. An equivalence relation $\kappa$ over $V_0 \sqcup V_0$ is
called \defn{right invariant} if for all $a \in A$, $(\alpha, \beta) \in
\kappa$ implies $(\alpha', \beta') \in \kappa$, where $\alpha'\in V_0$ and
$\beta'\in V_1$ are such that $(\alpha, a, \alpha')\in E_0$ and $(\beta, a,
\beta')\in E_1$. The following result is Lemma 1 in~\cite{Hopcroft1971aa}.

\begin{lemma}\label{lemma-HK}
Let $\kappa$ be the equivalence relation on $V_0 \sqcup V_1$ in line 15 of
\texttt{JoinWordGraphs}. Then $\kappa$ is the least right invariant
equivalence relation on $V_0\sqcup V_1$ containing $(0_{\Gamma_0},
0_{\Gamma_1})$.
\end{lemma}

The next result relates right invariant equivalences on a word graph to its
deterministic quotients.

\begin{lemma}\label{lemma-invariant-kernel}
If $\Gamma = (V, E)$ is a complete word graph and $\kappa$ is an equivalence
relation on $V$, then $\Gamma / \kappa$ is deterministic if and only if
$\kappa$ is right invariant.
\end{lemma}

\begin{proof}
($\Rightarrow$)
Suppose that $\kappa$ is an equivalence relation on $V$ such that
$\Gamma/\kappa$ is deterministic. Suppose that $a \in A$ and that
$\alpha,\beta \in V$ are arbitrary. We will prove that $(\alpha, \beta) \in
\kappa$ implies that $(\alpha', \beta') \in \kappa$ where $(\alpha, a,
\alpha'), (\beta, a, \beta')\in E$ are the unique edges labelled by $a$ with
sources $\alpha$ and $\beta$.

Assume that $\alpha, \beta\in V$ are such that $(\alpha, \beta) \in
\kappa$. By the definition, $(\alpha /\kappa, a, \alpha'/\kappa)$ and
$(\beta/\kappa, a, \beta'/\kappa)$ are edges in $\Gamma/\kappa$. Since
$\Gamma/\kappa$ is deterministic and $\alpha/\kappa = \beta/\kappa$, it
follows that $\alpha'/\kappa = \beta'/\kappa$ and so $(\alpha', \beta') \in
\kappa$, as required.

($\Leftarrow$)
Conversely, let $\kappa$ be a right invariant equivalence relation on $V$.
The word graph $\Gamma/\kappa$ is complete by definition. If
$(\alpha_0/\kappa, a, \beta_0/\kappa)$ and $(\alpha_1/\kappa, a,
\beta_1/\kappa)$ are edges in $\Gamma/\kappa$ such that $(\alpha_0,
\alpha_1)\in \kappa$, then, since $\kappa$ is right invariant $(\beta_0,
\beta_1) \in \kappa$.  It follows that $\Gamma/\kappa$ is deterministic.
\end{proof}

We can now show that the quotient word graph returned by
\texttt{JoinWordGraphs} represents the join of the congruences represented by
the input word graphs.

\begin{prop}\label{prop-joins}
The word graph $(\Gamma_0 \sqcup \Gamma_1)/ \kappa$ returned by
\texttt{JoinWordGraphs} in \cref{join} is the graph of the join $\rho_0\vee
\rho_1$ of congruences $\rho_0$ and $\rho_1$ represented by $\Gamma_0$ and
$\Gamma_1$ respectively.
\end{prop}

\begin{proof}
Let $\Gamma_0 = (V_0, E_0)$ and $\Gamma_1 = (V_1, E_1)$. We start by showing
that if $\kappa$ is a right invariant equivalence relation on $V_0 \sqcup V_1$
containing $(0_{\Gamma_0}, 0_{\Gamma_1})$, then the quotient graph
$(\Gamma_0 \sqcup \Gamma_1)/ \kappa$ represents a right congruence $\tau$ of
$M$ containing both $\rho_0$ and $\rho_1$. To show that $(\Gamma_0 \sqcup
\Gamma_1)/ \kappa$ represents a right congruence of the monoid $M$ defined by
the presentation $\langle A \mid R\rangle$ it suffices by
\cref{thm-graph-to-right-congruence} to show
that $(\Gamma_0 \sqcup \Gamma_1)/ \kappa$ is complete, deterministic,
compatible with $R$, and every node is reachable from
$0_{\Gamma_0}/\kappa$.

Since $\Gamma_0$ and $\Gamma_1$ are complete graphs, any quotient of
$\Gamma_0 \sqcup \Gamma_1$ is also complete. By
\cref{lemma-invariant-kernel}, $(\Gamma_0 \sqcup \Gamma_1)/ \kappa$ is
deterministic.
Since both $\Gamma_0$ and $\Gamma_1$ are compatible
with $R$, it follows that $\Gamma_0\sqcup \Gamma_1$ is also compatible with $R$.
Suppose that $\phi: \Gamma_0 \sqcup \Gamma_1 \to
(\Gamma_0 \sqcup \Gamma_1)/ \kappa$ is the natural word graph homomorphism
with $\ker(\phi) = \kappa$. Then, since word graph homomorphisms preserve
paths, it follows that $\im \phi =  (\Gamma_0 \sqcup \Gamma_1)/ \kappa$ is
compatible with $R$ too.
By assumption, every node in $V_0$ is reachable from
$0_{\Gamma_0}$ and every node in $V_1$ is reachable from $0_{\Gamma_1}$ in
${\Gamma_0 \sqcup \Gamma_1}$, and $(0_{\Gamma_1}, 0_{\Gamma_2})\in \kappa$.
Thus every node in $(\Gamma_0 \sqcup \Gamma_1)/\kappa$ is reachable from
$0_{\Gamma_1}/\kappa$. It follows that $(\Gamma_0 \sqcup \Gamma_1)/\kappa$
represents a right congruence $\tau$ on $M$.
Again, since the homomorphism $\phi: \Gamma_0 \sqcup \Gamma_1 \to (\Gamma_0
\sqcup \Gamma_1)/ \kappa$ preserves paths, and by
\cref{thm-graph-to-right-congruence}, the path relations $\rho_0 =
(0_{\Gamma_0})\pi_{\Gamma_0}$ and $\rho_1 = (0_{\Gamma_1})\pi_{\Gamma_1}$ on
$\Gamma_0$ and $\Gamma_1$, respectively, are contained in the path relation
$\tau = (0_{\Gamma_0}/\kappa)\pi_{(\Gamma_0\sqcup \Gamma_1)/\kappa}$.

Suppose that $\Gamma_2$ is the word graph of $\rho_0\vee \rho_1$. Then, since
$\rho_0\subseteq \rho_0\vee \rho_1$ and $\rho_1\subseteq \rho_0\vee \rho_1$,
by \cref{lemma-kernel-right-invariant}, there exist
unique word graph homomorphisms $\theta_0: \Gamma_0 \to \Gamma_2$, and
$\theta_1:\Gamma_1 \to \Gamma_2$ such that $(0_{\Gamma_0})\theta_0 =
(0_{\Gamma_1})\theta_1 = 0_{\Gamma_2}$.
Clearly, $\theta_{01}: \Gamma_0\sqcup \Gamma_1 \to
\Gamma_2$ defined by $(\alpha_{\Gamma_i})\theta_{01} =
(\alpha_{\Gamma_i})\theta_i$ for $i\in \{0, 1\}$ is a word graph homomorphism
where $(0_{\Gamma_0})\theta_{01} = (0_{\Gamma_1})\theta_{01} = 0_{\Gamma_2}$,
and this homomorphism is unique by \cref{corollary-homomorphism}.
Since $\rho_0\subseteq \tau$ and $\rho_1\subseteq \tau$, and $\tau$ is a
right congruence, it follows that $\rho_0\vee \rho_1 \subseteq \tau$. Hence,
again by \cref{lemma-kernel-right-invariant}, there exist a unique word graph
homomorphism $\theta_2: \Gamma_2 \to (\Gamma_0 \sqcup \Gamma_1)/\kappa$ such
that $(0_{\Gamma_2})\theta_2 = 0_{\Gamma_0}/\kappa$; see
\cref{fig-comm-diagram}. We will show that $\theta_{01} = \phi$.

\begin{figure}
  \centering
  \begin{tikzcd}
    \Gamma_0\sqcup \Gamma_1 \arrow[r, "\phi"] \arrow[d, "\theta_{01}"] &
    (\Gamma_0\sqcup \Gamma_1) / \kappa \\
    \Gamma_2\arrow[ru, "\theta_2"]
  \end{tikzcd}
  \caption{The commutative diagram from the proof of \cref{prop-joins}.}
  \label{fig-comm-diagram}
\end{figure}

Since the composition of word graph homomorphisms is a word graph
homomorphism, it follows that $\theta_{01}\circ \theta_2: \Gamma_0\sqcup
\Gamma_1\to (\Gamma_0\sqcup \Gamma_1)/\kappa$ is a word graph homomorphism
and
\[
  (0_{\Gamma_0})\theta_{01}\theta_2
  = (0_{\Gamma_1})\theta_{01}\theta_2
  = (0_{\Gamma_2})\theta_2
  = 0_{\Gamma_0}/ \kappa.
\]
But $\phi: \Gamma_0\sqcup \Gamma_1 \to (\Gamma_0\sqcup \Gamma_1)/\kappa$ is
also a word graph homomorphism with $(0_{\Gamma_0})\phi =
(0_{\Gamma_1})\phi= 0_{\Gamma_0}/ \kappa$, and so $\phi = \theta_{01}\circ
\theta_2$ by \cref{corollary-homomorphism}. In particular,
$\ker(\theta_{01}) \subseteq \ker(\phi) = \kappa$.  Since $\Gamma_2 =
(\Gamma_0\sqcup \Gamma_1)/\ker(\theta_{01})$ is a word graph representing a
right congruence, it is deterministic. Hence, by
\cref{lemma-invariant-kernel}, it follows that $\ker(\theta_{01})$ is right
invariant. Also $(0_{\Gamma_0})\theta_{01} = (0_{\Gamma_1})\theta_{01}$
implies that $(0_{\Gamma_0}, 0_{\Gamma_1}) \in \ker(\theta_{01})$. But
$\kappa$ is the least right invariant equivalence relation on $V_0\sqcup
V_1$ containing $(0_{\Gamma_0}, 0_{\Gamma_1})$, by \cref{lemma-HK}, and so
$\ker(\theta_{01}) = \ker(\phi) = \kappa$. It follows that $\Gamma_2$ and
$(\Gamma_0\sqcup \Gamma_1) / \kappa$ coincide, and so $(\Gamma_0\sqcup
\Gamma_1) / \kappa$ represents $\rho_0\vee \rho_1$, as required.
\end{proof}

\subsection{Automata intersection for meets}

In this section we present a slightly modified version of a standard
algorithm from automata theory for finding an automaton recognising the
intersection of two regular languages.  As in the previous section the key
difference is that we use word graphs rather than automata.

If $\Gamma_0=(V_0, E_0)$ and $\Gamma_1=(V_1, E_1)$ are word graphs over the
same alphabet $A$, then we define a word graph $\Gamma_2 =  (V_2, E_2)$ where
$V_2$ is the largest subset of $V_0\times V_1$ such that every node in $V_2$ is
reachable from $(0, 0)$ in $\Gamma_2$ and
$((\alpha_0, \alpha_1), a, (\beta_0, \beta_1)) \in E_2$ if and only if
$(\alpha_0, a, \beta_0) \in E_0$ and $(\alpha_1, a, \beta_1) \in E_1$; we will
refer to $\Gamma_2$ as the \defn{meet word graph} of $\Gamma_0$ and $\Gamma_1$.
This is directly analogous to the corresponding construction for automata; for
more details see the comments following the proof of Theorem 1.25
in~\cite{Sipser2013aa}.

\begin{prop}\label{meet-graph}
If $\Gamma_0=(V_0, E_0)$ and $\Gamma_1=(V_1, E_1)$ are word graphs
representing congruences of the monoid $M$ defined by the presentation
$\langle A \mid R\rangle$, then the meet word graph $\Gamma_2$ of $\Gamma_0$ and
$\Gamma_1$ is a complete deterministic word graph which is compatible with
$R$ and each node of $\Gamma_2$ is reachable from $(0, 0)$.
\end{prop}

\begin{proof}
The word graph $\Gamma_2$ is complete and deterministic since $\Gamma_0$ and
$\Gamma_1$ are, and $\Gamma_2$ was constructed so that every node is
reachable from $(0, 0)$.

It remains to prove that $\Gamma_2$ is compatible with the set of relations
$R$. If $w\in A ^ *$, then, since $\Gamma_0$ and $\Gamma_1$ are
complete, then for $i \in \{0, 1\}$ and each node $\alpha$ in
$\Gamma_i$, there is a unique path in $\Gamma_i$ labeled by $w$ with
source $\alpha$. If $w\in A ^*$ labels a
$(\alpha_0, \beta_0)$-path in $\Gamma_0$ and an $(\alpha_1, \beta_1)$-path in
$\Gamma_1$, then $w$ labels a $((\alpha_0, \alpha_1), (\beta_0,
\beta_1))$-path in $\Gamma_2$. Since $\Gamma_0$ is compatible with $R$, for
every $(u, v) \in R$ and for every $\alpha_0\in V_0$ there exists a
$\beta_0\in V_0$ such that both $u$ and $v$ label $(\alpha_0, \beta_0)$-paths
in $\Gamma_0$. Similarly, for every $\alpha_1\in V_1$ there is $\beta_1\in
V_1$ such that $u$ and $v$ both label $(\alpha_1, \beta_1)$-paths in
$\Gamma_1$. Hence for every $(u, v) \in R$ and every $(\alpha_0, \alpha_1)\in
V_2$ there exists $(\beta_0, \beta_1)\in V_2$ such that both $u$ and $v$
label $((\alpha_0, \alpha_1), (\beta_0, \beta_1))$-paths in $\Gamma_2$, and
so $\Gamma_2$ is compatible with $R$ also.
\end{proof}

\begin{cor}\label{is-meet}
If $\Gamma_0=(V_0, E_0)$ and $\Gamma_1=(V_1, E_1)$ are word graphs
representing congruences $\rho_0$ and $\rho_1$, respectively, of the monoid
$M$ defined by the presentation $\langle A \mid R\rangle$, then the meet word
graph $\Gamma_2$ of $\Gamma_0$ and $\Gamma_1$ represents the meet
$\rho_0\wedge \rho_1$ of $\rho_0$ and $\rho_1$.
\end{cor}

\begin{proof}
By~\cref{meet-graph}, $\Gamma_2$ represents a right congruence $\tau$ on
$M$.
If $(\alpha, \beta) \in V_2$, then by the comments
after~\cite[Theorem 1.25]{Sipser2013aa}
a word $w\in A ^ *$ labels a $((0, 0), (\alpha, \beta))$-path in $\Gamma_2$
if and only if $w$ labels a $(0, \alpha)$-path in $\Gamma_0$ and a $(0,
\beta)$-path in $\Gamma_1$. In other words, if $\pi_{\Gamma_i}$ denotes the
path relation on $\Gamma_i$ for $i\in \{0, 1, 2\}$ and $u, v\in A ^ *$ are
arbitrary, then
$(u, v) \in (0)\pi_{\Gamma_2}$ if and only if $(u, v) \in (0)\pi_{\Gamma_0}
\cap (0)\pi_{\Gamma_1}$ if and only if $(u/R^ {\#}, v/R ^ {\#})\in
\rho_0\wedge \rho_1$, as required.
\end{proof}

The algorithm \texttt{MeetWordGraphs} in \cref{meet} can be used to
construct the meet word graph $\Gamma_2$ from the input word graphs
$\Gamma_0$ and $\Gamma_1$. \cref{meet} starts by constructing a graph
with nodes that have the form $(\alpha, \beta, \gamma)$ for $\alpha,
\beta, \gamma \in \mathbb{N}$. These nodes get relabeled at the last
step of the procedure and a standard word graph is returned by
\texttt{MeetWordGraphs}. The triples $(\alpha, \beta, \gamma)$
belonging to $V$ in \texttt{MeetWordGraphs}, then the first component
corresponds to a node in $\Gamma_0$, the second component to a node
in $\Gamma_1$, and the third component labels the corresponding node
in the meet word graph returned by \cref{meet}.

\begin{algorithm}
\caption{- \texttt{MeetWordGraphs}}
\label{meet}
\textbf{Input:} Word graphs $\Gamma_0 = (V_0, E_0)$ and
$\Gamma_1 = (V_1, E_1)$ corresponding to the right congruences $\rho_0$ and
$\rho_1$ of the monoid $M$. \\
\textbf{Output:} The word graph of the meet $\rho_0\wedge \rho_1$ of
$\rho_0$ and $\rho_1$
\begin{algorithmic}[1]
  \State $V = \{(0, 0, 0)\}, E = \varnothing, n = 0$
  \For{$(\alpha_0, \alpha_1, \beta) \in V$}
  \For{$a \in A$}
  \State Let $\alpha_0'\in V_0, \alpha_1'\in V_1$ be such that
  $(\alpha_0, a, \alpha_0')\in E_0$ and $(\alpha_1, a, \alpha_1')\in E_1$
  \If{there exists $\beta'$ such that $(\alpha_0', \alpha_1', \beta') \in V$}
  \State  $E \gets E \cup ((\alpha_0, \alpha_1, \beta), a,
  (\alpha_0', \alpha_1', \beta'))$
  \Else
  \State $n \gets n+1$
  \State $V \gets V \cup (\alpha_0', \alpha_1', n)$
  \State  $E \gets E \cup ((\alpha_0, \alpha_1, \beta), a,
  (\alpha_0', \alpha_1', n))$
  \EndIf
  \EndFor
  \EndFor
  \State Relabel every $(\alpha_0, \alpha_1, \beta) \in V$ to $\beta$
  \State \Return $(V, E)$
\end{algorithmic}
\end{algorithm}

In comparison to the procedure for the construction of the intersection
automaton in~\cite[Theorem 1.25]{Sipser2013aa},
\texttt{MeetWordGraphs} includes a few extra steps so
that the output word graph is standard.

\begin{prop}\label{standard-meet}
The output of~\texttt{MeetWordGraphs} in \cref{meet} is a complete
deterministic standard word graph which is compatible with $R$ and
each node of this word graph is reachable from 0.
\end{prop}

\begin{proof}
Let $\Gamma_2$ denote the word graph obtained in \cref{meet} before the
relabelling of the nodes in line 14. It is clear that this word graph is
isomorphic to the meet word graph of $\Gamma_0$ and $\Gamma_1$ and hence it
follows by \cref{meet-graph} that it is complete, deterministic, compatible
with $R$ and each node in $\Gamma_2$ is reachable from 0. For every node
$(\alpha, \beta, \gamma)$ the edge leaving $(\alpha, \beta, \gamma)$ labelled
by $a$ is defined before the edge labelled by $b$ whenever $a < b$. In
addition, all edges with source $(\alpha_0, \beta_0, \gamma)$ have been
defined before any edge with source  $(\alpha_1, \beta_1, \gamma + 1)$ is
defined. It follows that the output word graph after the relabelling in line
14 is standard.
\end{proof}

\section{Algorithm 2: principal congruences and joins}
\label{section-algo-2}
In the preceding sections we described algorithms that permit the computation
of the lattice of right or left congruences of a finitely presented semigroup
or monoid. In this section we consider an alternative method for computing the
lattice of 1-sided or 2-sided congruences of a finite monoid.

If $M$ is a finite monoid, then the basic outline of the method considered in
this section is: to compute the set of all principal congruences of $M$; and
then compute all possible joins of the principal congruences. This approach has
been considered by several authors; particularly relevant here
are~\cite{Araujo2022aa, Freese2008aa, Torpey2019aa}. To compute the set of all
principal congruences of a monoid $M$, it suffices to compute  the principal
congruence generated by $(x, y)$ for every $(x, y) \in M \times M$, and to find
those pairs generating distinct congruences. This requires the computation of
$O(|M| ^ 2)$ principal congruences. In this part of the paper, we describe a
technique based on~\cite{East2019aa} for reducing the number of pairs $(x,
y)\in M \times M$ that must be considered. More specifically, we perform a
preprocessing step where some pairs that generate the same congruence are
eliminated. The time taken by this preprocessing step is often insignificant
when compared to the total time required to compute the lattice, as mentioned
above, in many cases results in a significant reduction in the number of
principal congruences that must be computed, and in the overall time to compute
the lattice; see \cref{appendix-benchmarks} for more details. Of course, there
are also examples where the preprocessing step produces little or no reduction
in the number of principal congruences that must be computed, and so increases
the total time required to compute the lattice (the Gossip
monoids~\cite{Brouwer2015aa} are class of such examples).

% We also investigate the second step of the algorithm outlined above: the
% computation of the joins of the principal congruences. The lattice of
% congruences $\mathcal{L}(M)$ of a monoid is itself a monoid (even a
% commutative monoid of idempotents) under the operation of joins; this lattice
% is generated (as a monoid) by the principal congruences on $M$. As such it is
% possible to apply standard algorithms from computational semigroup theory to
% compute the lattice of congruences from its generating set of principal
% congruences. In the literature, and existing implementations, the standard
% mechanism for determining all of the joins is a simple brute force
% enumeration. In some cases, when $M$ is large enough, the number of principal
% congruences is relatively small, and so too is the size of the generated
% lattice, the computation of the lattice can be significantly improved by using
% the more sophisticated algorithm of Froidure and Pin~\cite{Froidure1997aa}.

Recall that a \defn{submonoid} $N$ of a monoid $M$ is a subset of $M$, which
is a monoid under the same operation as $M$ with the same identity element; we
denote this by $N\leq M$. If $X \subseteq M$, then the least submonoid of $M$
containing $X$ is called the \defn{submonoid generated} by $X$ and is denoted
by $\langle X \rangle$.

For the remainder of the paper, we suppose that $U$ is a monoid, $M$ is a
submonoid of $U$, and $N$ is a submonoid of $M$. We also denote the identity
element of any of these 3 monoids by $1$. Recall that $M$ is \defn{regular} if
for every $x\in M$ there exists some $x'\in M$ such that $xx'x = x$.

In the case that $U$ is regular, in~\cite{East2019aa}, it was shown, roughly
speaking, how to utilise the Green's structure of $U$ to determine the
Green's structure of $M$. The idea being that the structure of $U$ is
``known'', in some sense. The prototypical example is when $M$ is the full
transformation monoid $T_n$ (see~\cite{Linton1998aa}), consisting of all
transformations on the set $\{0, \ldots, n - 1\}$  for some $n$. In this case,
the Green's structure of $U$ is known and can be used to compute the Green's
structure of any submonoid of $U$. In this part of the paper, we will show that
certain results from~\cite{East2019aa} hold in greater generality, for the
so-called, relative Green's relations.

We say that $x, y\in M$ are $\mathscr{L} ^ {M, N}$-related if the sets $Nx =
\set{nx}{n\in N}$ and $Ny$ coincide; and we refer to $\mathscr{L} ^
{M, N}$ as the \defn{relative Green's $\mathscr{L} ^ {M, N}$-relation
on $M$}. The \defn{relative Green's $\mathscr{R} ^ {M, N}$-relation}
on $M$ is defined dually. We say that $x, y\in M$ are $\mathscr{J} ^
{M, N}$-related if the sets $NxN = \set{nxn}{n\in N}$ and $NyN$
coincide; and we refer to $\mathscr{J} ^ {M, N}$ as the
\defn{relative Green's $\mathscr{J} ^ {M, N}$-relation on $M$}. When
$N = M$, we recover the
classical definition of Green's relations, which are ubiquitous in
the study of semigroups and monoids. For further information about
Green's relations see~\cite{Howie1995aa}. Relative Green's relations
were first introduced in~\cite{Wallace1962aa,Wallace1963aa} and
further studied in~\cite{Cain2012aa,
Gray2008aa}. It is routine to show that each of $\mathscr{L} ^ {M, N}$,
$\mathscr{R} ^ {M, N}$, and $\mathscr{J} ^ {M, N}$ is an equivalence relation
on $M$; and that $\mathscr{L} ^ {M, N}$ is a right congruence, and $\mathscr{R}
^ {M, N}$ a left congruence. We denote the $\mathscr{L} ^ {M, N}$-,
$\mathscr{R} ^ {M, N}$-, or $\mathscr{J} ^ {M, N}$-class of an element $x\in M$
by $L_x^{M, N}$, $R_x^{M, N}$ and $J_x^{M, N}$, respectively.

It may be reasonable to ask, at this point, what any of this has to do with
determining the principal congruences of a monoid? This is addressed in the
next proposition.

\begin{prop}\label{prop-relative-greens-cong}
Let $M$ be a monoid and let $\Delta_M = \{(m, m) : m\in M\}\leq M \times M$.
Then the following hold:
\begin{enumerate}[\rm (i)]
  \item\label{prop-relative-greens-cong-i}
    If $(x_0, y_0) \mathscr{R} ^ {M\times M, \Delta_M} (x_1, y_1)$, then
    the  right congruences generated by $(x_0, y_0)$ and $(x_1, y_1)$
    coincide;
  \item\label{prop-relative-greens-cong-ii}
    If $(x_0, y_0) \mathscr{L} ^ {M\times M, \Delta_M} (x_1, y_1)$, then
    the  left congruences generated by $(x_0, y_0)$ and $(x_1, y_1)$
    coincide;
  \item\label{prop-relative-greens-cong-iii}
    If $(x_0, y_0) \mathscr{J} ^ {M\times M, \Delta_M} (x_1, y_1)$, then
    the  2-sided congruences generated by $(x_0, y_0)$ and $(x_1, y_1)$
    coincide.
\end{enumerate}
\end{prop}
\begin{proof}
We only prove part \ref{prop-relative-greens-cong-i}, the proofs in
the other cases are similar. If $(x_0,
y_0) \mathscr{R} ^ {M\times M, \Delta_M} (x_1, y_1)$, then there exists $(m, m)
\in \Delta_M$ such that $(x_0m, y_0m) = (x_0, y_0)(m, m) = (x_1, y_1)$. In
particular, $(x_1, y_1)$ belongs to the right congruence generated by $(x_0,
y_0)$. By symmetry $(x_0, y_0)$ also belongs to the right congruence generated
by $(x_1, y_1)$ and so these two congruences coincide.
\end{proof}

A corollary of \cref{prop-relative-greens-cong} is: if $X$ is a set of
$\mathscr{R} ^ {M \times M, \Delta_M}$-class representatives in $M\times M$,
then every principal right congruence on $M$ is generated by a pair in $X$. So,
knowing $\mathscr{R} ^ {M\times M, \Delta_M}$-class representatives in $M\times
M$, will allow us to compute the set of all principal right congruences of $M$.
By doing this we hope for two things: that we can compute the representatives
efficiently and that the number of such representatives is relatively small
compared to $|M\times M|$.
Analogous statements hold for
principal left congruences and $\mathscr{L} ^ {M\times M, \Delta_M}$; and for
2-sided congruences and $\mathscr{J} ^ {M \times M, \Delta_M}$.

The rest of this section is dedicated to showing how to can compute
$\mathscr{R} ^ {M\times M, \Delta_M}$-, and $\mathscr{J} ^ {M\times M,
\Delta_M}$-class representatives in $M\times M$. We will not discuss
$\mathscr{L} ^ {M \times M, \Delta_M}$-classes beyond the following comments.
Suppose that we can compute relative $\mathscr{R} ^ {M \times M,
\Delta_M}$-class representatives in $M \times M$ for any arbitrary monoid $M$.
Relative $\mathscr{L}^{M \times M, \Delta_M}$-classes can be computed in one of
two ways: by performing the dual of what is described in this section for
computing  relative $\mathscr{R}^{M \times M, \Delta_M}$-classes; or by
computing an anti-isomorphism from $\phi: M\to M ^ {\dagger}$ from $M$ to its
dual $M ^ {\dagger}$, and computing relative $\mathscr{R}^{M^{\dagger}\times
M^{\dagger}, \Delta_{M^ {\dagger}}}$-class representatives in  $M ^
{\dagger}\times M ^ {\dagger}$. For the sake of simplicity, we opt for the
second approach. An anti-isomorphism into a transformation monoid can be found
from the left Cayley graph of $M$. The lattice of congruences of a monoid $M$
is generally several orders of magnitude harder to determine than the left
Cayley graph, and as such computing an anti-isomorphism from $M$ to a
transformation monoid is an insignificant step in this process. The degree of
the transformation representation of $M ^ {\dagger}$ obtained from the left
Cayley graph of $M$ is $|M|$. If $|M|$ is large, this can have an adverse
impact on the computation of relative $\mathscr{R}^{M ^ {\dagger} \times M ^
{\dagger}, \Delta_{M^ {\dagger}}}$-class representatives. However, it is
possible to reduce the degree of this representation by finding a right
congruence of $M ^ {\dagger}$ on which $M ^ {\dagger}$ acts faithfully using
the algorithms given in \cref{section-low-index}.

If $U$ is a fixed regular monoid, $M$ is an arbitrary submonoid of $U$, and $N$
an arbitrary submonoid of $N$, then we show how to compute $\mathscr{R} ^ {M ,
N}$-class representatives for $M$ using the structure of $U$. Algorithm 11 in
\cite{East2019aa} describes how to obtain the $\mathscr{R} ^ {M , M}$-class
representatives for $M$. We will show that, with minimal changes, Algorithm 11
from~\cite{East2019aa} can be used to compute $\mathscr{R}^{M, N}$-class
representatives for $M$. We will then show how, as a by-product of the algorithm
used to compute $\mathscr{R} ^ {M , N}$-class representatives, to compute
$\mathscr{J} ^ {M , N}$-class representatives.

The essential idea is to represent an $\mathscr{R}^{M, N}$-class by a group and
a strongly connected component of the action of $N$ on the $\mathscr{L}^{U,
U}$-class containing elements of $M$. We will show (in
\cref{prop-R-membership}) that this representation reduces the problem of
checking membership in an $\mathscr{R}^{M, N}$-class to checking membership in
the corresponding group. Starting with the $\mathscr{R}^{M, N}$-class of the
identity, new $\mathscr{R}^{M, N}$-class representatives are computed by
left multiplying the existing representatives by the generators of $N$,
and testing whether these multiples are $\mathscr{R}^ {M, N}$-related to an
existing representative.

Before we can describe the algorithm and prove the results showing that it is
valid, we require the following. If $\Psi: X \times M \to X$ is a right action
of $M$ on a finite set $X$, $Y$ is any subset of $X$, and $m \in M$, then we
define
\begin{equation*}\label{eq-induced-action}
(Y, m)\Psi = \set{(y, m)\Psi}{y \in Y}
\end{equation*}
and we define $m|_{Y}:Y\to (Y, m)\Psi$ by
\[
(y)m|_{Y} = (y, m)\Psi
\]
for all $y\in Y$ and all $m\in M$. When $\Psi$ is clear from the context, we
may write $x\cdot m$ and $Y\cdot m$ instead of $(x, m)\Psi$ and $(Y, m)\Psi$,
respectively. We define the \defn{stabiliser} of $Y$ to be
\[
\Stab_{M}(Y)=\set{m\in M}{(Y, m)\Psi =Y}.
\]
Clearly, if $m\in \Stab_{M}(Y)$, then $m|_{Y}: Y \to Y$ is a permutation of
$Y$. The quotient of the stabiliser by the kernel of its action on $Y$,
i.e.~the congruence
\[
\ker(\Psi) = \set{(m,n)\in M \times M}{m,n\in\Stab_{M}(Y),\ m|_Y=n|_Y},
\]
is isomorphic to $\set{m|_{Y}}{m\in \Stab_{M}(Y)}$ which is a subgroup of the
symmetric group $\Sym(Y)$ on $Y$. When using the $\cdot$ notation for actions
we write $\ker(\cdot)$ to denote the kernel of the action $\cdot$. We denote the
equivalence class of an element $m \in \Stab_{M}(Y)$ with respect to
$\ker(\Psi)$ by $[m]$. Clearly, since $N$ is a submonoid of $M$,
$\QuoStab{\Psi}{N}{Y}$ is a subgroup of $\QuoStab{\Psi}{M}{Y}$.

We denote the right action of the monoid $U$ on $U$ by right multiplication by
$\Phi: U\times U \to U$. If $L$ is a $\L^{U, U}$-class of $U$, then the group
$\QuoStab{\Phi}{U}{L}$, and its subgroup $\QuoStab{\Phi}{M}{L}$,  act
faithfully by permutations on $L$. The algorithms described in this part of the
paper involve computing with these permutation groups using standard algorithms
from computational group theory, such as the Schreier-Sims
Algorithm~\cite{Seress2003ab, Sims1971aa, Sims1970aa}.  The $\L^{U, U}$-classes
are often too large themselves for it to be practical to compute with
permutations of $\L$-classes directly. For many well-studied classes of monoids
$U$, such as the full transformation monoid, the symmetric inverse monoid, or
the partition monoid, there are natural faithful representations of the action
of $\QuoStab{\Phi}{U}{L}$ on the $\L^{U, U}$-class $L$ in a symmetric group of
relatively low degree. To avoid duplication we refer the reader to
\cite[Section 4]{East2019aa} for details. Throughout the rest of this paper, we
will abuse notation by writing $\QuoStab{\Phi}{N}{L}$,  to mean a faithful
low-degree representation of $\QuoStab{\Phi}{N}{L}$ when one is readily
computable.

It might be worth noting that we are interested in computing $\mathscr{R}^{M
\times M, \Delta_{M}}$-class representatives, but the results in
\cite{East2019aa} apply to submonoids of $M$ when $U$ is the full
transformation monoid $T_n$,  the partition monoid $P_n$, or the symmetric
inverse monoid $I_n$, for example, rather than to submonoids of $U \times U$.
If a monoid $U$ is regular, then so too is $U\times U$, and hence we may apply
the techniques from~\cite{East2019aa} to compute submonoids of $U\times U$. The
missing ingredients, however, are the analogues for $U\times U$ of the results
in~\cite[Section 4]{East2019aa}, that provide efficient faithful
representations of the right action of $N$ on the $\L^{U, U}$-classes
containing elements of $M\leq U$. It is possible to prove such analogues for
$U\times U$. However, in the case that $U=U_n$ is one of $T_n$, $P_n$, and
$I_n$, at least, this is not necessary, since $U_n\times U_n$ embeds into
$U_{2n}$. As such we may directly reuse the methods described in~\cite[Section
4]{East2019aa}.

In order to make practical use of $\QuoStab{\Phi}{M}{L}$, it is necessary that
we can efficiently obtain a generating set. The following
analogue of Schreier's Lemma for monoids provides a method for doing so.
If $X$ is any set and $Y\subseteq X$, then we denote the identity function from
$Y$ to $Y$ by $\id_Y$.

\begin{prop}[cf. Proposition 2.3 in~\cite{East2019aa}]\label{prop-schreier}
Let $M = \langle A \rangle$ be a monoid, let $\Psi: X \times M \to X$ be a
right action of $M$, and  let $Y_0, \ldots, Y_{n - 1} \subseteq X$ be
the elements
of a strongly connected component of the right action of $M$ on
$\mathcal{P}(X)$ induced by $\Psi$. Then the following hold:
\begin{enumerate}[\rm(i)]
  \item\label{prop-schreier-i} if $(Y_0, u_i)\Psi = Y_i$ for some
    $u_i\in M$, then there exists
    $\overline{u_i}\in M$ such that $(Y_i,  \overline{u_i})\Psi = Y_0$,
    $(u_i\overline{u_i})|_{Y_0} = \id_{Y_0}$, and $(\overline{u_i}u_i)|_{Y_0}
    = \id_{Y_i}$;
  \item\label{prop-schreier-ii} $\QuoStab{\Psi}{M}{Y_i}$  and
    $\QuoStab{\Psi}{M}{Y_j}$ are conjugate
    subgroups of $\Sym(X)$ for all $i, j\in \{0, \ldots, n - 1\}$;
  \item\label{prop-schreier-iii} if $u_0 = \overline{u_0} = 1_M$ and
    $u_i, \overline{u_i}\in M$ are as
    in part \ref{prop-schreier-i} for $i > 0$, then
    $\QuoStab{\Psi}{M}{Y_0}$ is generated by
    \[
      \set{(u_ia\overline{u_j})|_{Y_0}}{0\leq i, j< n,\ a\in A,\ Y_i \cdot
      a = Y_j}.
    \]
\end{enumerate}
\end{prop}

We require the following two right actions of $N$. One is the action on
$\mathcal{P}(U)$ induced by right multiplication, i.e. for $n\in N$ and $X\in
\mathcal{P}(U)$:
\begin{equation}\label{eqn-action-1}
(X, n)\Phi = \{xn : x\in X\}.
\end{equation}
The second right action of $N$ is that on $\mathscr{L}^{U, U}$-classes:
\begin{equation}\label{eqn-action-2}
(L^{U, U}_{x}, n)\Psi  = L^{U, U}_{xn}.
\end{equation}
where $n\in N$ and $x\in M$. The latter is an action because $\mathscr{L}^{U,
U}$ is a right congruence on $U$. The actions given in \eqref{eqn-action-1} and
\eqref{eqn-action-2} coincide in the case described by the following lemma.

\begin{lemma}[cf. Lemma 3.3 in~\cite{East2019aa}]\label{prop-actions-equiv}
Let $x, y\in U$ be arbitrary. Then the $\mathscr{L}^{U, U}$-classes $L_{x}^{U,
U}$ and  $L_{y}^{U, U}$ belong to the same strongly connected component of the
right action $\Phi$ of $N$ defined in \eqref{eqn-action-1} if and only if they
belong to the same strongly connected component of the right action $\Psi$ of
$N$ defined in \eqref{eqn-action-2}.
\end{lemma}

\cref{prop-actions-equiv} allows us use the actions defined in
\eqref{eqn-action-1} and \eqref{eqn-action-2} interchangeably within a strongly
connected component of either action. We will denote both of the right actions
in \eqref{eqn-action-1} and \eqref{eqn-action-2} by $\cdot$. Although the
actions in \eqref{eqn-action-1} and \eqref{eqn-action-2} are interchangeable
the corresponding stabilisers are not. Indeed, the stabiliser of any $\L^{U,
U}$-class with respect to the action given in \eqref{eqn-action-2} is always
trivial, but the stabiliser with respect to \eqref{eqn-action-1} is not. When
we write $\Stab_N(X)$ or $\Stab_U(X)$ for some subset $X$ of $U$, we will
always mean the stabiliser with respect to \eqref{eqn-action-1}.

We require the following result from~\cite{East2019aa} which relate to
non-relative Green's relations and classes.

\begin{lemma}[cf. Lemma 3.6 in~\cite{East2019aa}]\label{lem-free}
Let $x\in U$ and $s, t\in \Stab_{U}(L_{x}^{U, U})$ be arbitrary. Then $ys=yt$
for all $y\in L_x^{U, U}$ if and only if there exists $y\in L_x^{U, U}$ such
that $ys=yt$.
\end{lemma}

We also require the following results, which are modifications of the
corresponding
results in~\cite{East2019aa} for relative Green's relations and classes.

If $x, y\in U$, then we write $L_{x}^{U, U} \sim L_{y} ^ {U, U}$ to denote
that the $\mathscr{L}^{U, U}$-classes  $L_{x}^{U, U}$ and  $L_{y} ^{U, U}$
belong to the same strongly connected component of either of the right actions
of $N$ defined in \eqref{eqn-action-1} or \eqref{eqn-action-2}. Similarly, we
write $R_{x}^{U, U} \sim R_{y}^{U, U}$ for the analogous statement for
$\mathscr{R} ^ {U, U}$-classes.

Recall that we do not propose acting on the $\mathscr{L} ^ {U, U}$-classes
directly but rather we use a more convenient isomorphic action when available.
For example, if $U$ is the full transformation monoid, then the action of any
submonoid $M$ of $U$ on $\mathscr{L} ^ {U, U}$-classes of elements in $M$ is
isomorphic to the natural right action of $M$ on the set
$\{\operatorname{im}(m): m\in M\}$; for more examples and details see
\cite[Section 4]{East2019aa}. In~\cite[Algorithm 1]{East2019aa} a (simple brute
force) algorithm is stated that can be used to compute the word graph
corresponding to the right action of $M$ on $\{L_x^{U, U}: x\in M\}$. In the
present paper we must compute the word graph for the right action of $N$ on
$\{L_x^{U, U}: x\in M\}$.~\cite[Algorithm 1]{East2019aa} relies on the fact
that $L_1^ {U, U} \cdot M = \{L_x^{U, U}: x\in M\}$. Clearly, $L_{1_U}^ {U, U}
\cdot N$ is not equal to $\{L_x^{U, U}: x\in M\}$ in general. As such we cannot
use ~\cite[Algorithm 1]{East2019aa} directly to compute $\{L_x^{U, U}: x\in
M\}$. However, we can apply~\cite[Algorithm 1]{East2019aa} to compute the set
$\{L_x^{U, U}: x\in M\}$ and subsequently compute the word graph of the action
of $N$ on this set. The latter can be accomplished by repeating~\cite[Algorithm
1]{East2019aa} with the generating set $A$ for $N$, setting $C :=\{L_x^{U, U}:
x\in M\}$ in line 1. Since $N$  is a submonoid of $M$, the condition in line 3
never holds, and the condition in line 6 always holds.

The next result states some properties of relative Green's relations that are
required to prove the main propositions in this section.

\begin{lemma}[cf. Lemma 3.4, Corollaries 3.8 and 3.13
in~\cite{East2019aa}]\label{lem-greens}
Let $x, y\in M$ and let $s\in N$. Then the following hold:
\begin{enumerate}[\rm (i)]
  \item \label{lem-greens-simple}
    $L_{x}^{U, U} \sim L_{xs}^{U, U}$ if and
    only if $x \mathscr{R}^{M, N} xs$;
  \item \label{cor-greens-simple}
    if $x \R^{M, N} y$ and $xs \L^{U, U} y$, then $xs \R^{M, N} y$;
  \item \label{cor-collect}
    if $x\R^{M, N}y$ and  $xs \L^{U, U} y$, then $f: L_x^{U, U}\cap R_x^{M,
    N}\to L_y^{U, U}\cap R_x^{M, N}$ defined by $t \mapsto ts$ is a
    bijection.
\end{enumerate}
\end{lemma}
\begin{proof}
\noindent \textbf{(i).} ($\Rightarrow$)
By assumption, $L_{xs}^{U, U}$ and $L_x^{U, U}$ belong to the same strongly
connected component of the action of $N$ on $U/\L^{U, U}$ by right
multiplication. Hence, by \cref{prop-schreier}\ref{prop-schreier-i},
there exists
$\overline{s}\in N$ such that $L_{xs}^{U, U}\overline{s}=L_x^{U, U}$ and
$s\overline{s}$ acts on $L_x^{U, U}$ as the identity.  Hence, in particular,
$xs\overline{s}=x$ and so $xs\R^{M, N} x$.

($\Leftarrow$)
Suppose $x\R^{M, N} xs$. Then there exists $t\in N$ such that $xst=x$. It
follows that $L_{x}^{U, U}\cdot s= L_{xs}^{U, U}$ and $L_{xs}^{U, U}\cdot t =
L_{x}^{U,U}$. Hence $L_{x}^{U,U}\sim L_{xs}^{U,U}$.
\medskip

\noindent \textbf{(ii).}
Since $x\R ^ {M, N} y$ there exists $t\in N$ such that $yt = x$. Hence
$L_x^{U, U}\cdot s = L_{xs} ^ {U, U}$ and $L_{xs}^ {U, U} \cdot t = L_y ^ {U,
U} \cdot t = L_{yt} ^ {U, U} = L_x ^ {U, U}$. In particular, $L_x ^ {U, U}
\sim L_{xs} ^ {U, U}$ and so, by part \ref{lem-greens-simple}, $y\mathscr{R} ^
{M, N} x \mathscr{R}^{M, N} xs$, as required.
\medskip

\noindent \textbf{(iii).}
Let $t\in L_x^{U, U}\cap R_x^{M, N}$ be arbitrary. Then $t\R^{M, N}x\R^{M,
N}y$ and $ts\L^{U, U} xs\L^{U, U}y$ and so, by part \ref{cor-greens-simple},
$ts\R^{M,N} y\R ^ {M, N} x$. In other words, $ts\in L_y^{U, U}\cap R_x^{M, N}$
for all $t\in L_x^{U, U}\cap R_x^{M, N}$. In particular, $x\R ^ {M, N}xs$ and
so, by part \ref{lem-greens-simple}, $L_x^ {U, U}\sim L_{xs} ^ {U, U} =
L_y^{U, U}$.  Hence, by \cref{prop-schreier}\ref{prop-schreier-i},
there exists $\overline{s}\in
N$ such that $ts\overline{s}=t$ for all $t\in L_x^{U, U}\cap R_x^{M, N}$.
Therefore $t\mapsto ts$ and $u\mapsto u\overline{s}$ are mutually inverse
bijections from $L_x^{U, U}\cap R_x^{M, N}$ to $L_y^{U, U}\cap R_x^{M, N}$ and
back.
\end{proof}

The next proposition allows us to decompose the $\R ^ {M, N}$-class of $x\in M$
into the sets $R_{x} ^ {M, N}\cap L_y ^{U, U}$ where the $L_y ^ {U, U}$ form a
${\sim}$-strongly connected component with respect to $N$.

\begin{prop}[cf. Proposition~3.7(a) in~\cite{East2019aa}]\label{prop-main-0}
Suppose that  $x, y\in M$ are arbitrary.  If $x\R ^ {M, N}y$, then $L_x^{U,
U}\sim L_y^{U, U}$. Conversely, if $L_x^{U, U}\sim L_y^{U, U}$, then there
exists $z\in M$ such that $z\R ^ {M, N}x$ and $L_z^{U, U} = L_y ^ {U, U}$.
\end{prop}
\begin{proof}
Suppose that $y \in M$ is such that $x\neq y$. Then $y \R ^ {M, N} x$ implies
that there exists $s, t \in N$ such that $xs = y$ and $yt = x$. In particular,
$xs \R ^ {M, N} x$ and so $L_{x} ^ {U, U}\sim L_{xs}^{U, U}=L_y^{U, U}$, by
\cref{lem-greens}\ref{lem-greens-simple}($\Leftarrow$).

If $y\in M$ is such that $L_x ^ {U, U} \sim L_{y}^{U, U}$, then there exists
$s\in N$ such that $L_y ^ {U, U} = L_{xs} ^ {U, U}$ and so, by
\cref{lem-greens}\ref{lem-greens-simple}($\Rightarrow$), $x \R ^ {M, N} xs$.
\end{proof}

The next result, when combined with \cref{prop-main-0}, completes the
decomposition of the $\R ^ {M, N}$-class of $x\in M$ into ${\sim}$-strongly
connected component with respect to $N$ and a group, by showing that $L_{x} ^
{U, U}\cap R_x ^{M, N}$ is a group with the operation defined in part
\ref{prop-main-1-i} of the next proposition.

\begin{prop}[cf. Proposition 3.9 in~\cite{East2019aa}]\label{prop-main-1}
Suppose that $x\in M$ and there exists $x'\in U$ where $xx'x=x$
(i.e.~$x$ is regular in $U$).  Then the following hold:
\begin{enumerate}[\rm (i)]
  \item\label{prop-main-1-i}
    $L_x^{U, U}\cap R_x^{M, N}$ is a group under the multiplication $*$
    defined by $s\ast t=sx't$ for all $s,t \in L_x^{U, U}\cap R_x^{M, N}$ and
    its identity is $x$;
  \item\label{prop-main-1-ii}
    $\phi:\QuoStab{\cdot}{N}{L_x^ {U, U}}\to L_x^{U, U}\cap R_x^{M, N}$
    defined by $([s])\phi=xs$, for all $s\in \Stab_{N}(L_x^ {U, U})$, is an
    isomorphism;
  \item\label{prop-main-1-iii}
    $\phi^{-1}:L_x^{U, U}\cap R_x^{M, N}\to \QuoStab{\cdot}{N}{L_x^ {U, U}}$
    is defined by $(s)\phi^{-1}=[x's]$ for all $s\in L_x^{U, U}\cap R_x^{M,
    N}$.
\end{enumerate}
\end{prop}
\begin{proof}
We begin by showing that $x$ is an identity under the multiplication $*$ of
$L_x^{U, U}\cap R_x^{M, N}$. Since $x'x\in L_x^{U, U}$ and $xx'\in R_x^{U,
U}$ are idempotents, it follows that $x'x$ is a right identity for $L_x^{U,
U}$ and $xx'$ is a left identity for $R_x^{M, N}\subseteq R_x^{U, U}$.  So,
if $s\in L_x^{U, U}\cap R_x^{M, N}$ is arbitrary, then
\begin{equation*}
  x*s=xx's=s=sx'x=s*x,
\end{equation*}
as required.

We will prove that part (b) holds, which implies part (a). \smallskip

\noindent\textbf{$\phi$ is well-defined.}
If $s\in M$ and $s\in \Stab_N(L_x^ {U, U})$, then $xs\L^{U, U} x$. Hence, by
\cref{lem-greens}\ref{cor-greens-simple}, $xs\R^{M, N} x$ and so
$(s)\phi=xs\in L_x^{U, U}\cap R_x^{M, N}$. If $t\in\Stab_N(L_x^{U, U})$ is
such that $[t]=[s]$, then, by \cref{lem-free}, $xt=xs$.\smallskip

\noindent\textbf{$\phi$ is surjective.}
Let $s\in L_x^{U, U}\cap R_x^{M, N}$ be arbitrary. Then $xx's=x*s=s$ since
$x$ is the identity of $L_x^{U, U}\cap R_x^{M, N}$. It follows that $$L_x^{U,
U}\cdot x's=L^{U, U}_s=L^{U, U}_x$$ and so $x's\in \Stab_U(L_x^{U, U})$.
Since $x\R^{M, N}s$, there exists $u\in N$ such that $xu=s=xx's$.  It follows
that $u\in \Stab_N(L_x^{U, U})$ and, by \cref{lem-free}, $[u]=[x's]$. Thus
$(u)\phi=xu=s$ and $\phi$ is surjective.
\smallskip

\noindent\textbf{$\phi$ is a homomorphism.}
Let $s, t\in \Stab_N(L_x^{U, U})$. Then, since $xs\in L_x^{U, U}$ and $x'x$
is a right identity for $L_x^{U, U}$,
\[
  ([s])\phi*([t])\phi=xs*xt=xsx'xt=xst= ([st])\phi=([s][t])\phi,
\]
as required.
\smallskip

\noindent\textbf{$\phi$ is injective.}
Let $\theta: L_x^{U, U}\cap R_x^{M, N}\to \QuoStab{\cdot}{N}{L_x^ {U, U}}$ be
defined by $(y)\theta=[x'y]$ for all $y\in L_x^U\cap R_x^S$. We will show
that $\phi\theta$ is the identity mapping on $\QuoStab{\cdot}{N}{L_x^ {U,
U}}$, which implies that $\phi$ is injective, that $(y)\theta\in
\QuoStab{\cdot}{N}{L_x^ {U, U}}$ for all $y\in L_x^{U, U}\cap R_x^{M, N}$
(since $\phi$ is surjective), and also proves part (c) of the proposition. If
$s\in \Stab_N(L_x^{U, U})$, then $([s])\phi\theta=(xs)\theta=[x'xs]$.  But
$xx'xs=xs$ and so $[x'xs]=[s]$ by \cref{lem-free}.  Therefore,
$([s])\phi\theta=[s]$, as required.
\end{proof}

Finally, we combine the preceding results to test membership in an
$\R ^ {M, N}$-class.

\begin{prop}
\label{prop-R-membership}
Suppose that $x\in M$ and there is $x'\in U$ with $xx'x=x$. If $y\in U$ is
arbitrary, then $y \R^{M, N} x$ if and only if $y \R^{U, U} x$, $L_{y}^{U,
U}\sim L_{x}^{U, U}$, and $[x'yv]\in \QuoStab{\cdot}{N}{L_x^ {U, U}}$ where
$v\in N$ is any element such that $L_{y}^U\cdot v=L_{x}^U$.
\end{prop}
\begin{proof}
($\Rightarrow$)
Since $R_x^{M, N}\subseteq R_x^{U, U}$, $y \R^{U, U} x$ and from
\cref{prop-main-0}, $L_{y}^{U, U}\sim L_{x}^{U, U}$. Suppose that
$v\in N$ is such that $L_{y}^U\cdot v= L_{x}^U$. Then, by
\cref{lem-greens}\ref{cor-collect},  $yv\in L_x^{U, U}\cap R_x^{M, N}$ and so,
by \cref{prop-main-1}\ref{prop-main-1-iii}, $[x'yv]\in
\QuoStab{\cdot}{N}{L_x^ {U,
U}}$.
\smallskip

($\Leftarrow$) Since $y\in R_x^{U, U}$ and $xx'$ is a left identity in its
$\R^{U, U}$-class, it follows that $xx'y=y$.  Suppose that $v\in N$ is any
element such that $L_{y}^U\cdot v=L_{x}^U$ (such an element exists by the
assumption that $L_y^{U, U} \sim L_x^{U, U}$). Then, by assumption, $[x'yv]\in
\QuoStab{\cdot}{N}{L_x^ {U, U}}$ and so by
\cref{prop-main-1}\ref{prop-main-1-ii},
$yv = x \cdot x'yv\in L_x^{U, U}\cap R_x^{M, N}$. But $L_{y}^{U, U}\sim
L_{yv}^{U, U}$, and so, by \cref{lem-greens}\ref{lem-greens-simple}, $y\R ^
{M, N} yv$, and so $x\R^{M, N}yv\R^{M, N}y$, as required.
\end{proof}

We have shown that analogues of all the results required to prove the
correctness of~\cite[Algorithm 11]{East2019aa} hold for relative Green's
relations in addition to their non-relative counterparts. For the sake of
completeness, we state a version of Algorithm 11 from~\cite{East2019aa} that
computes  the set $\mathfrak{R}$ of $\R ^ {M, N}$-class representatives and the
word graph $\Gamma$ of the left action (by left multiplication) of $N$ on
$\mathfrak{R}$; see \cref{algorithm-enumerate}. We require the word graph
$\Gamma$ to compute $\mathscr{J}^{M, N}$-class representatives in the next
section, it is not required for finding  the $\R ^ {M, N}$-class
representatives. The algorithm presented in \cref{algorithm-enumerate} is
simplified somewhat from Algorithm 11 from~\cite{East2019aa} because we only
require the representatives and the word graph, and not the associated data
structures.

\begin{algorithm}
\caption{(cf. Algorithm 11 in~\cite{East2019aa}) Enumerate $\R ^ {M,
N}$-class representatives}
\label{algorithm-enumerate}
\begin{algorithmic}[1]
\item[\textbf{Input:}] A monoid $M$ and a submonoid $N:=\genset{A}$
  where $A:=\{a_0,\ldots, a_{m -1}\}$
\item[\textbf{Output:}] The set $\mathfrak{R}$ of $\R ^ {M, N}$-classes
  representatives of elements in $M$ and the word graph $\Gamma$ of the
  action of $N$ on $\mathfrak{R}$.

\item set $\mathfrak{R}:=(r_0:=1_M)$ where $1_M\in M$ is the identity of $M$
  \Comment{initialise the list of $\R^{M, N}$-class representatives}
\item $\Gamma:= (V, E)$ where $V := \{0\}$ and $E: = \varnothing$
  \Comment{initialise the word graph of the action of $N$ on $\mathfrak{R}$}
\item find $(L_{1_M}^ {U, U})\cdot M = \set{L_x^{U, U}}{x\in M}$
  \Comment{Algorithm~1 from~\cite{East2019aa}}
\item compute the action of $N$ on $(L_{1_M}^ {U, U})\cdot M$
  \Comment{use the modified version of Algorithm~1
  from~\cite{East2019aa} discussed above}

\item
  find the strongly connected components of $(L_{1_M}^ {U, U})\cdot M$
  \Comment{standard graph theory algorithm}
\item set  $z_0, \ldots, z_{k - 1}\in M$ to be representatives of
  $\L^{U, U}$-classes in each strongly connected component
\item find generating sets for the groups $\QuoStab{\cdot}{N}{L_{z_n}
  ^ {U, U}}$,
  $n\in\{0,\ldots, k - 1\}$
  \Comment{Algorithm~4 in~\cite{East2019aa}}
\item $i := 0$
  \While{$i < |\mathfrak{R}|$}
  \Comment{loop over: existing $\R$-representatives}
  \State $j := 0$
  \While {$j < m$}
  \Comment{loop over: generators of $N$}

  \State find $n\in \{0,\ldots, k - 1\}$ such that $L_{z_n}^{U, U} \sim
  L_{a_jr_i}^{U, U}$

  \State find $u\in N$ such that $L_{z_n}^{U, U} \cdot u=L_{a_jr_i}^{U, U}$
  \Comment{Algorithm~2 in~\cite{East2019aa}}

  \State find $\overline{u}\in U ^ 1$ such that $L_{a_jr_i}^{U, U}
  \cdot \overline{u} =
  L_{z_n}^{U, U}$ and $z_nu\overline{u} = z_n$

  \For{$r_l\in \mathfrak{R}$ with $(r_l,a_jr_i)\in \mathscr{R}^{U, U}$ and
  $(r_l, z_n) \in \mathscr{L}^{U, U}$}
  \If{$[r_l'a_jr_i\overline{u}]\in
    \QuoStab{\cdot}{N}{L_{r_l}^ {U, U}} = \QuoStab{\cdot}{N}{L_{z_n}^ {U,
  U}}$}\Comment{$a_jr_i\R^{M, N}r_l$ by \cref{prop-R-membership}}
  \State $\Gamma \gets (V, E\cup \{(i, a_j, l)\})$
  \Comment{update the word graph for the action of $N$ on $\mathfrak{R}$}
  \State go to line 9
  \EndIf
  \EndFor
  \Comment{$a_jr_i\overline{u}$ is a new $\R ^ {M, N}$-representative}
  \State $\Gamma \gets (V\cup \{|\mathfrak{R}|\}, E\cup \{(i, a_j,
  |\mathfrak{R}|)\})$
  \Comment{update the word graph}
  \State Append $r_{|\mathfrak{R}|} := a_jr_i\overline{u}$ to $\mathfrak{R}$
  \Comment{update the set of $\R^{M, N}$-representatives}
  \State $j := j + 1$
  \EndWhile
  \State $i := i + 1$
  \EndWhile
  \State \Return  $\mathfrak{R}$, $\Gamma$
\end{algorithmic}
\end{algorithm}

The next proposition shows that relative $\mathscr{J} ^ {M, N}$-classes
correspond to strongly connected components of the word graph output
by~\cref{algorithm-enumerate}.

\begin{prop}\label{prop-J-scc}
Let $x, y\in M$. Then $x\mathscr{J} ^ {M, N} y$ if and only if $R_x^{M, N}$
and $R_y^{M, N}$ belong to the same strongly connected component of the action
of $N$ on the $\mathscr{R}^{M, N}$-classes of $M$ by left multiplication.
\end{prop}
\begin{proof}
This follows almost immediately since $\mathscr{D} ^ {M, N} = \mathscr{L} ^ {M,
N} \circ \mathscr{R} ^ {M, N} = \mathscr{J} ^ {M, N}$, because $M$ and $N$ are
finite.
% ($\Rightarrow$) If $x\mathscr{J} ^ {M, N} y$, since $M$ and $N$ are finite,
% it follows that $x\mathscr{D} ^ {M, N} y$. In particular, there exists
% $z\in M$ such that $x\mathscr{L}^{M, N}z \mathscr{R}^{M, N} y$ and so there
% exist $u, v\in N$ such that $ux = z$ and $vz = x$. In particular, $u\cdot
% R_{x} ^ {M, N} = R_{ux} ^ {M, N} = R_{z} ^ {M, N} = R_y ^ {M, N}$, and
% similarly, $v\cdot R_{y} ^ {M, N} = R_{x} ^ {M, N}$. Therefore $R_x^{M, N}$
% and $R_y^{M, N}$ belong to the same strongly connected component of the
% action of $N$ on the $\mathscr{R}^{M, N}$-classes of $M$ by left
% multiplication.
%
% ($\Leftarrow$) Suppose that $x, y\in M$ and $R_x^{M, N}$ and $R_y^{M, N}$
% belong to the same strongly connected component of the action of $N$ on the
% $\mathscr{R}^{M, N}$-classes of $M$ by left multiplication. Then there
% exists $s\in N$ such that $s\cdot R_x^{M, N} = R_y^{M, N}$ and so $sx \R ^
% {M, N} y$. Hence there exists $t\in N$ such that $sxt =y$, and by symmetry
% there exist $u, v\in N$ such that $uyv = x$. Thus $x\mathscr{J} ^ {M, N} y$
% as required.
\end{proof}

It follows from~\cref{prop-J-scc} that we can compute $\mathscr{J} ^ {M,
N}$-class representatives by using \cref{algorithm-enumerate} to find the word
graph $\Gamma$, and then using one of the standard algorithms from graph theory
to compute the strongly connected components of $\Gamma$.

\printbibliography
\appendix

\section{Performance comparison}\label{appendix-benchmarks}

% TODO maybe never:
% * Produce some benchmark graphs for Froidure-Pin vs simple enumeration of
%    lattices.
% *  Count the number of left + right congruences in various standard examples
%    of semigroups (including all small semigroups\textit{})

In this section we present some data related to the performance of the low-index
congruences algorithm as implemented in \libsemigroups and
\cref{algorithm-enumerate} as implemented in version 5.3.0 of
\Semigroups for \GAP by the
authors.
We compare the performance of our implementations in  \libsemigroups
and \Semigroups with the algorithm from \cite{Freese2008aa} implemented in
\CREAM and with earlier versions  of \Semigroups which do not contain
the optimizations described in \cref{section-algo-2}. It may be worth
bearing in mind that the input to the algorithms
implemented in \CREAM is the multiplication table of a semigroup or monoid, and
that these multiplication
tables were computed using the methods in the \Semigroups package.
The input to the low-index algorithm is the presentation of a
semigroup or monoid and the input for \cref{algorithm-enumerate} is a
black-box multiplication monoid and a set of generators of a submonoid.

\subsection{A parallel implementation of the low-index congruences algorithm}
In \cref{figure-parallel} we present some data related to the performance of
the parallel implementation of the low-index congruences algorithm in
\libsemigroups. It can be seen in \cref{figure-parallel} that, in these
examples, doubling the number of threads, more or less, halves the execution
time up to 4 threads (out of 8) on the left, and 16 threads (out of 64) on the
right. Although the performance continues to improve after these numbers of
threads, it does not continue to halve the runtime. The degradation in
performance might be a consequence of using too many resources on the host
computer, or due to there being insufficient work for the threads in the chosen
examples.

\begin{figure}
\centering
\includegraphics[width=0.48\textwidth]{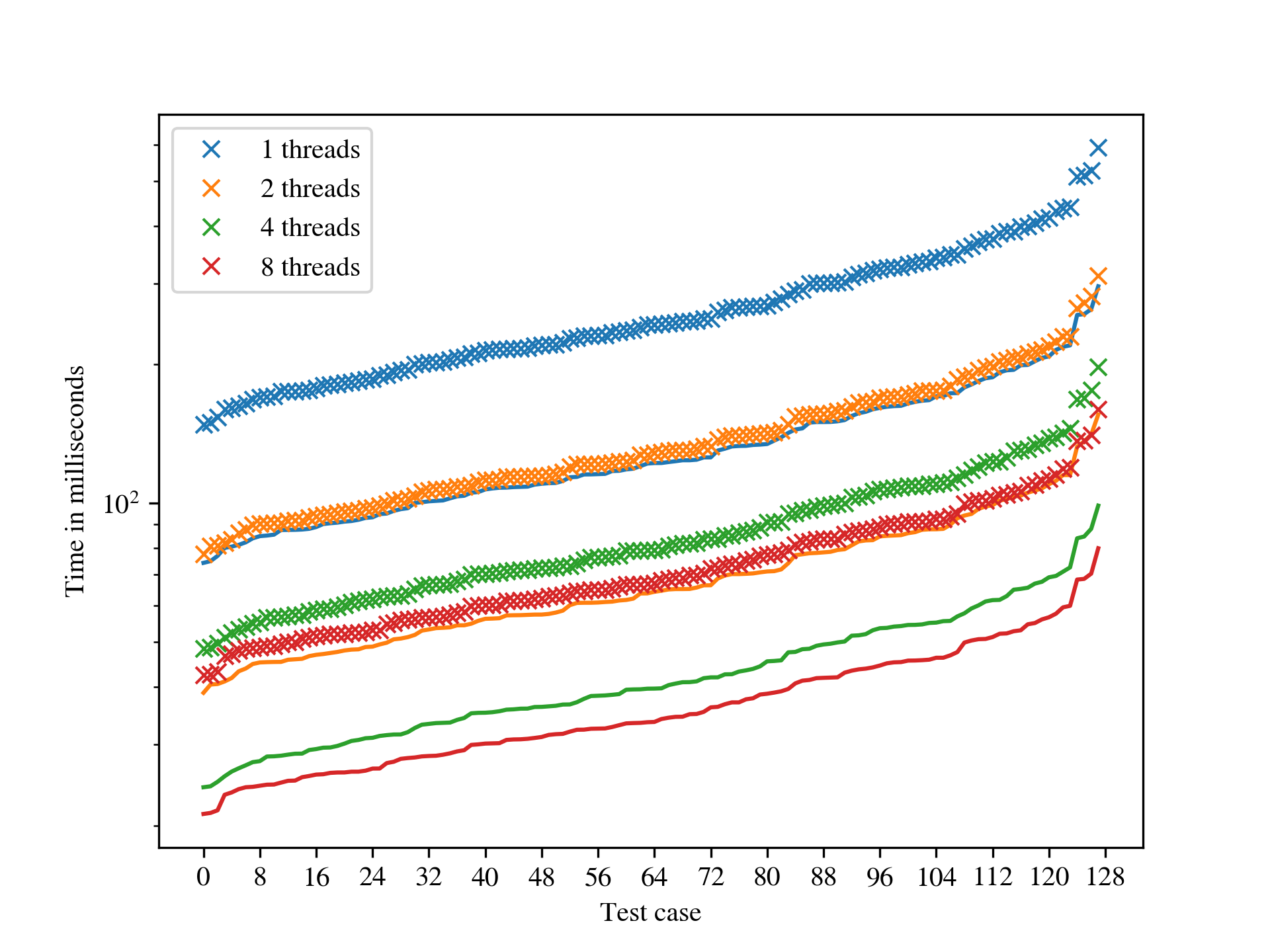}
\includegraphics[width=0.48\textwidth]{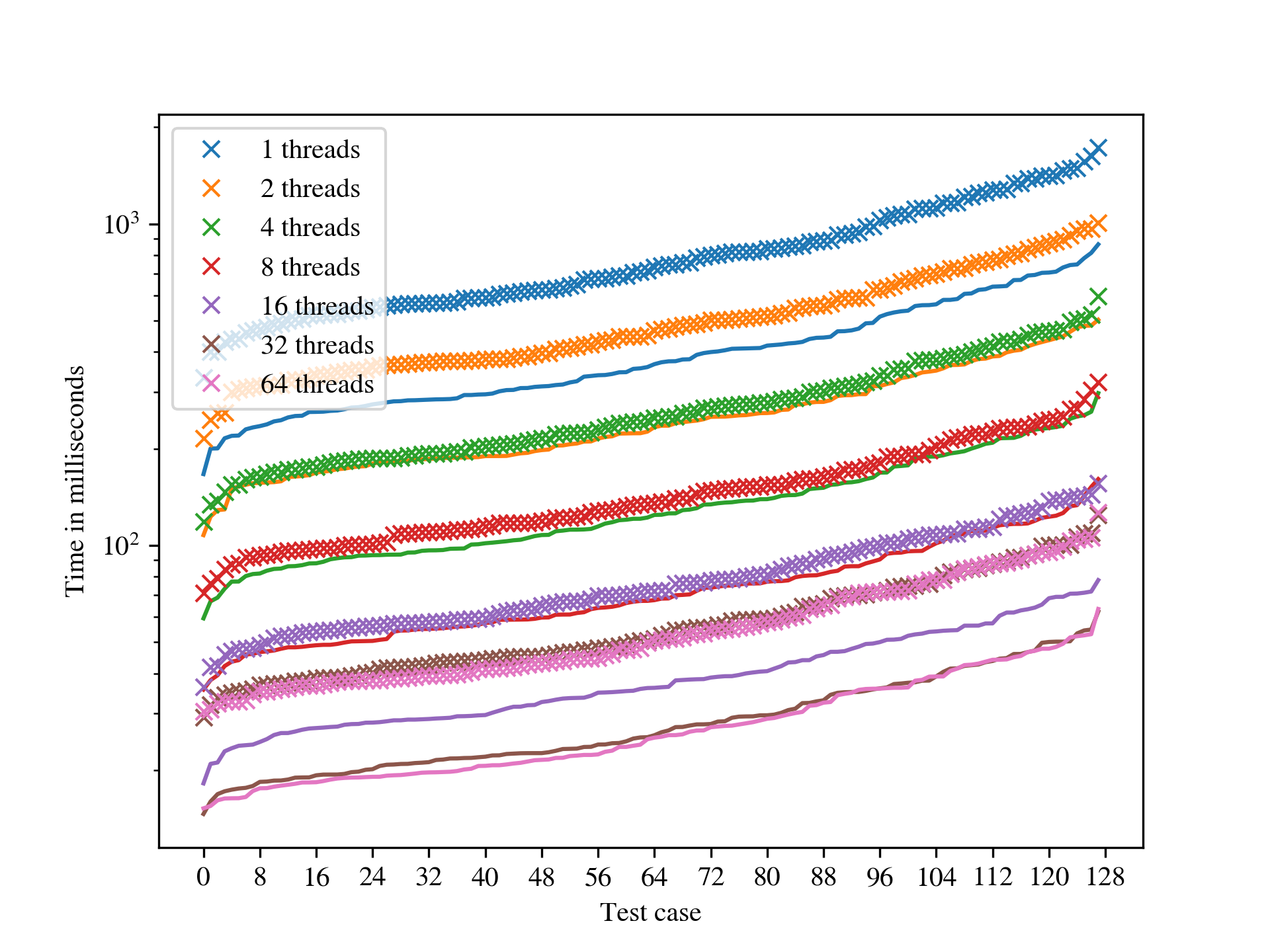}
\caption{Comparison of the performance of the parallel implementation
  of low-index congruences algorithm in \libsemigroups for 128 randomly chosen
  2-generated 1-relation monoid presentations where the lengths of the relation
  words is $10$. The times, indicated by crosses, are the means of the times
  taken in 100 trials to compute the right congruences with up to $5$ classes
  (inclusive).  For comparison, the solid lines indicate half the time taken for
  corresponding computation written using the same colour. The mean number of
  congruences computed per presentation was $189,589$ (left) and $204,739$
  (right). The graph on the left was produced using a 2021 MacBook Air with 4
  cores (each with 2 threads of execution) and on the right using a
  cluster of 64
2.3GHz AMD Opteron 6376 cores with 512GB of RAM.}
\label{figure-parallel}
\end{figure}

\subsection{The impact of presentation length on the low-index
congruences algorithm}

In this section we present some experimental evidence about how the length of a
presentation impacts the performance of the low-index congruences algorithm. If
$\P = \langle A\mid R\rangle$ is a monoid presentation, then we refer to the
sum $\sum_{(u, v) \in R} |u| + |v|$ of the lengths of the relation words in $R$
as the \defn{length} of $\P$. In \cref{figure-length}, a comparison of the
length of a presentation for the full transformation monoid $T_4$ versus the
runtime of the implementation of the low-index congruences algorithm in
\libsemigroups is given. In~\cref{subfigure-length-t-n}, the initial input
presentations were Iwahori's presentation from~\cite[Theorem
9.3.1]{Ganyushkin2009} for the full transformation monoid $T_4$, and a complete
rewriting system for $T_4$ output from the Froidure-Pin
Algorithm~\cite{Froidure1997aa}. Both initial presentations contain many
redundant relations, these were removed $5$ at a time, and the time to compute
the number of left congruences of $T_4$ with at most $16$ classes is plotted
for each resulting presentation. We also attempted to perform the same
computation with input Aizenstat's presentation from~\cite[Ch. 3, Prop
1.7]{Ruskuc1995aa} (which has length $162$), but the computation was extremely
slow.

In~\cref{subfigure-length-cyclic-1} and \cref{subfigure-length-cyclic-2}, a
similar approach is taken. The initial input presentations
in~\cref{subfigure-length-cyclic-1} and \cref{subfigure-length-cyclic-2} were
Fernandes' presentations from~\cite[Theorems 2.6 and 2.7]{Fernandes2022aa},
respectively, and the outputs of the Froidure-Pin
Algorithm~\cite{Froidure1997aa} with the corresponding generating sets.

\begin{figure}
\centering
\begin{subfigure}{0.48\textwidth}
  \includegraphics[width=\textwidth]{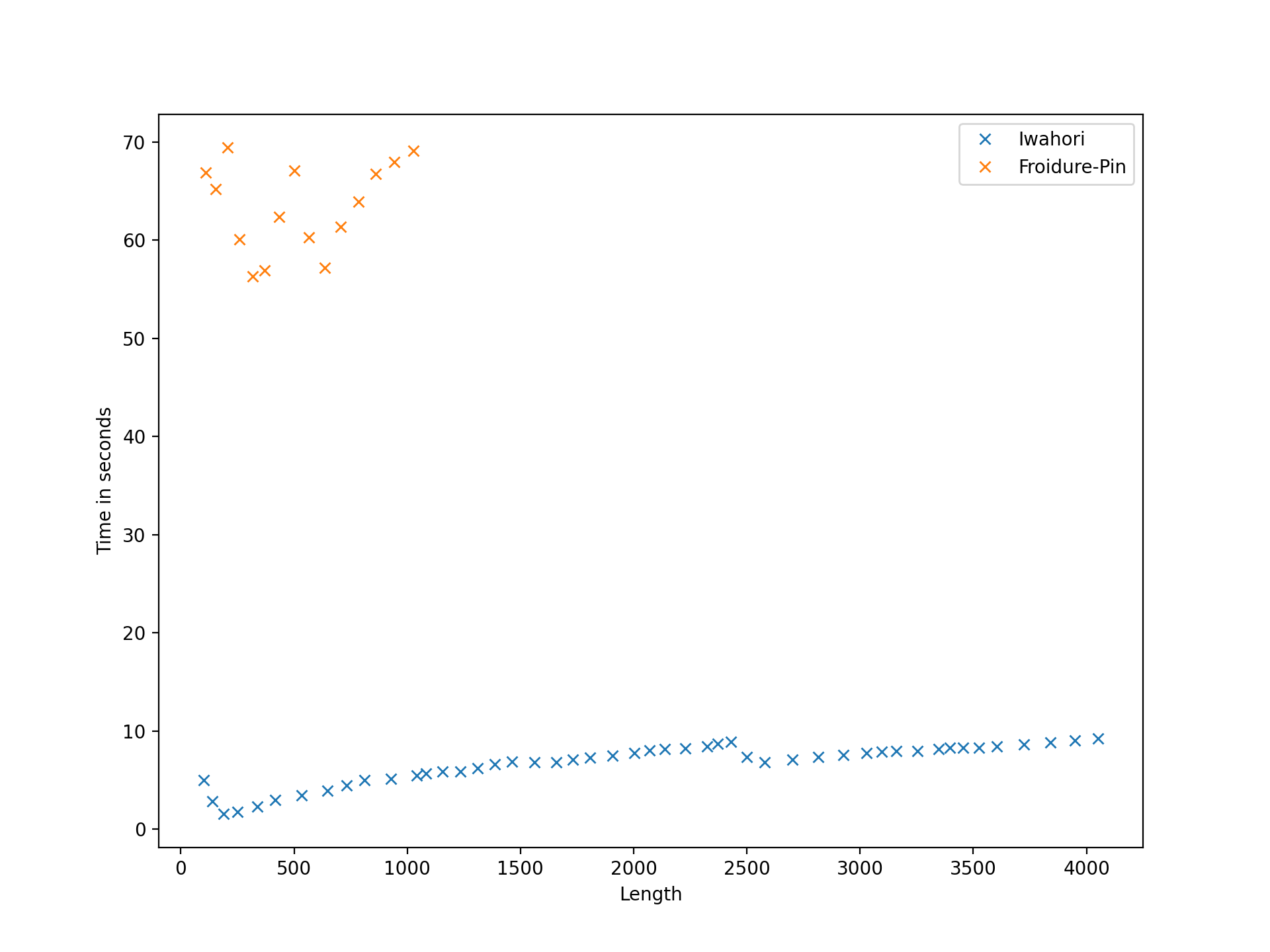}
  \caption{Full transformation monoid $T_4$.}
  \label{subfigure-length-t-n}
\end{subfigure}
\begin{subfigure}{0.48\textwidth}
  \includegraphics[width=\textwidth]{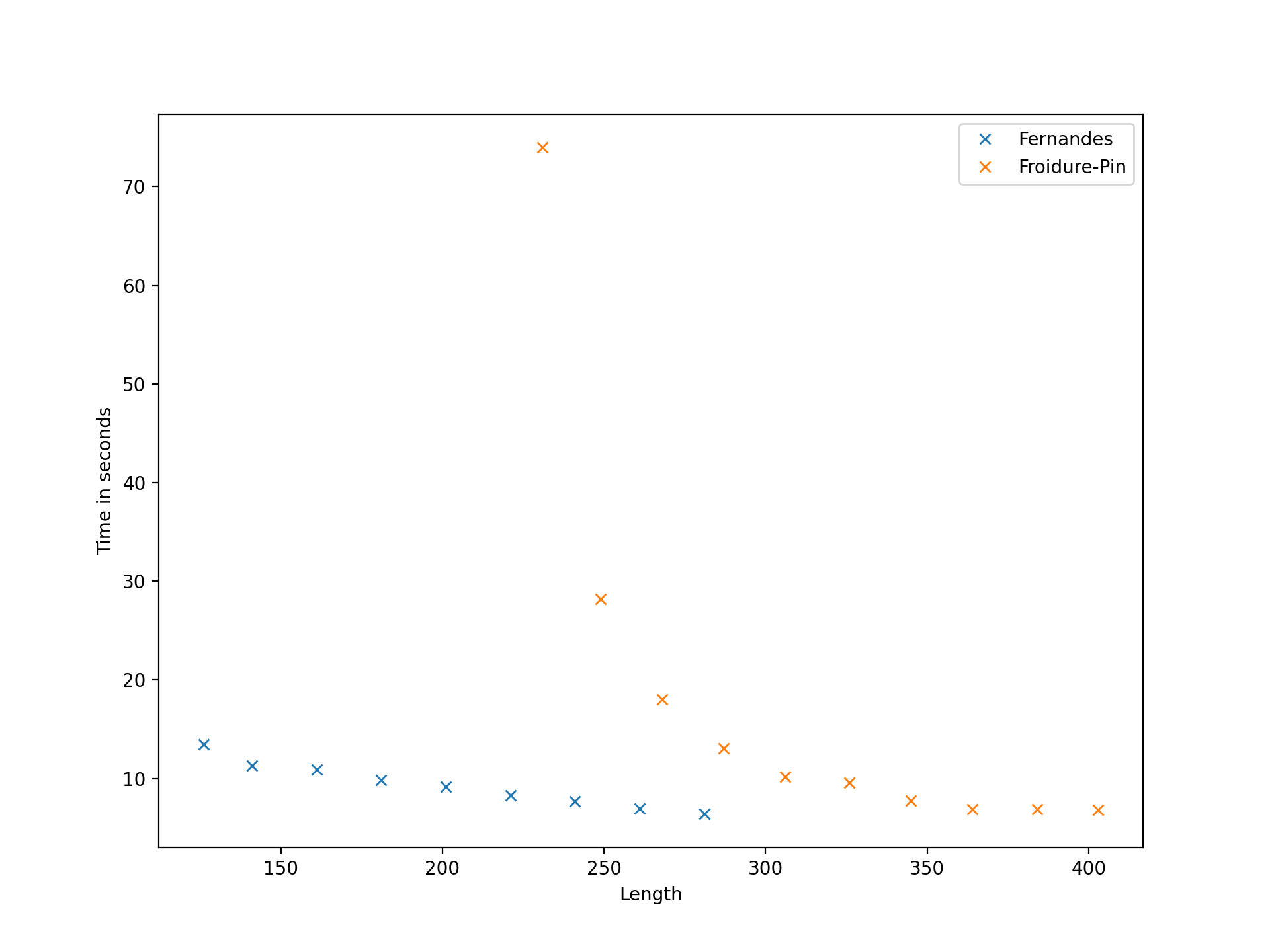}
  \caption{Cyclic inverse monoid $\mathcal{CI}_{10}$.}
  \label{subfigure-length-cyclic-1}
\end{subfigure}
\begin{subfigure}{0.48\textwidth}
  \includegraphics[width=\textwidth]{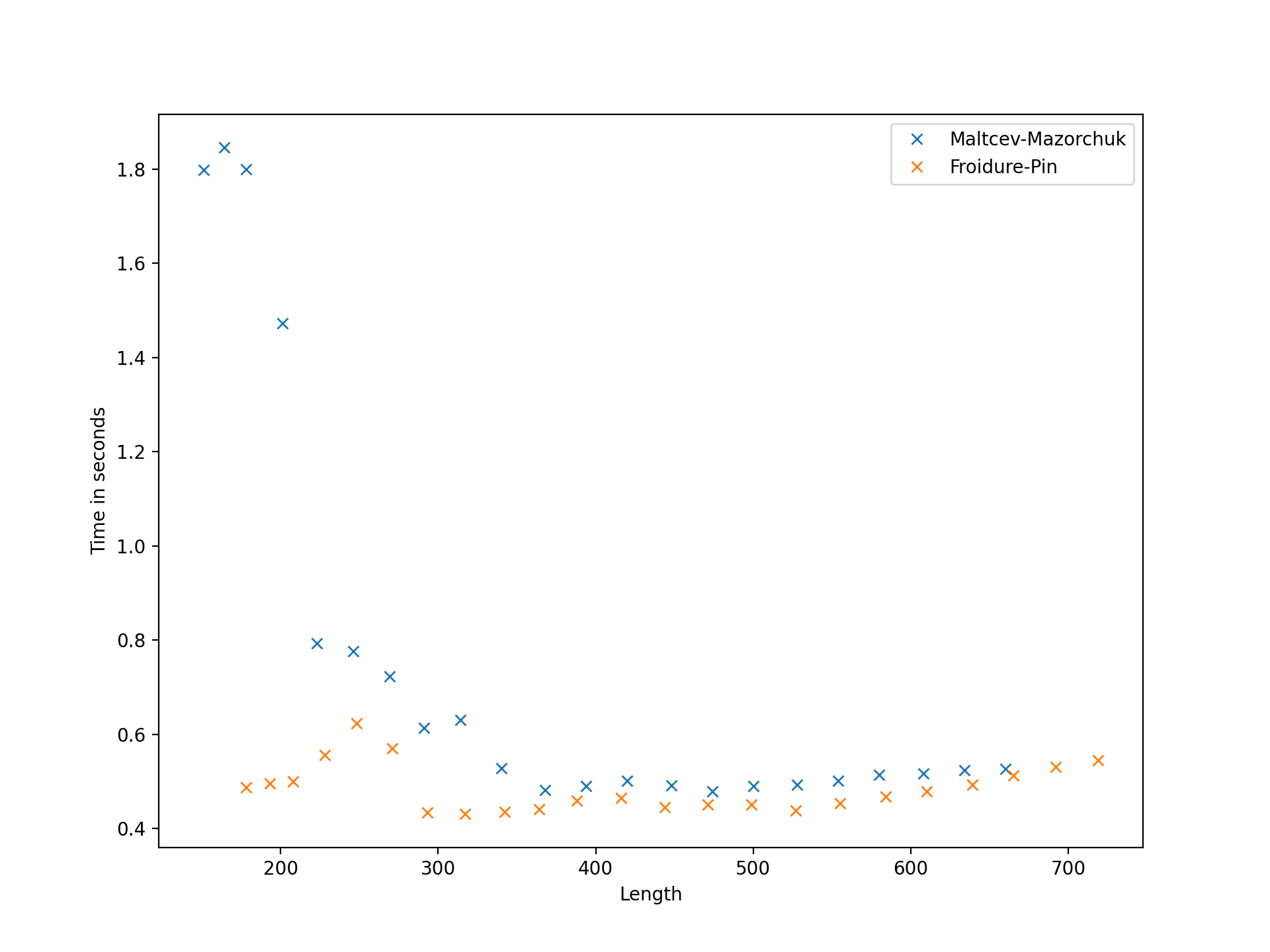}
  \caption{The singular part $B_4\setminus S_4$ of the Brauer monoid $B_4$.}
  \label{subfigure-length-cyclic-2}
\end{subfigure}
\caption{Comparison of the runtime of low-index congruences algorithm for
  computing the right congruences on the full transformation monoid $T_4$
  with up to 16 classes in~\cref{subfigure-length-t-n}; the right
  congruences on the cyclic inverse monoid $\mathcal{CI}_{10}$ with at most
  $4$ classes in~\cref{subfigure-length-cyclic-1} and the singular part of
  the Brauer monoid $B_4$~\cref{subfigure-length-cyclic-2} with
index up to $6$. Each value is the mean of $100$ trials.}
\label{figure-length}
\end{figure}

The outcome is mixed: in all of these examples the ``human'' presentations
provide better performance than their ``non-human'' counterparts; for $T_n$
shorter means faster up to a point; for the first presentation  for
$\mathcal{CI}_n$ shorter means slower; for the second presentation for
$\mathcal{CI}_n$ shorter means faster.  There is only a single value for the
computer generated presentation for $\mathcal{CI}_n$ with the second
generating set because it was not straightforward to find redundant relations
in this presentation. It might be worth noting that although the length of
this presentation is ${\sim}16,000$, the time taken is still considerably less
than the fastest time using the first presentation.

\subsection{1-sided congruences on finite monoids}
In this section we present some data about the relative performance of the
implementation of low-index congruences algorithm,
\cref{algorithm-enumerate}, and~\cite{Freese2008aa} in \libsemigroups,
\Semigroups, and \CREAM, respectively; see
Tables~\ref{table-catalan-1},~\ref{table-catalan-2},~\ref{table-catalan-3},~\ref{table-order-endo-2},
and~\ref{table-order-endo-3}.

Again, the outcome is somewhat mixed. Generally in the examples presented in
this section, the low-index congruence algorithm as implemented in
\libsemigroups is fastest for finding all right congruences, followed by \CREAM
for the small examples, and \cref{algorithm-enumerate} from \Semigroups for
larger values. For finding only principal congruences, again \CREAM is faster
for smaller examples, and \Semigroups is faster for larger examples.

Given that \CREAM is tool for arbitrary finite algebras of type
$(2^m, 1^n)$ (i.e. with a finite number of binary and unary
operations), it might be expected that specialised algorithms (such
as those in this article) for semigroups and monoids would always
perform better. There are two possible reasons for this.
Firstly, as mentioned above, for some examples, the number of pairs
output by \cref{algorithm-enumerate} is more or less the same as the
maximum possible number; see \cref{figure-sporadic-pairs}(a),
\cref{figure-random-transf-pairs}(a) and \cref{figure-endo-pairs}(a).
As such running \cref{algorithm-enumerate} prior to finding principal
congruences, and then computing the joins, can increase the overall
runtime. Secondly, the
implementation in \CREAM is written almost entirely in C, and the
main loop of the corresponding implementation in \Semigroups is
written in the interpreted \GAP language. Specifically, the main
loops in the implementation of \cref{algorithm-enumerate} in
\Semigroups, and the code for determining the set of principal
congruences are written in \GAP. The code for forming all joins in
\Semigroups is written in C++.
Interpreted languages typically incur a performance penalty when
compared to compiled languages.

\begin{table}
\centering
\begin{tabular}[t]{r|r|r|r|r|r|r}
  \multirow{2}{*}{$n$}
  & \multicolumn{2}{c|}{low-index}
  & \multicolumn{2}{c|}{\cref{algorithm-enumerate}}
  &  \multicolumn{2}{c}{\cref{algorithm-enumerate} + joins}
  \\\cline{2-7}
  & \multicolumn{1}{c|}{mean} & \multicolumn{1}{c|}{s.d.}
  & \multicolumn{1}{c|}{mean} & \multicolumn{1}{c|}{s.d.}
  & \multicolumn{1}{c|}{mean} & \multicolumn{1}{c}{s.d.}
  \\\hline
  1   & $2.2 \times 10 ^ {-6}$ & $1.8\times 10 ^ {-6}$ & $5.0\times
  10 ^ {-4}$ & $2.0\times 10 ^ {-4}$ & $8.0 \times 10 ^ {-4}$ &
  $4.0\times 10 ^ {-4}$ \\
  2   & $4.39\times 10 ^ {-6}$ & $0.14 \times 10 ^ {-6}$ &
  $1.6\times 10 ^ {-4}$ & $0.3 \times 10 ^ {-4}$
  & $2.2 \times 10 ^ {-4}$ & $1.1\times 10 ^ {-4}$          \\
  3   & $1.3\times 10 ^ {-5}$ & $0.2\times 10 ^ {-5}$ & $5.8\times 10
  ^ {-4}$ & $0.7 \times 10 ^ {-4}$ & $1.0 \times 10 ^ {-3}$ & $0.4
  \times 10 ^ {-3}$                  \\
  4   & $9.20\times 10 ^ {-4}$ & $0.10 \times 10 ^ {-4}$ & $5.5\times
  10 ^{-3}$ & $0.8 \times 10 ^ {-3}$ & $1.1 \times 10 ^ {-2}$ & $0.2
  \times 10 ^ {-2}$              \\\cline{6-7}
  5    & $22.500 \times 10 ^ 0$ & $0.012 \times 10 ^ 0$ & $7.0\times
  10 ^ {-2}$ & $0.6\times 10 ^ {-2}$ & \multicolumn{2}{c}{Exceeded
  available memory}     \\\cline{2-3}
  6   & \multicolumn{2}{c|}{Exceeded available time}  & $1.23 \times
  10 ^0$ & $0.11\times 10 ^ 0$          & \multicolumn{2}{c}{Exceeded
  available memory}  \\
  7   & \multicolumn{2}{c|}{Exceeded available time} & $3.0 \times 10
  ^ 1$ & $0.2 \times 10 ^ 1$  & \multicolumn{2}{c}{Exceeded available memory}
\end{tabular}
\caption{
  Runtimes in seconds for the low-index congruence algorithm
  (1 thread); \cref{algorithm-enumerate}; and
  \cref{algorithm-enumerate} and all joins of principal right
  congruences of the Catalan monoids. For
  the low-index congruences algorithm, the values are the means of
  $100$ trials using the presentation output by the Froidure-Pin
  Algorithm\cite{Froidure1997aa}. The values for
\cref{algorithm-enumerate} (and all joins) are the means of $200$ trials.}
\label{table-catalan-2}
\end{table}

\begin{table}
\centering
\begin{tabular}[t]{r|r|r|r|r}
  \multirow{2}{*}{$n$}
  & \multicolumn{2}{c|}{principal}
  & \multicolumn{2}{c}{all}\\\cline{2-5}
  & \multicolumn{1}{c|}{mean} & \multicolumn{1}{c|}{s.d.}
  & \multicolumn{1}{c|}{mean} & \multicolumn{1}{c}{s.d.}
  \\\hline
  1  &$6\times 10 ^ {-5}$ & $4\times 10 ^ {-5}$ &   $9.7 \times 10 ^
  {-5}$ & $3.0 \times 10 ^ {-4}$                      \\
  2  & $7\times 10 ^ {-5}$ & $3\times 10 ^ {-5}$ & $6.0 \times 10 ^
  {-5}$ & $4.0\times 10 ^{-5}$               \\
  3  &$9\times 10 ^ {-5}$ & $6\times 10 ^ {-5}$ &  $9.8 \times 10 ^
  {-5}$  & $6.0 \times 10 ^{-5}$                  \\
  4  & $2.2\times 10 ^ {-4}$ & $1.1\times 10 ^ {-4}$ & $4.5 \times 10
  ^ {-3}$ & $1.0 \times 10 ^ {-3}$              \\ \cline{4-5}
  5  & $2.0\times 10 ^{-3}$ & $0.7 \times 10 ^ {-3}$ &
  \multicolumn{2}{c}{Exceeded available memory}      \\\cline{2-3}
  6  &\multicolumn{2}{c|}{Seg. fault} & \multicolumn{2}{c}{Seg. fault}       \\
  7  &\multicolumn{2}{c|}{Seg. fault} & \multicolumn{2}{c}{Seg. fault}
\end{tabular}
\caption{
  The means of runtimes (in seconds) for $200$ trials of
  \textsc{CreamPrincipalCongruences} and \textsc{CreamAllCongruences}
from \CREAM applied to the Catalan monoids (1 thread).}
\label{table-catalan-3}
\end{table}

\begin{table}
\centering
\begin{tabular}[t]{r|r|r|r|r|r|r}
  \multirow{2}{*}{$n$}
  & \multicolumn{2}{c|}{low-index}
  & \multicolumn{2}{c|}{\cref{algorithm-enumerate}}
  &  \multicolumn{2}{c}{\cref{algorithm-enumerate} + joins}
  \\\cline{2-7}
  & \multicolumn{1}{c|}{mean} & \multicolumn{1}{c|}{s.d.}
  & \multicolumn{1}{c|}{mean} & \multicolumn{1}{c|}{s.d.}
  & \multicolumn{1}{c|}{mean} & \multicolumn{1}{c}{s.d.}
  \\\hline
  1  & $2.2 \times 10 ^ {-6}$ & $1.8\times 10 ^ {-6}$ & $5.0\times 10
  ^ {-4}$ & $2.0\times 10 ^ {-4}$ & $8.0 \times 10 ^ {-4}$ &
  $4.0\times 10 ^ {-4}$ \\
  2 & $5.6 \times 10 ^ {-6}$ & $0.8 \times 10 ^ {-6}$ & $1.0\times 10
  ^ {-3}$ & $0.2\times 10 ^ {-3}$ & $1.0\times 10 ^ {-3}$ &
  $0.2\times 10 ^ {-3}$ \\
  3 & $4.9 \times 10 ^ {-5}$ & $0.1\times 10 ^ {-5}$ &  $4.31\times
  10 ^ {-3}$ & $0.09\times 10 ^ {-3}$ & $4.6\times 10 ^ {-3}$ &
  $0.3\times 10 ^ {-3}$ \\
  4 & $3.555 \times 10 ^ {-2}$ & $0.009\times 10 ^ {-2}$ & $3.9\times
  10 ^ {-2}$ & $0.5\times 10 ^ {-2}$ & $5.1\times 10 ^ {-2}$ &
  $0.5\times 10 ^ {-2}$ \\
  5 & $3.364 \times 10 ^ {1}$ & $0.004\times 10 ^ {1}$ & $5.33\times
  10 ^ {-1}$ & $0.08\times 10 ^ {-1}$ & $2.763 \times 10 ^ {1}$ &
  $0.012\times 10 ^ {1}$\\ \cline{2-3}\cline{6-7}
  6 & \multicolumn{2}{c|}{Exceeded available time} & $9.39 \times 10
  ^ {0}$ & $0.19 \times 10 ^ {0}$ & \multicolumn{2}{c}{Exceeded
  available memory}\\
  7 & \multicolumn{2}{c|}{Exceeded available time} & $2.95 \times 10
  ^ 2$ & - & \multicolumn{2}{c}{Exceeded available memory}
  % the number of principal right congruences of O_7 was only
  % computed once, since it required ~36gb of memory (on my 8gb memory laptop)
\end{tabular}
\caption{
  Runtimes in seconds for low-index congruences algorithm (1 thread);
  \cref{algorithm-enumerate}; and \cref{algorithm-enumerate} and all joins of
  principal right congruences of the monoids $O_n$ of order preserving
  transformations on $\{1, \ldots, n\}$. For the low-index
  congruences algorithm,
  the values are the means of $100$ trials using the presentation
  from~\cite[Section 2]{Arthur2000aa} for $n\geq 3$ and using the output of the
  Froidure-Pin Algorithm~\cite{Froidure1997aa} for $n < 3$. The values for
\cref{algorithm-enumerate} (and all joins) are the means of $200$ trials.}
\label{table-order-endo-2}
\end{table}
\begin{table}
\centering
\begin{tabular}[t]{r|r|r|r|r}
  \multirow{2}{*}{$n$}
  & \multicolumn{2}{c|}{principal}
  & \multicolumn{2}{c}{all}\\\cline{2-5}
  & \multicolumn{1}{c|}{mean} & \multicolumn{1}{c|}{s.d.}
  & \multicolumn{1}{c|}{mean} & \multicolumn{1}{c}{s.d.}
  \\\hline
  1 & $6\times 10 ^ {-5}$             & $4\times 10 ^ {-5}$
  & $9.7 \times 10 ^ {-5}$                                      &
  $3.0 \times 10 ^ {-4}$                      \\
  2 & $2.7 \times 10 ^ {-4}$          & $0.6 \times 10 ^ {-4}$
  & $1.2  \times 10 ^ {-4}$                                     &
  $0.1 \times 10 ^ {-4}$ \\
  3 & $1.72\times 10 ^ {-4}$          & $0.06\times 10 ^ {-4}$
  & $2.05\times 10 ^ {-4}$                                      &
  $0.05\times 10 ^ {-4}$ \\
  4 & $1.475\times 10 ^ {-3}$         & $0.015\times 10 ^ {-3}$
  & $1.435\times 10 ^ {-2}$                                     &
  $0.007\times 10 ^ {-2}$\\
  5 & $1.112 \times 10 ^ {-1}$        & $0.011 \times 10 ^ {-1}$
  & $3.08\times 10 ^ {1}$                                       &
  $0.14 \times 10 ^ {1}$\\ \cline{4-5}
  6 & $1.6\times 10 ^ {1}$            & $0.1 \times 10 ^ {1}$
  & \multicolumn{2}{c}{Exceeded available memory} \\\cline{2-3}
  7 & \multicolumn{2}{c|}{Seg. fault} & \multicolumn{2}{c}{Seg. fault}
\end{tabular}
\caption{
  The means of runtimes (in seconds) for $200$ trials of
  \textsc{CreamAllPrincipalCongruences} and \textsc{CreamAllCongruences} from
  \CREAM applied to the monoids $O_n$ of order preserving transformations of
$\{1, \ldots, n\}$ (1 thread).}
\label{table-order-endo-3}
\end{table}

\subsection{Low index subgroups}

In this section, we compare the performance of: the implementation of
the low-index congruences algorithm in \libsemigroups; the
function \texttt{LowIndexSubgroups} in \GAP;
the function \texttt{permutation\_reps} from the C++ library
\textsc{3Manifolds}~\cite{Culler2025aa};
see~\cref{table-infinite}.
The authors thank Derek Holt for suggesting the examples in this
section. It is important to note that the implementation in \libsemigroups has
no optimizations for groups, and that the inputs to \libsemigroups are monoid
presentations; while the input to \cite{Culler2025aa} and \GAP are
group presentations.
For details of the particular presentations used please see
\url{https://github.com/libsemigroups/libsemigroups/tree/main/benchmarks/sims}.

\begin{table}
\centering
\begin{tabular}[t]{l||r|r|c|c|c}
  group             & index & subgroups &  \threemanifolds &\GAP &
  \libsemigroups\\\hline
  $(2, 3, 7)$-triangle &
  $50$  &
  $1,747$ &
  \cellcolor{green!25} $8.94 \times 10 ^ {-2}$ &
  \cellcolor{red!25} $4.20\times 10 ^ {1}$ &
  \cellcolor{orange!25} $1.76 \times 10 ^ {-1}$
  \\
  4-generated braid group&
  $12$ &
  $21$ &
  \cellcolor{orange!25} $9.87 \times 10 ^ {-1}$ &
  \cellcolor{red!25} $1.11 \times 10 ^ {0}$ &
  \cellcolor{green!25} $2.02 \times 10 ^ {-1}$
  \\
  modular group&
  $23$ &
  $109,859$ &
  \cellcolor{green!25} $3.79 \times 10 ^ {-1}$ &
  \cellcolor{red!25} $1.26 \times 10 ^ {2}$ &
  \cellcolor{orange!25} $4.39 \times 10 ^ {-1}$
  \\
  fundamental group of K11n34 &
  $7$ &
  $52$ &
  \cellcolor{green!25} $7.01 \times 10 ^ {-1}$ &
  \cellcolor{red!25} $1.50 \times 10 ^ {0}$&
  \cellcolor{orange!25} $1.06 \times 10 ^ {0}$\\
  Heineken            &
  $8$  &
  $3$ &
  \cellcolor{red!25} $2.26 \times 10 ^ {0}$ &
  \cellcolor{green!25} $9.5 \times 10 ^ {-1}$ &
  \cellcolor{orange!25}$1.81 \times 10 ^ {0}$\\
  fundamental group of K15n12345 &
  $7$ &
  $40$ &
  \cellcolor{green!25}$4.47 \times 10 ^ {-1}$ &
  \cellcolor{red!25}$2.12 \times 10 ^ {0}$ &
  \cellcolor{orange!25}$6.33 \times 10 ^ {-1}$ \\
  fundamental group of o9$\mbox{}_{15405}$ &
  $9$&
  $38$ &
  \cellcolor{red!25} $2.10 \times 10 ^ {0}$ &
  \cellcolor{green!25} $1.65 \times 10 ^ {0}$ &
  \cellcolor{orange!25} $1.68 \times 10 ^ {0}$
\end{tabular}
\caption{The time to determine subgroups up to some index for some
  groups. All times are in seconds.
  The indicated times are the mean
  values of between $10$ and $100$ trials. Standard deviations not
  included due to lack of space. Note that both \threemanifolds and
  \libsemigroups can be multithreaded with a corresponding decrease in
  the runtime. The fastest times are highlighted in green, the next
fastest in orange, and the slowest in red.}
\label{table-infinite}
\end{table}

\subsection{2-sided congruences of finite monoids}
% If ever returning to this section  we could add:
% * some benchmarks for CREAM
% * some benchmarks for computing the entire lattice (not only the
% principal congruences)
% * some benchmarks like those presented here for 1-sided congruences, rather
%   than only 2-sided congruences, and these could include using the low-index
%   algorithm too.

In this section we present a comparison of the performance of computing the
distinct 2-sided principal congruences of a selection of finite monoids.  We
compare several different versions of \Semigroups for \GAP that contain
different optimisations. \Semigroups v3.4.1 contains none of the optimisations
from the present paper, v4.0.2 uses \cref{algorithm-enumerate} where possible,
and v5.3.0 uses \cref{algorithm-enumerate} and some further improvements to the
code for finding distinct principal congruences.
All computations in this section are single-threaded.
We would have liked to include
a comparison with \CREAM also, but unfortunately \CREAM suffers a segmentation
fault (i.e. abnormal termination of the entire \GAP programme) on almost every
input with more than approximately $400$ elements. This could have been
circumvented, but due to time limitations, it was not. See
\cref{figure-sporadic-pairs}, \cref{figure-random-transf-pairs},
\cref{figure-endo-pairs}, and \cref{table-numbers}.

\begin{figure}
\centering
\begin{subfigure}{0.48\textwidth}
  \includegraphics[width=\textwidth]{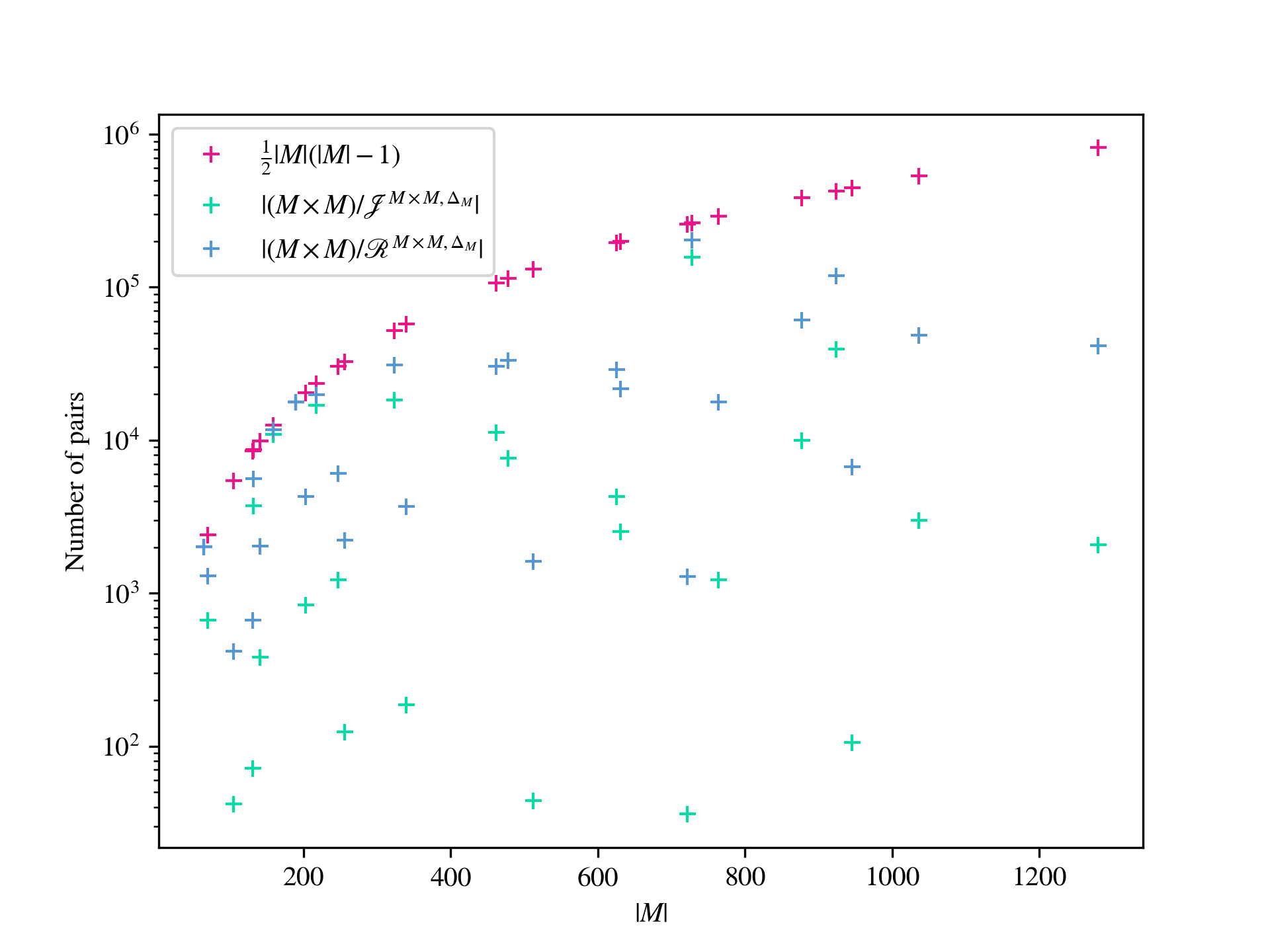}
  \caption{The number of pairs in total and output by
  \cref{algorithm-enumerate}.}
\end{subfigure}
\begin{subfigure}{0.48\textwidth}
  \includegraphics[width=\textwidth]{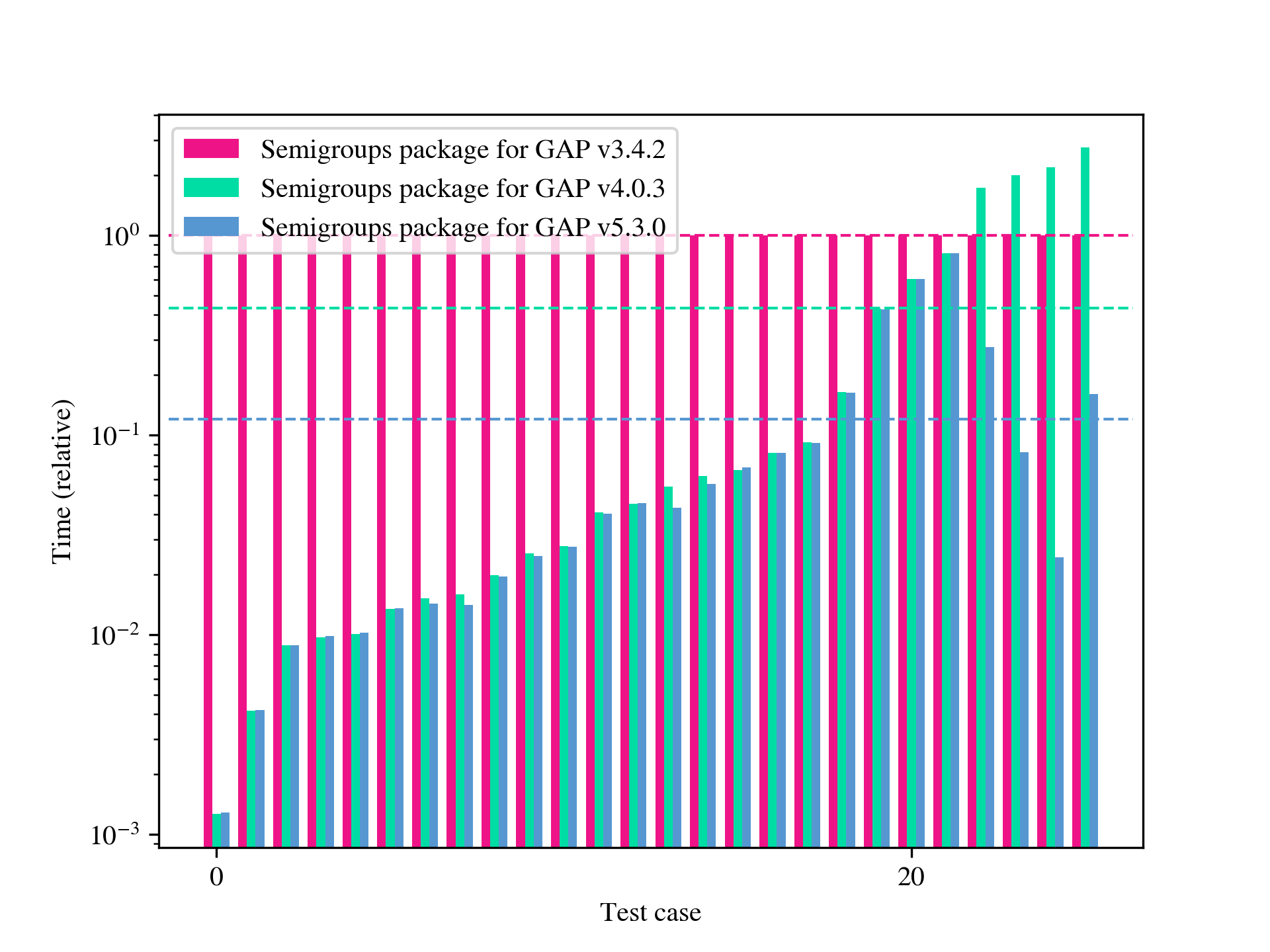}
  \caption{Relative times to find distinct principal $2$-sided congruences.}
\end{subfigure}
\caption{
  Comparison of the performance with and without the use of
  \cref{algorithm-enumerate} for $30$ sporadic
examples of semigroups and monoids with size at most ${\sim}1000$ (1 thread).}
\label{figure-sporadic-pairs}
\end{figure}

\begin{figure}
\centering
\begin{subfigure}{0.48\textwidth}
  \includegraphics[width=\textwidth]{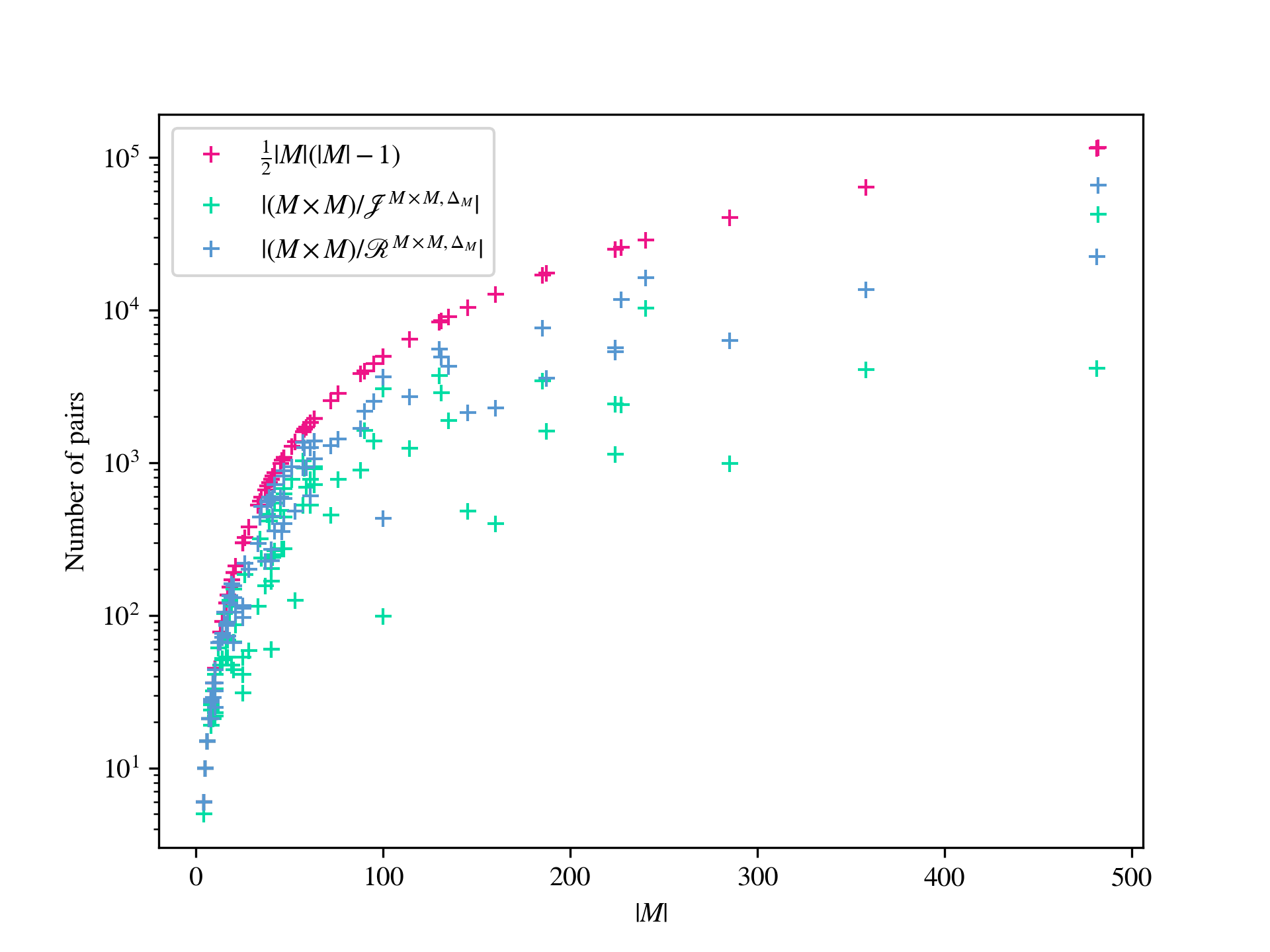}
  \caption{The number of pairs in total and output by
  \cref{algorithm-enumerate}.}
\end{subfigure}\qquad
\begin{subfigure}{0.48\textwidth}
  \includegraphics[width=\textwidth]{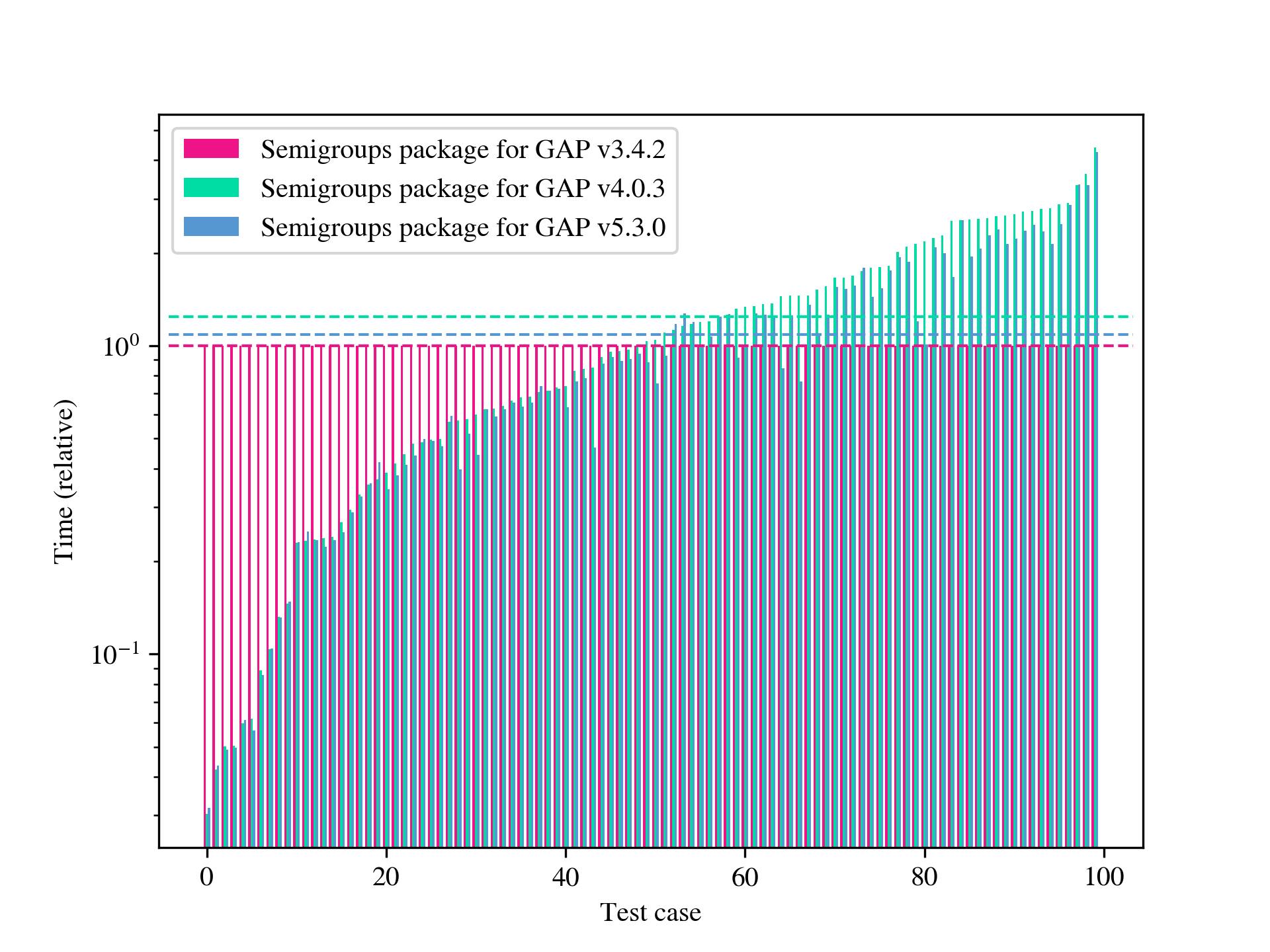}
  \caption{Relative times to find distinct principal $2$-sided congruences.}
\end{subfigure}
\caption{
  Comparison of the performance with and without the use of
  \cref{algorithm-enumerate} for 100 random 2-generated transformation
semigroups of degree $5$ (1 thread).}
\label{figure-random-transf-pairs}
\end{figure}

\begin{figure}
\centering
\begin{subfigure}{0.48\textwidth}
  \includegraphics[width=\textwidth]{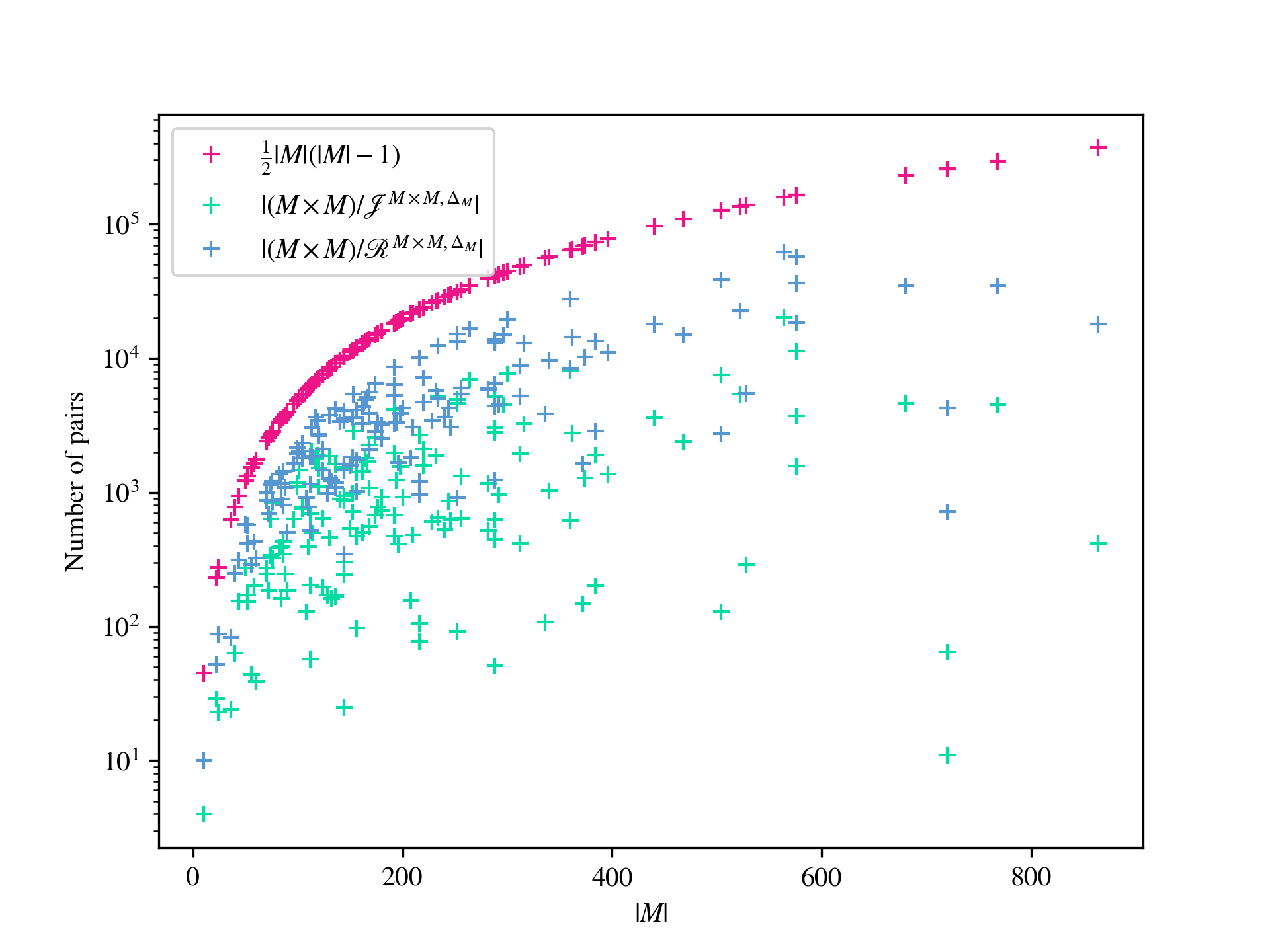}
  \caption{The number of pairs in total and output by
  \cref{algorithm-enumerate}.}
  \label{subfigure-endo-1}
\end{subfigure}\qquad
\begin{subfigure}{0.48\textwidth}
  \includegraphics[width=\textwidth]{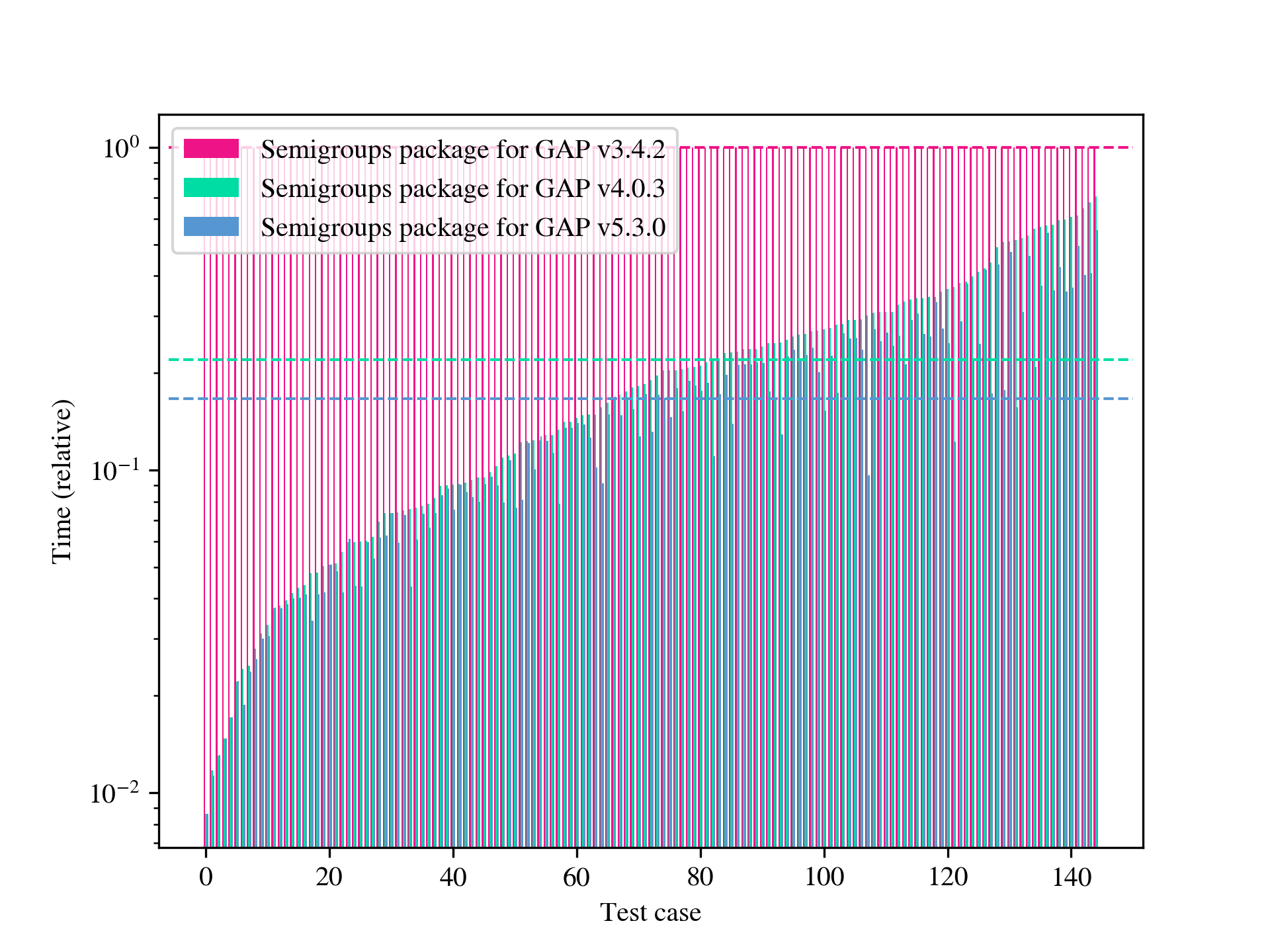}
  \caption{Relative times to find distinct principal $2$-sided congruences.}
  \label{subfigure-endo-2}
\end{subfigure}
\caption{
  Comparison of the performance with and without the use of
  \cref{algorithm-enumerate} for  selected endomorphism monoids of
graphs with 6 nodes (1 thread).}
\label{figure-endo-pairs}
\end{figure}

\begin{landscape}
\begin{table}
  \centering
  \begin{tabular}{l|l|l|l|l|l|l|l|l|l}
    \multirow{2}{*}{Monoid = $M_n$}
    & \multirow{2}{*}{$n$}
    & \multirow{2}{*}{$|M_n|$}
    & \multirow{2}{*}{$|\mathcal{L}|$}
    & \multicolumn{2}{c|}{\CREAM}
    & \multicolumn{2}{c|}{\Semigroups}
    & \multicolumn{2}{c}{\libsemigroups}
    \\\cline{5-10}
    & & &
    & \multicolumn{1}{c|}{mean}
    & \multicolumn{1}{c|}{s.d.}
    & \multicolumn{1}{c|}{mean}
    & \multicolumn{1}{c|}{s.d.}
    & \multicolumn{1}{c|}{mean}
    & \multicolumn{1}{c}{s.d.}
    \\\hline\hline
    Gossip  & $3$ & $11$ & $84$ & \cellcolor{orange!25}$6.0 \times 10
    ^{-4}$ & \cellcolor{orange!25}$1.0 \times 10 ^ {-4}$ &
    \cellcolor{red!25}$5.2 \times 10 ^ {-3}$ &
    \cellcolor{red!25}$0.2\times 10 ^{-3}$       &
    \cellcolor{green!25}$1.87 \times 10 ^ {-4}$ &
    \cellcolor{green!25}$0.08\times 10 ^ {-4}$ \\\hline
    Jones  & $4$ & $14$ & $9$ & \cellcolor{orange!25}$5.0 \times 10
    ^{-4}$ & \cellcolor{orange!25}$1.0 \times 10 ^ {-4}$ &
    \cellcolor{red!25}$1.2
    \times 10 ^ {-2}$ & \cellcolor{red!25}$0.2\times 10 ^{-2}$ &
    \cellcolor{green!25}$3.5 \times 10 ^ {-5}$ &
    \cellcolor{green!25}$0.2 \times 10 ^ {-5}$  \\\hline
    Brauer & $2$ & $3$ & $3$ & \cellcolor{orange!25}$3.7\times 10
    ^{-4}$ & \cellcolor{orange!25}$0.6\times
    10 ^ {-4}$ & \cellcolor{red!25}$2.1\times 10 ^ {-3}$ &
    \cellcolor{red!25}$0.4\times 10 ^{-3}$ & \cellcolor{green!25}$3.8
    \times 10 ^ {-6}$ & \cellcolor{green!25}$0.4 \times 10 ^ {-6}$\\\hline
    Partition  & $2$ & $15$ & $13$ & \cellcolor{orange!25}$4.5 \times
    10 ^{-4}$ & \cellcolor{orange!25}$0.8\times 10 ^ {-4}$
    & \cellcolor{red!25}$6.9\times 10 ^ {-3}$ &
    \cellcolor{red!25}$0.9\times 10 ^{-3}$ & \cellcolor{green!25}$3.4
    \times 10 ^ {-5}$  & \cellcolor{green!25}$0.5 \times 10 ^ {-5}$\\\hline
    Full PBR  & $1$ & $16$ & $167$ & \cellcolor{orange!25}$1.4\times
    10 ^{-3}$ & \cellcolor{orange!25}$0.2\times 10 ^ {-3}$ &
    \cellcolor{red!25}$1.2\times 10 ^ {-2}$ &
    \cellcolor{red!25}$0.2\times 10 ^{-2}$ &
    \cellcolor{green!25}$2.73 \times 10 ^ {-4}$ &
    \cellcolor{green!25}$0.07 \times 10 ^ {-4}$            \\\hline
    Symmetric group & $4$ & $24$ & $4$ & \cellcolor{orange!25}$1.4
    \times 10 ^{-3}$ & \cellcolor{orange!25}$0.3 \times 10 ^{-3}$ &
    \cellcolor{red!25}$4.0 \times 10 ^ {-3}$ &
    \cellcolor{red!25}$0.3\times 10 ^{-3}$     &
    \cellcolor{green!25}$5.5 \times 10 ^ {-5}$ &
    \cellcolor{green!25}$0.8 \times 10 ^ {-5}$             \\\hline
    Full transformation  & $3$ & $27$ & $7$ &
    \cellcolor{orange!25}$1.6\times 10 ^{-3}$ &
    \cellcolor{orange!25}$0.3 \times 10 ^ {-3}$ &
    \cellcolor{red!25}$5.3\times 10 ^ {-3}$ & \cellcolor{red!25}$0.6\times 10 ^
    {-3}$ & \cellcolor{green!25}$8.6 \times 10 ^ {-5}$ &
    \cellcolor{green!25}$0.9 \times 10 ^ {-5}$\\\hline
    Symmetric inverse  & $3$ & $34$ & $7$ &
    \cellcolor{orange!25}$3.8\times 10 ^{-3}$ &
    \cellcolor{orange!25}$0.4 \times 10 ^ {-3}$ &
    \cellcolor{red!25}$8.7\times 10 ^ {-3}$ &
    \cellcolor{red!25}$0.9\times 10 ^{-\
    3}$ & \cellcolor{green!25}$9.6 \times 10 ^ {-5}$ &
    \cellcolor{green!25}$0.4 \times 10 ^ {-5}$\\\hline
    Jones  & $5$ & $42$ & $6$ & \cellcolor{orange!25}$7.5 \times 10 ^{-3}$ &
    \cellcolor{orange!25}$0.2\times 10 ^ {-3}$ & \cellcolor{red!25}$9
    .0 \times 10 ^ {-2}$ & \cellcolor{red!25}$8.0\times 10 ^{-2}$ &
    \cellcolor{green!25}$3.8 \times 10 ^ {-4}$ &
    \cellcolor{green!25}$0.1\times 10 ^ {-4}$
    \\\hline
    Motzkin  & $3$ & $51$ & $10$ & \cellcolor{orange!25}$1.5 \times
    10 ^{-2}$ & \cellcolor{orange!25}$0.1\times 10 ^ {-2}$ &
    \cellcolor{red!25}$1.4 \times 10 ^ {-1}$ &
    \cellcolor{red!25}$0.9\times 10 ^{-1}$ & \cellcolor{green!25}$8.2
    \times 10 ^ {-3}$ & \cellcolor{green!25}$0.05 \times 10 ^ {-3}$ \\\hline
    Partial transformation & $3$ & $64$ & $7$
    &\cellcolor{red!25}$3.8\times 10 ^{-2}$ & \cellcolor{red!25}$0.1
    \times 10 ^ {-2}$ &\cellcolor{orange!25}$2.0 \times 10 ^ {-2}$ &
    \cellcolor{orange!25}$1.0\times 10 ^{-2}$   &
    \cellcolor{green!25}$1.96 \times 10 ^ {-3}$ &
    \cellcolor{green!25}$0.02 \times 10 ^ {-3}$\\\hline
    Partial brauer & $3$ & $76$ & $16$ & \cellcolor{red!25}$6.8\times
    10 ^{-2}$ & \cellcolor{red!25}$0.5 \times 10 ^ {-2}$ &
    \cellcolor{orange!25}$5.4\times 10 ^ {-2}$ &
    \cellcolor{orange!25}$0.3\times 10 ^{-2}$ &
    \cellcolor{green!25}$5.06 \times 10 ^ {-3}$ &
    \cellcolor{green!25}$0.05 \times 10 ^ {-3}$ \\\hline
    Brauer  & $4$ & $105$ & $19$ & \cellcolor{red!25}$2.52 \times 10
    ^{-1}$ & \cellcolor{red!25}$0.06\times 10 ^ {-1}$\
    & \cellcolor{orange!25}$2.0 \times 10 ^ {-2}$ &
    \cellcolor{orange!25}$0.1\times 10 ^{-2}$&
    \cellcolor{green!25}$4.3 \times 10 ^ {-3}$ &
    \cellcolor{green!25}$0.1 \times 10 ^ {-3}$\\\hline
    Symmetric group& $5$ & $120$ & $3$ &
    \cellcolor{red!25}$4.60\times 10 ^{-1}$ &
    \cellcolor{red!25}$0.08\times 10 ^ {-1}$ &
    \cellcolor{orange!25}$2.9\times 10 ^ {-2}$ &
    \cellcolor{orange!25}$0.3\times 10 ^{-2}$   &
    \cellcolor{green!25}$1.27 \times 10 ^ {-2}$ &
    \cellcolor{green!25}$0.01 \times 10 ^ {-2}$                   \\\hline
    Jones & $6$ & $132$ & $10$ & \cellcolor{orange!25}$7.14\times 10
    ^{-1}$ & \cellcolor{orange!25}$0.08\times 10 ^ {-1}$ &
    \cellcolor{red!25}$1.001\times 10 ^ {0}$ &
    \cellcolor{red!25}$0.007\times 10 ^{0}$  &
    \cellcolor{green!25}$6.22 \times 10 ^ {-3}$ &
    \cellcolor{green!25}$0.03\times 10 ^{-3}$              \\\hline
    Partition & $3$ & $203$ & $16$ & \cellcolor{red!25}$4.08\times 10
    ^{0}$ & \cellcolor{red!25}$0.02 \times 10 ^ {0}$
    & \cellcolor{orange!25}$2.7\times 10 ^ {-1}$ &
    \cellcolor{orange!25}$0.2\times 10 ^{-1}$ &
    \cellcolor{green!25}$2.82 \times 10 ^ {-2}$  &
    \cellcolor{green!25}$0.02 \times 10 ^ {-2}$               \\\hline
    Symmetric inverse  & $4$ & $209$ & $11$ &
    \cellcolor{red!25}$4.65\times 10 ^{0}$ & \cellcolor{red!25}$0.02 \times
    10 ^ {0}$ & \cellcolor{orange!25}$7.0 \times 10 ^ {-2}$ &
    \cellcolor{orange!25}$0.5\times 10 ^{-2}$         &
    \cellcolor{green!25}$4.34 \times 10 ^ {-3}$ &
    \cellcolor{green!25}$0.04 \times 10 ^ {-3}$            \\\hline
    Full transformation  & $4$ & $256$ & $11$ &
    \cellcolor{red!25}$9.65\times 10 ^{0}$ & \cellcolor{red!25}$0.05
    \times 10 ^ {0}$ & \cellcolor{green!25}$7.7\times 10 ^ {-2}$ &
    \cellcolor{green!25}$0.4\times 10  ^{-2}$
    &\cellcolor{orange!25}$1.007 \times 10 ^ {-1}$ &
    \cellcolor{orange!25}$0.006 \times 10 ^ {-1}$              \\\hline
    Motzkin  & $4$ & $323$ & $11$ & \cellcolor{orange!25}$2.54 \times
    10 ^{1}$ & \cellcolor{orange!25}$0.01 \times 10 ^ {1}$ &
    \cellcolor{green!25}$8.9\times 10 ^ {0}$ &
    \cellcolor{green!25}$0.1\times 10 ^{0}$                &
    \cellcolor{red!25}$3.38 \times 10 ^ {1}$ &
    \cellcolor{red!25}$0.03 \times 10 ^ {1}$\\\hline
    Jones  & $7$ & $429$ & $7$ & \cellcolor{red!25}$8.56\times 10
    ^{1}$ & \cellcolor{red!25}$0.02 \times 10 ^ {1}$ & \cellcolor{orange!25}$1.
    72 \times 10 ^ {1}$ & \cellcolor{orange!25}$0.02\times 10 ^{1}$
    &\cellcolor{green!25}$1.54 \times 10 ^ {-1}$&
    \cellcolor{green!25}$0.02\times 10 ^{-1}$
  \end{tabular}
  \caption{Performance comparison of \CREAM, the algorithm described
    in~\cref{section-algo-2} as implemented in the \Semigroups package for
    \GAP, and the algorithm described in \cref{subsec-2-sided-congs} as
    implemented in \libsemigroups for computing all 2-sided congruences
    (1 thread). All times are in seconds, and the fastest times are
    highlighted in green, the next fastest in orange, and the slowest
  in red.}\label{table-numbers}
\end{table}
\end{landscape}

\section{Congruence statistics}\label{appendix-numbers}

In this appendix we provide the numbers of finite index congruences of some
well-known finite and infinite finitely presented monoids in a number of
tables. By searching for these numbers in the \cite{oeis} and in the
literature, it appears that many of them were not previously known.

\newcommand{\oeis}[1]{\href{https://oeis.org/#1}{#1}}
\newcommand{\searchfive}[5]{\href{https://oeis.org/search?q=#1\%2C+#2\%2C+#3\%2C+#4\%2C+#5}{OEIS}}
\newcommand{\searchfour}[4]{\href{https://oeis.org/search?q=#1\%2C+#2\%2C+#3\%2C+#4}{OEIS}}
\newcommand{\searchthree}[3]{\href{https://oeis.org/search?q=#1\%2C+#2\%2C+#3}{OEIS}}

\begin{table}
\newcommand{\ncard}{$(2n)!/(n!(n+1)!)$}
\newcommand{\sprinc}{\searchfour{1}{8}{67}{641}}
\newcommand{\smin}{\searchfour{3}{6}{10}{15}}
\newcommand{\nmin}{$(n+1)(n+2)/2$}
\newcommand{\sall}{\searchfour{1}{2}{11}{575}}
\centering
\begin{tabular}[t]{r|r|r|r|r}
  $n$     & $|C_n|$        & principal & minimal & all        \\\hline\hline
  1       & 1              & 0         & 0       & 1         \\
  2       & 2              & 1         & 0       & 2         \\
  3       & 5              & 8         & 3       & 11        \\
  4       & 14             & 67        & 6       & 575       \\
  5       & 42             & 641       & 10      & 5,295,135 \\
  6       & 132            & 6,790     & 15      & ? \\
  7       & 429            & 76,568    & 21      & ?         \\
  \vdots  & \vdots         & \vdots    & \vdots  & \vdots    \\\hline
  $n$     & \ncard         & ?         & \nmin?  & ?         \\\hline
  & \oeis{A000108} & \sprinc   & \smin   & \sall
\end{tabular}
\caption{The numbers of principal, minimal, and all right congruences of the
Catalan monoids $C_n$~\cite{Higgins1993aa} for some small values of $n$.}
\label{table-catalan-1}
\end{table}

\begin{table}
\centering
\begingroup
\newcommand{\ncard}{$\binom{2n+1}{n+1}$}
\newcommand{\sprinc}{\searchfive{0}{2}{18}{138}{1055}}
\newcommand{\sall}{\searchfive{1}{3}{31}{2634}{6964196}}
\begin{tabular}[t]{r|r|r|r|r}
  \multicolumn{2}{c}{} & \multicolumn{3}{c}{left congruences} \\

  $n$   & $|O_n|$        & principal & minimal        & all       \\\hline\hline
  1     & 1              & 0         & 0              & 1         \\
  2     & 3              & 2         & 1              & 3         \\
  3     & 10             & 18        & 3              & 31        \\
  4     & 35             & 138       & 6              & 2,634     \\
  5     & 126            & 1,055     & 10             & 6,964,196 \\
  6     & 462            & 8,234     & 15             & ?         \\
  7     & 1,716          & ?         & ?              & ?         \\
  \vdots& \vdots         & \vdots    & \vdots         & \vdots    \\\hline
  $n$   & \ncard         & ?         & $n(n+1)/2$?    & ?         \\\hline
  & \oeis{A001700} & \sprinc   & \oeis{A253145} & \sall     \\
\end{tabular}
\endgroup
\begingroup
\newcommand{\sprinc}{\searchfive{0}{3}{18}{116}{853}}
\newcommand{\sall}{\searchfive{1}{5}{25}{385}{37951}}
\begin{tabular}[t]{r|r|r}
  \multicolumn{3}{c}{right congruences} \\

  principal    & minimal        & all            \\\hline\hline
  0            & 0              & 1              \\
  3            & 3              & 5              \\
  18           & 6              & 25             \\
  116          & 10             & 385            \\
  853          & 15             & 37,951         \\
  6,707        & 21             &                \\
  54,494       & 28             & ?              \\
  \vdots       & \vdots         & \vdots         \\\hline
  ? & $n(n+1)/2$?    & ?              \\\hline
  \sprinc      & \oeis{A253145} & \sall
\end{tabular}
\endgroup
\caption{The numbers of principal, minimal, and all left and right
  congruences of the monoids $O_n$ of order-preserving transformations of an
$n$-chain for some small values of $n$.}
\label{table-order-endo}
\end{table}

\begin{table}
\centering
\begingroup
\newcommand{\sprinc}{\searchfive{0}{3}{27}{222}{1831}}
\newcommand{\sall}{\searchfour{1}{4}{94}{32571}}
\newcommand{\smin}{\searchfive{0}{1}{3}{6}{10}}
\begin{tabular}[t]{r|r|r|r|r}
  \multicolumn{2}{c}{} & \multicolumn{3}{c}{left congruences} \\

  $n$    & $|OR_n|$        & principal & minimal       & all    \\\hline\hline
  1      & 1               & 0         & 0             & 1      \\
  2      & 4               & 3         & 1             & 4      \\
  3      & 17              & 27        & 3             & 94     \\
  4      & 66              & 222       & 6             & 32,571 \\
  5      & 247             & 1,831     & 10            & ?      \\
  6      & 918             & 15,137    & 15            & ?      \\
  \vdots & \vdots          & \vdots    & \vdots        & \vdots \\\hline
  $n$    & $\binom{2n}{n}$ & ?         & $n(n + 1)/2$? & -      \\\hline
  & \oeis{A045992}  & \sprinc   & \smin         & \sall  \\
\end{tabular}
\endgroup
\begingroup
\newcommand{\sprinc}{\searchfive{0}{4}{25}{176}{1382}}
\newcommand{\sall}{\searchfour{1}{7}{54}{1335}}
\begin{tabular}[t]{r|r|r}
  \multicolumn{3}{c}{right congruences} \\
  principal & minimal  & all    \\\hline\hline
  0         & 0        & 1      \\
  4         & 4        & 7      \\
  25        & 9        & 54     \\
  176       & 16       & 1,335  \\
  1,382     & 25       & ?      \\
  11,575    & 36       & ?      \\
  \vdots    & \vdots   & \vdots \\\hline
  ?         & $n ^ 2$? & ?      \\\hline
  \sprinc   &          & \sall
\end{tabular}
\endgroup
\caption{The numbers of principal, minimal, and all left and right
  congruences of the monoids $OR_n$ of order-preserving or -reversing
transformations of an $n$-chain for some small values of $n$.}
\label{table-order-anti}
\end{table}

\begin{table}
\centering
\begingroup
\newcommand{\ncard}{$2\sum_{k=0}^{n-1} \binom{n-1}{k}\binom{n+k}{k}$}
\newcommand{\smax}{\searchfour{1}{4}{13}{40}}
\newcommand{\smin}{\searchfour{0}{3}{6}{10}}
\newcommand{\scard}{\searchfour{2}{8}{38}{192}}
\newcommand{\sprinc}{\searchfour{1}{8}{67}{653}}
\newcommand{\sall}{\searchfour{2}{9}{142}{16239}}
\begin{tabular}[t]{r|r|r|r|r}
  \multicolumn{2}{c}{} & \multicolumn{3}{c}{left congruences} \\

  $n$    & $|PO_n|$       & principal & minimal & all    \\\hline\hline
  1      & 2              & 1         & 0       & 2      \\
  2      & 8              & 8         & 3       & 9      \\
  3      & 38             & 67        & 6       & 142    \\
  4      & 192            & 653       & 10      & 16,239 \\
  5      & 1,002          & 7,314     & 15      & ?      \\
  \vdots & \vdots         & \vdots    & \vdots  & \vdots \\\hline
  $n$    & \ncard         & ?         & ?       & ?      \\\hline
  & \oeis{A002003} & \sprinc   & \smin   & \sall  \\
\end{tabular}
\endgroup
\begingroup
\newcommand{\sprinc}{\searchfour{1}{12}{172}{2612}}
\newcommand{\smin}{\searchfive{0}{6}{28}{120}{496}}
\newcommand{\sall}{\searchthree{2}{18}{10036}}
\centering
\begin{tabular}[t]{r|r|r}
  \multicolumn{3}{c}{right congruences} \\

  principal & minimal  & all             \\\hline\hline
  1         & 0        & 2               \\
  12        & 6        & 18              \\
  172       & 28       & 10,036          \\
  2,612     & 120      & ?               \\
  40,074    & 496      & ?               \\
  \vdots    & \vdots   & \vdots          \\\hline
  ?         & ?        & ?               \\\hline
  \sprinc   & \smin    & \sall           \\
\end{tabular}
\endgroup
\caption{The numbers of principal, minimal, and all left and right
  congruences of the monoids $PO_n$ of partial order-preserving
transformations of an $n$-chain for some small values of $n$.}
\label{table-partial-order}
\end{table}

\begin{table}
\centering
\begingroup
\newcommand{\scard}{\searchfour{2}{9}{54}{323}}
\newcommand{\sprinc}{\searchfour{1}{9}{79}{874}}
\newcommand{\smin}{\searchfour{0}{3}{6}{10}}
\newcommand{\sall}{\searchfour{2}{11}{306}{104400}}
\newcommand{\ncard}{$2|PO_n|-(1+\sum_{k=1}^n \binom{n}{k} n)$}
\begin{tabular}[t]{r|r|r|r|r}
  \multicolumn{2}{c}{} & \multicolumn{3}{c}{left congruences} \\

  $n$     & $|POR_n|$ & principal & minimal  & all     \\\hline\hline
  1       & 2         & 1         & 0        & 2       \\
  2       & 9         & 9         & 3        & 11      \\
  3       & 54        & 79        & 6        & 306     \\
  4       & 323       & 874       & 10       & 104,400 \\
  5       & 1,848     & ?         & ?        & ?       \\
  \vdots  & \vdots    & \vdots    & \vdots   & \vdots  \\\hline
  $n$     & \ncard   & ?         & ?        & ?       \\\hline
  & \scard    & \sprinc   & \smin    & \sall   \\
\end{tabular}
\endgroup
\begingroup
\newcommand{\sprinc}{\searchfour{1}{13}{191}{3195}}
\newcommand{\smin}{\searchfour{0}{6}{28}{120}}
\newcommand{\sall}{\searchthree{2}{23}{35598}}
\begin{tabular}[t]{r|r|r}
  \multicolumn{3}{c}{right congruences} \\
  principal & minimal & all    \\\hline\hline
  1         & 0       & 2      \\
  13        & 6       & 23     \\
  191       & 28      & 35,598 \\
  3,195     & 120     & ?      \\
  54,785    & 496     & ?      \\
  \vdots    & \vdots  & \vdots \\\hline
  ?         & ?       & ?      \\\hline
  \sprinc   & \smin   & \sall  \\
\end{tabular}
\endgroup
\caption{The numbers of principal, minimal, and all left and right
  congruences of the monoids $POR_n$ of partial order-preserving
or -reversing transformations of an $n$-chain for some small values of $n$.}
\label{table-partial-order-anti}
\end{table}

\begin{table}
\centering
\begingroup
\newcommand{\scard}{$\binom{2n}{n}$}
\newcommand{\sprinc}{\searchfour{1}{7}{46}{330}}
\newcommand{\smin}{\searchfour{0}{3}{6}{10}}
\newcommand{\nmin}{$(n+1)(n+2)/2$}
\newcommand{\sall}{\searchfour{2}{8}{99}{8146}}
\begin{tabular}[t]{r|r|r|r|r}
  $n$   & $|POI_n|$      & principal & minimal        & all       \\\hline\hline
  1     & 2              & 1         & 0              & 2         \\
  2     & 6              & 7         & 3              & 8         \\
  3     & 20             & 46        & 6              & 99        \\
  4     & 70             & 330       & 10             & 8,146     \\
  5     & 252            & 2,602     & 15             & 18,732,669\\
  6     & 924            & 21,900    & 21             & ?         \\
  \vdots& \vdots         & \vdots    & \vdots         & \vdots    \\\hline
  $n$   & \scard         & ?         & \nmin          & -         \\\hline
  & \oeis{A000984} & \sprinc   & \oeis{A253145} & \sall \\
\end{tabular}
\endgroup
\begingroup
\newcommand{\scard}{\searchfour{2}{7}{30}{123}}
\newcommand{\sprinc}{\searchfour{1}{8}{56}{453}}
\newcommand{\smin}{\searchfour{0}{3}{6}{10}}
\newcommand{\nprinc}{$(n+1)(n+2)/2$}
\newcommand{\sall}{\searchfour{2}{10}{232}{64520}}
\begin{tabular}[t]{r|r|r|r|r}
  $n$   & $|PODI_n|$ & principal & minimal & all    \\\hline\hline
  1     & 2          & 1         & 0       & 2      \\
  2     & 7          & 8         & 3       & 10     \\
  3     & 30         & 56        & 6       & 232    \\
  4     & 123        & 453       & 10      & 64,520 \\
  5     & 478        & 4,032     & 15      & ?      \\
  6     & 1,811      & 37,410    & 21      & ?      \\
  \vdots& \vdots     & \vdots    & \vdots  & \vdots \\\hline
  $n$   &            & ?         & \nprinc & ?      \\\hline
  & \scard     & \sprinc   & \smin   & \sall  \\
\end{tabular}
\endgroup
\caption{The numbers of principal, minimal, and all left/right congruences of
  the monoids $POI_n$ and $PODI_n$ of order-preserving, and order-preserving
  and -reversing, respectively, partial permutations of an $n$-chain
for some small values of $n$.}
\label{table-poi-podi}
\end{table}

\begin{table}
\centering
\begingroup
\newcommand{\smin}{\searchfour{0}{3}{6}{10}}
\newcommand{\nprinc}{$(n+1)(n+2)/2$}
\newcommand{\sprinc}{\searchfour{1}{8}{56}{460}}
\newcommand{\ncard}{$(n/2)\binom{2n}{n} + 1$}
\newcommand{\sall}{\searchfour{2}{10}{220}{57357}}
\begin{tabular}[t]{r|r|r|r|r}
  $n$     & $|POPI_n|$     & principal & minimal  & all    \\\hline\hline
  1       & 2              & 1         & 0        & 2      \\
  2       & 7              & 8         & 3        & 10     \\
  3       & 31             & 56        & 6        & 220    \\
  4       & 141            & 460       & 10       & 57,357 \\
  5       & 631            & 4,322     & 15       & ?      \\
  \vdots  & \vdots         & \vdots    & \vdots   & \vdots \\\hline
  $n$     &          & ?         & \nprinc? & ?      \\\hline
  & \oeis{A289719} & \sprinc   & \smin    & \sall  \\
\end{tabular}
\endgroup
\begingroup
\newcommand{\sprinc}{\searchfour{1}{8}{59}{506}}
\newcommand{\nmin}{$(n+1)(n+2)/2$}
\newcommand{\smin}{\searchfour{0}{3}{6}{10}}
\newcommand{\sall}{\searchfour{2}{10}{274}{188740}}
\begin{tabular}[t]{r|r|r|r|r}
  $n$    & $|PORI_n|$     & principal & minimal & all     \\\hline\hline
  1      & 2              & 1         & 0       & 2       \\
  2      & 7              & 8         & 3       & 10      \\
  3      & 34             & 59        & 6       & 274     \\
  4      & 193            & 506       & 10      & 188,740 \\
  5      & 1,036          & 5,347     & 15      & ?       \\
  \vdots & \vdots         & \vdots    & \vdots  & \vdots  \\\hline
  $n$    &                & ?         & \nmin?  & ?       \\\hline
  & \oeis{A289720} & \sprinc   & \smin   & ?       \\
\end{tabular}
\endgroup
\caption{The numbers of principal, minimal, and all left/right congruences of
  the monoids $POPI_n$ and $PORI_n$ of orientation-preserving, and
  orientation-preserving or -reversing, respectively, partial permutations of an
$n$-chain for some small values of $n$.}
\label{table-popi-pori}
\end{table}

\begin{table}
\centering
\begingroup
\newcommand{\ncard}{$(n + 1) ^n$}
\newcommand{\sprinc}{\searchfour{1}{9}{84}{1086}}
\newcommand{\smin}{\searchfour{1}{3}{6}{10}}
\begin{tabular}[t]{r|r|r|r|r}
  \multicolumn{2}{c}{} & \multicolumn{3}{c}{left congruences} \\

  $n$      & $|PT_n|$       & principal & minimal     & all     \\\hline\hline
  1        & 2              & 1         & 1           & 2 \\
  2        & 9              & 9         & 3           & 11 \\
  3        & 64             & 84        & 6           & 371 \\
  4        & 625            & 1,086     & 10          & 335,497 \\
  \vdots   & \vdots         & \vdots    & \vdots      & \vdots  \\\hline
  $n$      & \ncard         & ?         & $n(n+1)/2$? & ?       \\\hline
  & \oeis{A289720} & \sprinc   & \smin       & ?       \\
\end{tabular}
\endgroup
\begingroup
\newcommand{\sprinc}{\searchfour{1}{13}{237}{6398}}
\newcommand{\smin}{\searchfour{1}{6}{28}{120}}
\newcommand{\sall}{\searchthree{2}{23}{92703}}
\begin{tabular}[t]{r|r|r}
  \multicolumn{3}{c}{right congruences} \\

  principal & minimal & all     \\\hline\hline
  1         & 1       & 2 \\
  13        & 6       & 23 \\
  237       & 28      & 92,703 \\
  6,398     & 120     & ? \\
  \vdots    & \vdots  & \vdots  \\\hline
  ?         & ?       & ?    \\\hline
  \sprinc   & \smin   & \sall    \\
\end{tabular}
\endgroup
\caption{The numbers of principal, minimal, and all left and right
  congruences of
the monoids $PT_n$ of all partial transformations on an $n$-set.}
\label{table-pt}
\end{table}

\begin{table}
\centering
\begingroup
\newcommand{\sprinc}{\searchfive{0}{3}{32}{370}{5892}}
\newcommand{\smin}{\searchfive{0}{1}{3}{6}{10}}
\newcommand{\nmin}{$n(n+1)/2$}
\newcommand{\sall}{\searchfour{1}{4}{120}{120121}}
\begin{tabular}[t]{l||r|r|r|r}
  \multicolumn{2}{c}{} & \multicolumn{3}{c}{left congruences} \\

  $n$   & $|T_n|$        & principal & minimal & all      \\\hline\hline
  1    & 1              & 0         & 0       & 1        \\
  2    & 4              & 3         & 1       & 4        \\
  3    & 27             & 32        & 3       & 120      \\
  4    & 256            & 370       & 6       & 120,121  \\
  5    & 3,125          & 5,892     & 10      & ? \\
  \vdots& \vdots         & \vdots    & \vdots  & \vdots  \\\hline
  $n$   & $n^n$          & ?         & \nmin?  & ? \\\hline
  & \oeis{A000312} & \sprinc   & \smin   & \sall \\
\end{tabular}
\endgroup
\begingroup
\newcommand{\sprinc}{\searchfive{0}{4}{44}{900}{28647}}
\newcommand{\smin}{\searchfive{0}{4}{16}{64}{256}}
\newcommand{\nmin}{$2^{2n-2}$}
\newcommand{\sall}{\searchfour{1}{7}{287}{22069828}}
\begin{tabular}[t]{r|r|r}
  \multicolumn{3}{c}{right congruences} \\
  principal  & minimal & all     \\\hline\hline
  0      & 0       & 1 \\
  4      & 4       & 7 \\
  44     & 16      & 287 \\
  900    & 64      & 22,069,828 \\
  28,647 & 256     & ? \\
  \vdots     & \vdots  & \vdots  \\\hline
  ?          & \nmin   & ?    \\\hline
  \sprinc    & \smin   & \sall    \\
\end{tabular}
\endgroup
\caption{The numbers of principal, minimal, and all left and right
  congruences of
the monoids $T_n$ of all transformations on an $n$-set.}
\label{table-t}
\end{table}

\begin{table}
\centering
\begingroup
\newcommand{\ncard}{$\sum_{k=0}^n\binom{n}{k}^2k!$}
\newcommand{\smin}{\searchfive{0}{1}{3}{6}{10}}
\newcommand{\nmin}{$n(n+1)/2$}
\newcommand{\sall}{\searchfour{1}{10}{274}{195709}}
\newcommand{\sprinc}{\searchfive{1}{8}{59}{516}{5667}}
\begin{tabular}[t]{l|r|r|r|r|r}
  $n$    & $|I_n|$        & principal & minimal & all \\ \hline\hline
  1      & 2              & 1         & 1       & 2       \\
  2      & 7              & 8         & 3       & 10      \\
  3      & 34             & 59        & 6       & 274     \\
  4      & 209            & 516       & 10      & 195,709 \\
  5      & 1,546          & 5,667     & 15      & ?       \\
  \vdots & \vdots         & \vdots    & \vdots  & \vdots  \\\hline
  $n$    & \ncard         & ?         & \nmin?  & ? \\\hline
  & \oeis{A002720} & \sprinc   & \smin   & \sall \\
\end{tabular}
\endgroup
\caption{The numbers of principal, minimal, and all left and right
  congruences of
the symmetric inverse monoids $I_n$ of all partial permutations on an $n$-set.}
\label{table-i}
\end{table}

\begin{table}
\centering
\begingroup
\newcommand{\ncard}{$(2n)!/(n!(n+1)!)$}
\newcommand{\sprinc}{\searchfive{0}{1}{6}{30}{118}}
\newcommand{\smin}{\searchfive{3}{7}{15}{29}{105}}
\newcommand{\sall}{\searchfive{1}{2}{9}{79}{2157}}
\begin{tabular}[t]{l|r|r|r|r}
  $n$    & $|J_n|$        & principal & minimal & all \\\hline\hline
  1      & 1              & 0         & 0       & 1         \\
  2      & 2              & 1         & 1       & 2         \\
  3      & 5              & 6         & 3       & 9         \\
  4      & 14             & 30        & 7       & 79        \\
  5      & 42             & 118       & 15      & 2,157     \\
  6      & 132            & 602       & 29      & 4,326,459 \\
  7      & 429            & 2,858     & 105     & ?         \\
  \vdots & \vdots         & \vdots    & \vdots  & \vdots  \\\hline
  $n$    & \ncard         & ?         & ?       & ? \\\hline
  & \oeis{A002720} & \sprinc   & \smin   & \sall \\
\end{tabular}
% seems that the number depends on odd/even with odd begin the same as
% the size of the Brauer monoids, and even being something else
\endgroup
\begingroup
\newcommand{\sprinc}{\searchfive{0}{2}{16}{142}{1636}}
\newcommand{\smin}{\searchfive{0}{1}{6}{18}{120}}
\newcommand{\sall}{\searchfour{1}{3}{48}{103406}}
\begin{tabular}[t]{l|r|r|r|r}
  $n$    & $|B_n|$        & principal & minimal & all  \\\hline
  1      & 1              & 0         & 0       & 1         \\
  2      & 3              & 2         & 1       & 3         \\
  3      & 15             & 16        & 6       & 48        \\
  4      & 105            & 142       & 18      & 103,406   \\
  5      & 945            & 1,636     & 120     & ?         \\
  \vdots& \vdots         & \vdots    & \vdots  & \vdots  \\\hline
  $n$   & $(2n-1)!!$     & ?         & ?       & ? \\\hline
  & \oeis{A001147} & \sprinc   & \smin   & \sall \\
\end{tabular}
\endgroup
\caption{The numbers of principal, minimal, and all left and right
  congruences of the Jones (or Temperley-Lieb) monoids $J_n$ and the Brauer
monoids $B_n$ of degree $n$.}\label{table-jones-brauer}
\end{table}

\begin{table}
\centering
\begingroup
\newcommand{\ncard}{$(2n)!/(n!(n+1)!)\sum_{k=0}^n \binom{2n}{2k}$}
\newcommand{\sprinc}{\searchfour{1}{18}{188}{2332}}
\newcommand{\smin}{\searchfour{1}{6}{18}{66}}
\newcommand{\sall}{\searchthree{2}{37}{15367}}
\begin{tabular}[t]{l|r|r|r|r}
  $n$   & $|M_n|$        & principal & minimal & all   \\\hline\hline
  1     & 2              & 1         & 1       & 2     \\
  2     & 9              & 18        & 6       & 37    \\
  3     & 51             & 188       & 18      & 15,367\\
  4     & 323            & 2,332     & 66      & ?     \\
  \vdots& \vdots         & \vdots    & \vdots  & \vdots\\\hline
  $n$   & \ncard         & ?         & ?       & ?     \\\hline
  & \oeis{A026945} & \sprinc   & \smin   & \sall \\
\end{tabular}
\endgroup
\begingroup
\newcommand{\sprinc}{\searchfour{0}{2}{26}{627}}
\newcommand{\smin}{\searchfour{0}{1}{9}{49}}
\newcommand{\sall}{\searchthree{1}{3}{108}}
\begin{tabular}[t]{l|r|r|r|r}
  $n$   & $|I_n ^ *|$    & principal & minimal       & all   \\\hline
  1     & 1              & 0         & 0             & 1     \\
  2     & 3              & 2         & 1             & 3     \\
  3     & 25             & 26        & 9             & 108   \\
  4     & 339            & 627       & 49            & ?     \\
  \vdots& \vdots         & \vdots    & \vdots        & \vdots\\\hline
  $n$   & ?              & ?         & $(2^n -1)^2$? & ?     \\\hline
  & \oeis{A026945} & \sprinc   & \smin         & ?\\
\end{tabular}
\endgroup
\caption{The numbers of principal, minimal, and all left and right
  congruences of the Motzkin monoids $M_n$ and the dual symmetric inverse
monoids $I_n^*$ of degree $n$.}
\label{table-motzkin-dual-sym-inv}
\end{table}

\begin{landscape}
\begin{table}
  \centering
  \begingroup
  \setlength{\tabcolsep}{2pt}
  \begin{tabular}[t]{l|r|r|r|r|r|r|r|r|r|r|r}
    & 2              & 3              & 4             & 5
    & 6           & 7            & 8               & 9              &
    10          & 11             & 12 \\\hline\hline
    1  & 1              & 1              & 1             & 1
    & 1           & 1            & 1               & 1              &
    1           & 1              & 1\\
    2  & 7              & 15             & 31            & 63
    & 127         & 255          & 511             & 1023           &
    2047        & 4095           & 8191\\
    3  & 27             & 102            & 351           & 1146
    & 3627        & 11262        & 34511           & 105186         &
    318627      & 962022         & 2898351\\
    4  & 94             & 616            & 3346          & 16360
    & 75034       & 330256       & 1414066         & 5940760        &
    24627274    & 101123296      & 412378786\\
    5  & 275            & 3126           & 26483         & 189420
    & 1220045     & 7335126      & 42061343        & 233221440      &
    1261948865  & 6705793626     & 35151482003\\
    6  & 833            & 16914          & 222425        & 2289148
    & 20263055    & 162391586    & 1214730797      & 8646767880     &
    59332761467 & 396002654398\\
    7  & 2307           & 91072          & 1927625       & 28829987
    & 349842816   & 3703956326   & 35686022019     & 321211486801\\
    8  & 6488           & 491767         & 36892600      & 995133667
    & 17898782992 & 257691084789 & 3222193441904\\
    9  & 18207          & 2583262        & 955309568     & 457189474269\\
    10 & 52960          & 14222001       & 12834025602\\
    11 & 156100         & 87797807\\
    12 & 462271         & 676179934\\
    13 & 1387117        & 7373636081\\
    14 & 4330708        & 120976491573\\
    15 & 14361633\\
    16 & 51895901\\
    17 & 209067418\\
    18 & 955165194\\
    19 & 4777286691\\
    20 & 24434867465\\
    21 & 114830826032\\
    22 & 499062448513
  \end{tabular}
  \endgroup
  \caption{Numbers of 2-sided congruences of the free monoids; see
    also~\cite[Appendix C]{Bailey2016aa} for the corresponding
    results in the free semigroup.
    The columns correspond to the number of
    generators and the rows to the index of the congruences, so that the
    value in row $i$ and column $j$ is the number of 2-sided congruences
  with index at most $i$ of the free monoid with $j$-generators.}
  \label{table-2-sided}
\end{table}
\end{landscape}

\begin{table}
\centering
\begin{tabular}{l||r|r|r|r|r|r|r}
  $n$            & index $\leq 2$       & index $\leq 3$
  & index $\leq 4$ & index $\leq 5$ &
  index $\leq 6$ & index $\leq 7$       & index $\leq 8$ \\\hline\hline
  3              & 29                   & 484
  & 6,896          & 103,204        & 1,773,360      & 35,874,182
  & 849,953,461\\ \hline
  4              & 67                   & 2,794
  & 106,264        & 4,795,980      & 278,253,841    & 20,855,970,290
  & - \\ \hline
  5              & 145                  & 14,851
  & 1,496,113      & 198,996,912    & 37,585,675,984 & -
  & - \\\hline
  6              & 303                  & 77,409
  & 20,526,128     & 7,778,840,717  & -              & -
  & - \\\hline
  7              & 621                  & 408,024
  & 281,600,130    & -              & -              & -
  & - \\\hline
  8              & 1,259                & 2,201,564
  & -              & -              & -              & -
  & -\\\hline
  $a_n$          & $2 a_{n-1} + 2n + 3$ & ?
  & ?              & ?              & ?              & ?              & ?
\end{tabular}
\caption{The number of left and right congruences of the Plactic monoid
with $n$-generators for some small values of $n$.}
\label{table-plactic}
\end{table}
\end{document}